\documentclass[12pt]{amsart}  

\usepackage[latin1]{inputenc}
\usepackage{amsmath} 
\usepackage{amsfonts}
\usepackage{amssymb}
\usepackage{stmaryrd}
\usepackage{latexsym} 
\usepackage{graphicx}
\usepackage{subfigure}
\usepackage{color}
\usepackage[colorlinks = true,
            linkcolor = blue,
            urlcolor  = blue,
            citecolor = blue,
            anchorcolor = blue]{hyperref}
\usepackage{verbatim}
\usepackage[all]{xy}
\usepackage{graphics}
\usepackage{pdfsync}
\usepackage{tikz}
\usepackage{url}
\usepackage{mathtools}
\usepackage{enumerate}
\usepackage{mathrsfs}
\usepackage{amsthm}
\usepackage{setspace}
\usepackage{chngcntr}
\usepackage{youngtab} 
\usepackage{young} 
\usepackage[nocompress]{cite}
\usetikzlibrary{decorations.pathreplacing,angles,quotes}
\usetikzlibrary{decorations.pathreplacing}

\theoremstyle{plain}
\newtheorem{thm}{Theorem}[subsection]
\counterwithin*{thm}{section}

\renewcommand*{\thethm}{
  \ifnum\value{subsection}<1 
    \thesection
  \else
    \thesubsection
  \fi
  .\arabic{thm}
}

\newtheorem{prop}[thm]{Proposition}

\newtheorem{lemma}[thm]{Lemma}
\newtheorem{cor}[thm]{Corollary}
\newtheorem{conj}[thm]{Conjecture}

\theoremstyle{remark}

\numberwithin{equation}{section}
\newcommand{\fS}{{\mathfrak S}}
\newcommand{\ii}{{\mathfrak i}}
\newcommand{\dd}{{\mathfrak d}}
\DeclareMathOperator{\inv}{inv}
\DeclareMathOperator{\coinv}{coinv}
\DeclareMathOperator{\Coinv}{Coinv}
\DeclareMathOperator{\maj}{maj}
\DeclareMathOperator{\comaj}{comaj}
\DeclareMathOperator{\des}{des}
\DeclareMathOperator{\Des}{Des}
\DeclareMathOperator{\Asc}{Asc}
\DeclareMathOperator{\asc}{asc}
\DeclareMathOperator{\Inv}{Inv}

\DeclarePairedDelimiter{\ceil}{\lceil}{\rceil}
\DeclarePairedDelimiter{\floor}{\lfloor}{\rfloor}
\definecolor{vividviolet}{rgb}{0.62, 0.0, 1.0}

\newcommand{\cI}{{\mathcal I}}
\newcommand{\ol}{\overline}

\newcommand{\hqed}{\hfill \qed}

\newlength\cellsize 
\savebox2{%
\begin{picture}(15,15)
\put(0,0){\line(1,0){15}}
\put(0,0){\line(0,1){15}}
\put(15,0){\line(0,1){15}}
\put(0,15){\line(1,0){15}}
\end{picture}}
\newcommand\cellify[1]{\def\thearg{#1}\def\nothing{}%
\ifx\thearg\nothing
\vrule width0pt height\cellsize depth0pt\else
\hbox to 0pt{\usebox2\hss}\fi%
\vbox to 15\unitlength{
\vss
\hbox to 15\unitlength{\hss$#1$\hss}
\vss}}
\newcommand\tableau[1]{\vtop{\let\\=\cr
\setlength\baselineskip{-16000pt}
\setlength\lineskiplimit{16000pt}
\setlength\lineskip{0pt}
\halign{&\cellify{##}\cr#1\crcr}}}
\savebox3{%
\begin{picture}(15,15)
\put(0,0){\line(1,0){15}}
\put(0,0){\line(0,1){15}}
\put(15,0){\line(0,1){15}}
\put(0,15){\line(1,0){15}}
\end{picture}}
\newcommand\expath[1]{%
\hbox to 0pt{\usebox3\hss}%
\vbox to 15\unitlength{
\vss
\hbox to 15\unitlength{\hss$#1$\hss}
\vss}}
\newcommand\bas[1]{\omit \vbox to \cellsize{ \vss \hbox to \cellsize{\hss$#1$\hss} \vss}}

\makeatletter
\def\l@subsection{\@tocline{2}{0pt}{2.5pc}{5pc}{}}
\makeatother
\begin{document}

\title[Permutation Statistics and Pattern Avoidance in Involutions]{Permutation Statistics and Pattern Avoidance in Involutions}

\author{Samantha Dahlberg}
\address{
Department of Mathematics,
 University of British Columbia,
 Vancouver BC V6T 1Z2, Canada}
\email{samadahl@math.ubc.ca}

\thanks{
The author was supported  in part by the National Sciences and Engineering Research Council of Canada.}
\subjclass[2010]{Primary 05E05; Secondary 05A05, 05A15, 05A19}
\keywords{pattern avoidance, involutions, major index, number of inversions, descents, ascents, permutations, standard Young Tableaux, standard q-analogues}

\begin{abstract}
Dokos et.\ al.\  studied the distribution of two statistics over permutations $\mathfrak{S}_n$ of $\{1,2,\dots, n\}$ that avoid one or more length three patterns. A permutation $\sigma\in\mathfrak{S}_n$ contains a pattern $\pi\in\mathfrak{S}_k$ if $\sigma$ has a subsequence of length $k$ whose letters are in the same relative order as $\pi$.
This paper is a  comprehensive study of the same two statistics, number of inversions and major index, over involutions $\mathcal{I}_n=\{\sigma\in\mathfrak{S}_n:\sigma^2=\text{id}\}$ that avoid one or more length three patterns. The equalities between the generating functions are {consequently} determined via symmetries and we conjecture this happens for longer patterns as well.
We describe  the generating functions for each set of patterns  including the fixed-point-free case, $\sigma(i)\neq i$ for all $i$. Notating $M\mathcal{I}_n(\pi)$ as the generating function for the major index over the avoidance class of involutions associated to $\pi$
 we particularly present an independent determination that  $M\mathcal{I}_n(321)$ is the $q$-analogue for the central binomial coefficient that first appeared in a paper by Barnebei, Bonetti, Elizalde and Silimbani. A shorter proof is presented that establishes a connection to core, a central topic in poset theory. We also prove that  $M\mathcal{I}_n(132;q)=q^{\binom{n}{2}}M\mathcal{I}_n(213;q^{-1})$ and that the same symmetry holds for the larger class of permutations conjecturing that the same equality is true for involutions and permutations  given any pair of patterns of the form $k(k-1)\dots 1(k+1)(k+2)\dots m$ and $12\dots (k-1) m(m-1)\dots k$, $k\leq m$. 
\end{abstract}

\maketitle
\tableofcontents

\section{Introduction}
The topic of pattern avoidance has received a lot of attention since Knuth's work in~\cite{knu:acp3}. 
To define pattern avoidance we start with saying that two words of integers $a_1a_2 \ldots a_k$ and $b_1 b_2 \ldots  b_k$ are {\it order isomorphic} when $a_i \leq a_j$ if and only if $b_i \leq b_j$  and $a_i \geq a_j$ if and only if $b_i \geq b_j$ for all $i$ and $j$. A permutation $\sigma \in \fS_n$ of the set $[n]=\{1,2,\dots, n\}$ is said to {\it contain the pattern } $\pi \in \fS_k$ if there exists an increasing sequence of indices $m_1, m_2, \ldots , m_k$ such that $\sigma(m_1)\sigma(m_2)\ldots \sigma(m_k)$ is order isomorphic to $\pi$. We say that $\sigma$ $\mathit{avoids}$ the pattern $\pi $ if $\sigma$ does not contain the pattern $\pi$. We notate  these pattern avoidance classes  by
$$\fS_n(\pi) = \{ \sigma \in \fS_n : \sigma \text{ avoids } \pi \}.$$ 
Two patterns $\pi_1$ and $\pi_2$ are {\it Wilf-equivalent} if $|\fS_n(\pi_1)| =|\fS_n(\pi_2)| $ for all $n\geq 0$. Knuth~\cite{knu:acp3} found that there is only one Wilf-equivalence class, $|\fS_n(\pi)|=\frac{1}{n+1}\binom{2n}{n}$ the $n$th Catalan number, for patterns $\pi\in\fS_3$. 
For length four patterns there are three Wilf-equivalence classes. Though the original proof of these classes includes the work  of~\cite{S96} and~\cite{W95}  it is sufficient to use symmetries, Stankova's equality in~\cite{S94} $|\fS_n(4132)|=|\fS_n(3142)|$ (Theorem 3.1) and Backelin, West,  and Xin's  result~\cite{BWX07} (Theorem 2.1) that $12\dots i \pi(i+1)\dots\pi(k)$ and $i\dots 21 \pi(i+1)\dots\pi(k)$ are Wilf-equivalent. The class represented by 1234 was enumerated by Gessel~\cite{G90} (page 281) using symmetric functions and the class represented by 1342 was enumerated by B\'{o}na~\cite{B97} (Theorem 3), however, the last class represented by 1324 has not yet been enumerated. 

Simion and Schmidt in~\cite{ss:rp} enumerated classes avoiding  multiple length three patterns. They also considered pattern avoidance  for the subclasses of involutions $\cI_n=\{\sigma\in \fS_n:\sigma^2=\text{id}\}$, even permutations and odd permutations. 
If we denote $${\mathcal I}_n(\pi) = \{ \iota \in {\mathcal I}_n : \iota \text{ avoids } \pi \}$$ as the pattern avoidance class for involutions avoiding $\pi$ we say that $\pi_1$ and $\pi_2$ are {\it ${\mathcal I}$-Wilf equivalent} if $|{\mathcal I}_n(\pi_1)| =|{\mathcal I}_n(\pi_2)| $.  
Simion and Schmidt found that there are two ${\mathcal I}$-Wilf equivalent classes for patterns in  $\fS_3$. 
\begin{thm}[Simion and Schmidt~\cite{ss:rp} Propositions 3, 5 and 6] There are two $\cI$-Wilf equivalence classes for patterns in $\fS_3$. 
$$
   |{\mathcal I}_n(\pi)| = \left\{
   \setstretch{1.5}
     \begin{array}{ll}
        \binom{n}{\ceil{ n/2}} &\pi \in \{123, 132, 321, 213\}, \\
       2^{n-1} &  \pi \in\{ 231, 312\}.
     \end{array}
   \right.
$$
\label{thm:SimionSchmidt}
\vspace{-.4in}

\hqed
\vspace{.5cm}
\end{thm}
For length four patterns there are eight different $\cI$-Wilf equivalence classes. The original classification of these classes includes the work of~\cite{GPP01} and~\cite{J02}, however, we can piece together the eight classes from the following. From Guibert's work~\cite{G95} we have $|\cI_n(3412)|=|\cI_n(4321)|$. 
Bousquet-M\'{e}lou and Steingr\'{i}msson in~\cite{BS05} (Theorem 1) proved that the map Backelin, West,  and Xin defined in~\cite{BWX07}  commutes with inverses so showed that $12\dots i \pi(i+1)\dots\pi(k)$ and $i\dots 21 \pi(i+1)\dots\pi(k)$ are $\cI$-Wilf equivalent. Bloom and Saracino in~\cite{BS12} present a shortened  proof using growth diagrams. Using some implications of the Robinson-Schensted-Knuth map and symmetries, the classification can be completed. 
Of these eight classes only a few have been enumerated. Regev~\cite{R81} (Section 4.5) enumerated the class represented by 1234, which is the Motzkin numbers. This class has received special attention as equalities between the pattern avoidance sets were established using bijections to 1-2 trees with $n$ edges~\cite{G95,GPP01,J02}. The class represented by 2413 was enumerated by Brignall, Huczynska and  Vatter~\cite{BHV08} (Example 6.6) where $|\cI_n(2413)|$ equals the number separable involutions. Two other classes represented by 2341 and 1342 were enumerated by B\'{o}na et.\ al.~\cite{BHPV16} (Section 5 and 6) who also did work on the asymptotic growth of all classes. 
The cardinalities for multiple pattern avoidance in involutions has been classified and enumerated by Guibert and  Mansour in~\cite{GM02a} (Examples 2.6, 2.8, 2.12, 2.18, 2.20)
who consider sets of patterns containing 132, Egge and Mansour~\cite{EM04} who enumerate the sets containing the pattern 231, and Wulcan who enumerates all pairs of length three patterns in~\cite{W02}. 

Though there is much more work on pattern avoidance in permutations and involutions for longer patterns, multiple patterns and many other kinds of restrictions or generalizations, we instead turn our focus to the more refined equivalence class defined by the distribution of statistics. 
Sagan and Savage~\cite{SS12} asked about the distribution of statistics over pattern avoidance classes of permutations, which Dokos et.\ al.\ answer in~\cite{DDJSS12}. 
In~\cite{DDJSS12} they consider the distribution of two permutation statistics, number of inversions and  major index, over the permutation  avoidance classes for any set of length three patterns. 
An {\it inversion} is a pair of indices $i<j$ such that $\pi(i)>\pi(j)$. The {\it set of inversions}  is 
$$\Inv (\sigma) = \{ (i,j):i<j \text{ and }  \sigma (i)>\sigma(j) \}$$ 
and the {\it inversion number} is  $\inv(\sigma) = |\Inv(\sigma)|$. The {\it decent set} of an integer word $w=w_1w_2\dots w_n$  is $\Des(w) = \{i\in [n-1]:w_i>w_{i+1}\}$ from which we define $\des(w)=|\Des(w)|$ and the {\it major index}, 
$$\maj(w)=\sum_{i\in \Des(w)} i.$$ 
One reason these statistics hold interest is because of a result by Major Percy MacMahon who found that the generating function of $\fS_n$ for $\maj$ or $\inv$ is 
$$\sum_{\sigma\in\fS_n}q^{\maj(\sigma)}=\sum_{\sigma\in\fS_n}q^{\inv(\sigma)}=[n]_q!=[n]_q[n-1]_q\cdots[1]_q$$
the standard $q$-analogue for $n!$ where $[n]_q=1+q+\cdots + q^{n-1}$ is the standard $q$-analogue for $n$. This result can  be found in~\cite{S97} (Corollary 1.3.13 and Proposition 1.4.6).
This function has many beautiful properties including symmetry and log-concavity~\cite{S97} (Exercise 1.50(e)). A polynomial $a_0+a_1q+\cdots +a_kq^k$ is said to be {\it symmetric} if $a_{i}=a_{k-i}$ for all $i$ and is {\it log-concave } if $a_i^2 \geq a_{i-1}a_{i+1}$ for all $i$. Log-concavity is particularly interesting because it implies that the polynomial is {\it unimodal}~\cite{S97} (Exercise 1.50(a)), that $a_0\leq a_1\leq \dots \leq a_j \geq a_{j+1}\geq \dots\geq a_k$ for some $j$. 
The associated generating functions for the restricted class of involutions have been studied by Dukes~\cite{D07}. He found that the generating function for $\maj$ over $\cI_n$  is symmetric (Corollary 2.4), he conjectured it to be additionally log-concave and proved some partial results about its unimodality. 
The generating function for involutions and $\inv$ was studied by D\'{e}sarm\'{e}nien~\cite{D82} who related the function to $q$-Hermite polynomials, however, this function is not unimodal nor log-concave. 

Dokos et.\ al.\ in~\cite{DDJSS12} have a full study for the permutation generating functions for any avoidance class that is a subset of $\fS_3$  for $\inv$ and $\maj$. Specifically, the generating functions they studied were 
$$I_n(\pi) = I_n(\pi;q) = \sum_{\sigma \in \fS_n} q^{\inv (\sigma)}$$
and
$$M_n(\pi) = M_n(\pi ;q) = \sum_{\sigma \in \fS_n(\pi)} q^{\maj (\sigma)}.$$  
They defined that $\pi_1$ and $\pi_2$ are {\it $I$-Wilf equivalent} if $I_n(\pi_1) =I_n(\pi_2)$ and {\it $M$-Wilf equivalent} if $M_n(\pi_1) =M_n(\pi_2)$. Let $[\pi]_I$ and $[\pi]_M$ denote the associated equivalence classes. They determined these classes for length three patterns and described the generating functions. 

\begin{thm}[Dokos et.\ al.~\cite{DDJSS12} Theorem 2.3 and 2.6] The $I$-Wilf equivalence and $M$-Wilf equivalence classes for single length three patterns are as follows. 
\begin{enumerate}[(i)]
\item The non-singular $I$-Wilf  classes  are $[132]_I=\{132,213\}$ and $[231]_I=\{231,312\}$.
\item The non-singular $M$-Wilf  classes  are $[132]_M=\{132,231\}$ and $[213]_M=\{213,312\}$.\hqed
\end{enumerate}
\label{thm:Dokos}
\end{thm}
Since all these pattern classes are counted by the Catalan numbers the generating functions  described are all $q$-analogues for $C_n$. Further work on these generating functions can be found in~\cite{B14,C15,CEKS13,GM09,T15,YGZ15}.

Our work in this paper parallels the work of Dokos et.\ al.\ since we aim to describe the generating functions
$$I{\mathcal I}_n (\pi)= I{\mathcal I}_n(\pi;q) = \sum_{\iota \in {\mathcal I}_n(\pi)} q^{\inv (\iota)}$$
and
$$M{\mathcal I}_n(\pi) = M{\mathcal I}_n(\pi ;q) = \sum_{\iota \in {\mathcal I}_n(\pi)} q^{\maj (\iota)}$$  
for single and later multiple patterns of length three as well as determine which patterns give equal generating functions. 

This is the first full study of these functions, though, some of these functions have been  well-studied individually by others. 
The generating function for $\cI_n(132)$ has been studied before by Guibert and Mansour in~\cite{GM02} (Theorem 4.2) who studied the generating function for $\des$ and  the number of occurrences of the pattern $12\dots k$, which counts $\binom{n}{2}$ minus $\inv$ when $k = 2$. 
The function $M\cI_n(321)$ was studied by Barnabei et.\ al.~\cite{BBES14} (Theorem 3.3) who found that this function is the standard $q$-analogue for the central binomial coefficient where the standard $q$-analogue for a general  binomial coefficient is
\begin{equation}{n\brack{k}}_q=\frac{[n]_q!}{[n-k]_q![k]_q!}.
\label{eq:qbinom}
\end{equation}
In their proof they establish a connection to hook decompositions. We independently determined this result with a shorter proof that establishes a connection to core, a topic that  is usually used to prove symmetric chain decomposition in poset theory. Additionally, our proof is easily  modified to prove another interpretation of the standard $q$-analogue for any  binomial coefficient, which is a result that also appears in~\cite{BBES16} (Corollary 14) by  Barnabei et.\ al. Some ideas  of the bijection we present later can be seen in~\cite{EFPT15, BBES16} in their discussions of associating involutions avoiding 321 to Dyke paths, though, our phrasing of it in terms of core is new.
In~\cite{E04} Egge considers the length four pattern 3412 and studies $I\cI_n(3412)$. There seems to be no more work done on these generating functions for longer patterns.

One goal of this paper is to determine which length three patterns give equivalent generating functions. 
We will say that $\pi_1$ and $\pi_2$ are {\it $I{\mathcal I}$-Wilf equivalent} if $I{\mathcal I}_n(\pi_1) =I{\mathcal I}_n(\pi_2)$,  $M{\mathcal I}$-Wilf equivalent if $M{\mathcal I}_n(\pi_1) =M{\mathcal I}_n(\pi_2)$ and write $[\pi]_{I\cI}$ and $[\pi]_{M\cI}$ for the associated equivalence classes. However, these equivalence classes can be established quickly  so this paper moreover considers the description of the generating functions. Some of these generating functions already have explicit descriptions. Since $\fS_n(231,312)=\cI_n(231)$, which can be concluded from the work in~\cite{ss:rp} (Propositions 6 and 8),  all generating functions regarding the pattern $231$ have been determined by Dokos et.\ al.~\cite{DDJSS12}. We  include a description for all patterns and generating function for the completeness of this study.

This paper is organized  as follows. In the next section we introduce some background information about using symmetries of the square and writing permutations as inflations. Section~\ref{inv} focuses on describing $I\cI_n(\pi)$ for length three patterns and considers the fixed-point-free case, $\iota(i)\neq i$ for all $i$. We find a connection to a $q$-Catalan analogue defined by Carlitz and Riordan in~\cite{CR64}  when determining $I\cI_n(132)$ in Proposition~\ref{II132}. 
This section finishes with the result that any $\iota\in\cI_{2k+1}(123)$ has $\inv(\iota)$  even if and only if $k$ is even in Corollary~\ref{cor:123inv}. In Section~\ref{maj}  we present our results about $M\cI_n(\pi)$ including our result about $M\cI_n(321)$ and $q$-analogues for  binomial coefficients in Theorem~\ref{theorem:321qanalougeequiv}. For each pattern we also consider the fixed-point-free case. In this section we also re-present the symmetry $M\cI_n(\pi_1;q)=q^{\binom{n}{2}}M\cI_n(\pi_2;q^{-1})$ between the patterns $\pi_1 =123$ and $\pi_2=321$ in Proposition~\ref{thm:123&321 symmetry}, which has been shown before in~\cite{BBES14,DRS07,ss:rp}. We particularly present this symmetry because we also prove this symmetry  for the pair of patterns 132 and 213 in Theorem~\ref{thm:132symm213}. We then summarize  in Section~\ref{multi} the generating functions for multiple pattern avoidance in involutions. We finish the paper with Section~\ref{permsymm} that includes a result about the symmetries for the larger class of permutations, $M_n(\pi_1;q)=q^{\binom{n}{2}}M_n(\pi_2;q^{-1})$,  between patterns 123 and 321 as well as 132 and 213. Both of these results are natural and the result about the pair of patterns 123 and 321 is classical as it is an elegant generalization of the involution case. What is innovative is the map we define between $\fS_n(132)$ and $\fS_n(213)$ in equations~\eqref{eq:theta1} and~\eqref{eq:theta2} that restricts to involutions. We conclude with Conjectures~\ref{conj} and~\ref{conj:invo} that this symmetry always happens for involutions and permutations for the pair of patterns $k(k-1)\dots 1(k+1)(k+2)\dots m$ and $12\dots (k-1) m(m-1)\dots k$ for any $1\leq k \leq m$.

\section{Diagrams and inflations of permutations}
The proofs behind the $M\cI$-Wilf and $I\cI$-Wilf equivalence classes are quick and can be shown using a  geometrical approach to permutations. The {\it diagram} of a permutation  $\sigma \in \fS_n$   is the collection of points $(i,\sigma(i))$ in the coordinate plane inside the square with corners at $(1,1)$ and $(n,n)$. Figure \ref{fig:boxexample} illustrates the involution $\iota = 216543$.

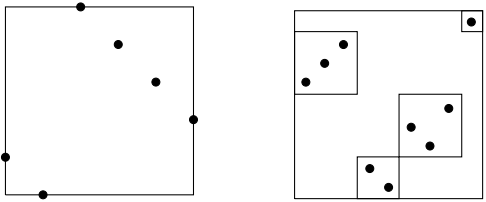
\begin{figure}[h]
\centering
\begin{tikzpicture} [scale=.5]
\filldraw [black] 
(1,2) circle (3pt)
(2,1) circle (3pt) 
(3,6) circle (3pt) 
(4,5) circle (3pt)
(5,4) circle (3pt)
(6,3) circle (3pt);
\draw (1,1) --  (6,1) -- (6,6) -- (1,6) -- (1,1);
 \end{tikzpicture}
 \hspace{1cm}
 \begin{tikzpicture} [scale=.5]
\filldraw [black] 
(2.3,5.1) circle (3pt)
(1.8,4.6) circle (3pt) 
(1.3,4.1) circle (3pt) 
(3,1.8) circle (3pt)
(3.5,1.3) circle (3pt)
(4.1,2.9) circle (3pt)
(4.6,2.4) circle (3pt)
(5.1,3.4) circle (3pt)
(5.7,5.7) circle (3pt);
\draw (1,34/9) rectangle (24/9,49/9);
\draw (34/9,19/9) rectangle (49/9,34/9);
\draw (24/9,1) rectangle (34/9,19/9);
\draw (49/9,49/9) rectangle (6,6);
\draw (1,1) --  (6,1) -- (6,6) -- (1,6) -- (1,1);
 \end{tikzpicture}
 \caption{From left to right we have the diagrams of $216543$ and $3124[\ii_3,\dd_2,213,1] = 678214359$.}
  \label{fig:boxexample}
\end{figure}

On this square we can do a number of operations that preserve the square. We can reflect the square across a line through the center of the square. If the square is preserved then the reflected diagram represents a permutation and we will notate this new permutation  $r_m(\sigma)$ where $m$ is the slope of the line. We can also perform rotations about the center of the square and those rotations that preserve the square will also produce a diagram associated to a permutation  $R_{\theta}(\sigma)$ where we rotate counterclockwise by $\theta$. The reflections  $r_1$, $r_0$, $r_{-1}$, and $r_{\infty}$ and rotations $R_0$, $R_{90}$, $R_{180}$ and $R_{270}$ all preserve the square and give us bijections $\fS_n\rightarrow\fS_n$. We will say equivalence classes between patterns are proven {\it trivially} if they can be proven purely from these maps.

Since we are particularly interested in involutions we are only going to be interested in the operations that map an involution to another involution. A two-cycle $(i,j)$ in a permutation implies we have the points $(i,j)$ and $(j,i)$, which are  symmetric around the line with slope $m=1$. If a permutation has only two-cycles and one-cycles then the diagram must be symmetric around the line with slope $m=1$ or its main diagonal. This means that any operation that maps a permutation with this symmetry to another with this symmetry will be a bijection $\cI_n\rightarrow\cI_n$.

\begin{lemma} We have the following properties for the operations on the square. 
\begin{enumerate}[(i)]
\item The operations  $r_1$, $r_{-1}$, $R_0$ and $R_{180}$ are bijective maps $\cI_n\rightarrow \cI_n$. 
\item The map $r_1$ and $R_0$ are both  the identity map on involutions. 
\item The maps $r_{-1}$ and $R_{180}$ are the same map $\cI_n\rightarrow \cI_n$. \hqed
\end{enumerate}
\label{lemma:operations and involutions}
\end{lemma}

We say that a map {\it preserves a statistic} if  the statistic remains unchanged under the map. For example if we have a map $\phi :\fS_n \rightarrow \fS_n$ that preserves $\inv$ then we mean that $\inv(\sigma)=\inv(\phi(\sigma))$ for all $\sigma$ in the domain $\fS_n$. 
Dokos et.\ al.~\cite{DDJSS12}  detailed which operations preserve $\inv$ and $\maj$. We find that all the maps that map involutions to involutions preserve $\inv$ and no operation except the identity preserves $\maj$.

\begin{lemma}[Dokos et.\ al.~\cite{DDJSS12} Lemma 2.1] For $\iota \in \cI_n$ the operations $r_1$, $r_{-1}$, $R_0$ and $R_{180}$ preserve $\inv$. \hqed
 \label{lemma:inv preserved}
\end{lemma}

Another tool that we will use often in this paper is describing a permutation as an inflation of another permutation. A {\it block} in a permutation is a  subsequence on some indices $[i,j]=\{i,i+1,\dots, j\}$ whose values $\sigma(i),\dots,\sigma(j)$ in union form  an  interval $[a,b]$ for some $b-a=j-i$.   Given  a  permutation $\tau\in \fS_k$ and a collection of  permutations $\sigma_1,\sigma_2,\dots,\sigma_k$ the {\it inflation} of $\tau$ by  the  collection $\sigma_1,\sigma_2,\dots,\sigma_k$ is  the permutation we get from $\tau$ by replacing the point $(i,\tau(i))$  with a block order-isomorphic to $\sigma_i$.  Note that this definition also works if we have an empty block $\sigma_j=\epsilon\in \fS_0$ for some $j$. Often times we will have blocks  order-isomorphic to a strictly increasing or decreasing sequence, so for convenience we define $\ii_k=12\dots k$ and $\dd_k=k\dots 21$.  For example $3124[\ii_3,\dd_2,213,1] = 678214359$, which is displayed in Figure~\ref{fig:boxexample}.

One can describe many pattern avoidance classes using inflations. The following proposition contains several well-known descriptions of permutations that avoid a certain pattern as  inflations. 
\begin{prop}We have the following descriptions of pattern avoiding permutations. 
\begin{enumerate}[(i)]
\item If $\sigma$ avoids 132 then $\sigma=231[\sigma_1,1,\sigma_2]$ for some $\sigma_1,\sigma_2$ that avoid 132. 
\item If $\sigma$ avoids 213 then $\sigma=312[\sigma_1,1,\sigma_2]$ for some $\sigma_1,\sigma_2$ that avoid 213. 
\item If $\sigma$ avoids 231 then $\sigma=132[\sigma_1,1,\sigma_2]$ for some $\sigma_1,\sigma_2$ that avoid 231. 
\item If $\sigma$ avoids 312 then $\sigma=213[\sigma_1,1,\sigma_2]$ for some $\sigma_1,\sigma_2$ that avoid 312. \hqed
\end{enumerate}
\label{prop:decomp}
\end{prop}

We illustrate this for the pattern 132 in Figure~\ref{fig:132}. Typically knowing how to write $\sigma$ as an inflation makes calculating $\maj$ or $\inv$ easier.  For a set $A=\{a_1,a_2,\dots, a_i\}$  we write $A+j=\{a_1+j,a_2+j,\dots,a_i+j\}$ to be the set of all the elements in $A$ increased by $j$ and let $|\sigma|$ be the length of $\sigma$. We can calculate the descent set, $\Des(\sigma)$, by considering the descents in the blocks and the descents between the blocks of the inflation. 
For example say $\sigma$ avoids 132 and is written as $\sigma=231[\sigma_1,1,\sigma_2]$ for some $\sigma_1,\sigma_2$ that avoid 132 with $|\sigma_2|\neq 0$. We then have $\Des(\sigma)=\Des(\sigma_1)\cup\{|\sigma_1|+1\}\cup (\Des(\sigma_2)+|\sigma_1|+1)$. Further we can calculate $\maj$ by adding up the descents between the blocks and the descents in  each block in the inflation by noting that $\maj$ contributed by the descents in $(\Des(\sigma_2)+|\sigma_1|+1)$ is  $\maj(\sigma_2)+(|\sigma_1|+1)\des(\sigma_2)$ so $\maj(\sigma)=\maj(\sigma_1)+|\sigma_1|+1+\maj(\sigma_2)+(|\sigma_1|+1)\des(\sigma_2)$.


\section{Number of inversions and length three patterns}
\label{inv}
We find that the $I\cI$-Wilf equivalence classes for length three patterns are trivially determined.  As result, most of this section will be spent discussing the decomposition of involutions that avoid a single pattern of length three with the goal of describing the generating functions for $\inv$. Some of these generating functions have been studied by others including Guibert and Mansour~\cite{GM02} (Theorem 4.2) who studied involutions avoiding 132.  Their generating function  counts the number of occurrences of the pattern $\ii_k$, which counts $\binom{n}{2}$ minus the number of inversions when $k=2$.  Dokos et.\ al.~\cite{DDJSS12} studied permutations avoiding length three patterns and their generating functions and since $\fS_n(231,312)=\cI_n(231)$, by Simion and Schmidt~\cite{ss:rp},  their work determines the generating function for involutions avoiding the pattern $231$. The goal of this section is to give a complete description for all the generating functions of all length three patterns. 
In this section we show connections to some $q$-analogues of the Catalan numbers and standard Young Tableau. We prove   a formula that quickly computes $\inv$ for $\iota\in\cI_n(321)$ using the two-cycles in Lemma~\ref{lemma:321inv} and  for $\iota\in\cI_{2k+1}(123)$ we discover that $\inv(\iota)$ is even if and only if $k$ is even, which is stated in Corollary~\ref{cor:123inv}.

We describe some generating functions, not directly, but in steps by first considering the subset of fixed-point-free involutions. A permutation $\sigma$ has a {\it fixed point} if there exists a $j$ such that $\sigma(j)=j$ and we call $\sigma$ {\it fixed-point-free} if $\sigma$  does not have any fixed points. To notate the subsets we will write $F\cI_n=\{\iota\in \cI_n:\iota(i)\neq i, \forall i\}$, $F\cI_n(\pi) = \cI_n(\pi)\cap F\cI_n$ and let the associated generating function be
$$IF\cI_n(\pi)=IF\cI_n(\pi;q)=\sum_{\iota\in F\cI_n(\pi)} q^{\inv \pi}.$$
Since there are no fixed-point-free involutions of odd length we will let $IF\cI_n(\pi)=0$ when $n$ is odd.
We also find in some cases it is easier to determine the generating function for the number of coinversions rather than the number of inversions. A {\it coinversion} of $\sigma \in \fS_n$ is a pair of indices $(i,j)$ such that $i<j$ and $\sigma(i)<\sigma(j)$. Let $\Coinv(\sigma)$ be the {\it set of coinversions} and the {\it number of coinversions} be $\coinv(\sigma)=|\Coinv(\sigma)|$. The associated generating functions are
$$\ol{  I\cI}_n(\pi)=\ol{I\cI}_n(\pi;q)=\sum_{\iota \in \cI_n(\pi)} q^{\coinv(\iota)}$$
and
$$\ol{  IF\cI}_n(\pi)=\ol{IF\cI}_n(\pi;q)=\sum_{\iota \in F\cI_n(\pi)} q^{\coinv(\iota)}.$$
The two statistics $\inv$ and $\coinv$ are closely related, and their generating functions determine each other. 

\begin{lemma}We have the following equalities involving inversions and coinversions.
\begin{enumerate}[(i)]
\item For $\sigma \in \fS_n$ we have $\inv(\sigma)=\binom{n}{2}-\coinv(\sigma)$. 
\item $\displaystyle q^{\binom{n}{2}}I\cI_n(\pi;q^{-1})=\ol{  I\cI}_n(\pi;q).$
\item $\displaystyle q^{\binom{n}{2}}IF\cI_n(\pi;q^{-1})=\ol{  IF\cI}_n(\pi;q).$
\end{enumerate}
\label{lemma:invproperties}
\end{lemma}

\begin{proof}
For a length $n$ permutation the total number of pairs of indices $(i,j)$ such that $i<j$ is $\binom{n}{2}$. Since all such pairs are either an inversion or a coinversion we have the equality in (i). The equations  in (ii) and (iii) follow. 
\end{proof}

For some patterns it will be simpler to describe properties using ascent sets rather than descent sets. The {\it ascent set} of $\sigma \in \fS_n$ is $\Asc(\sigma)=\{i:\sigma(i)<\sigma(i+1)\}$ with  the {\it number of ascents} equal to $\asc(\sigma)=|\Asc(\sigma)|$.

\subsection{The $I\cI$-Wilf equivalence classes for length three patterns}

The $I{\mathcal I}$-Wilf equivalences classes are all determined trivially. 

\begin{prop} There are only two non-singular  $I{\mathcal I}$-Wilf equivalences classes for length three patterns that are $\{132,213\}$ and $\{231,312\}$.
 \label{prop:IIWilfequiv}
\end{prop}
\begin{proof} An involution $\iota$ avoids 132 if and only if $r_{-1}(\iota)$ avoids 213 since $r_{-1}(132)=213$. 
By Lemma~\ref{lemma:operations and involutions} the operation $r_{-1}$ is a bijection from $\cI_n$ to itself so restricts to a bijection between ${\mathcal I}_n(132)$ and ${\mathcal I}_n(213)$. Finally since this operation also preserves $\inv$ by Lemma~\ref{lemma:inv preserved} we must have  that $213$ and $132$  are $I\cI$-Wilf equivalent.

By a similar argument using the map $r_1$ we can show that $231$ and $312$ are $I\cI$-Wilf equivalent. Lastly, we can see that we have four distinct classes just by looking  at the case of $n=3$. 
\end{proof}

We conjecture that all $I\cI$-Wilf equivalence classes are trivially determined. 

\begin{conj}
The $I\cI$-Will equivalence class for any pattern $\pi$ is $[\pi]_{I\cI}=\{\pi, r_1(\pi), r_{-1}(\pi), R_{180}(\pi)\}$.
\end{conj}

The above conjecture is confirmed for patterns up to length 5. However, for permutations the $I$-Wilf equivalence classes are not always determined trivially.  Chan in~\cite{C15} (Proposition 5) proved if $\pi_1$ and $\pi_2$ are shape and $I$-Wilf equivalent, a stronger condition than $I$-Wilf equivalence, then so are $12[\pi_1,\gamma]$ and $12[\pi_2,\gamma]$ for any permutation $\gamma$. This particularly applies to the pair $12[231,231]$ and $12[312,231]$ (Chan~\cite{C15} Corollary 6), which are not in the same symmetry class on the square. We note that this particular pair is not $I\cI$-Wilf equivalent because the generating functions are not equal when $n=8$.

\subsection{The patterns $231$ and $312$}

It turns out that any  permutation that avoids both $231$ and $312$ is actually an involution and these involutions are precisely those that avoid $231$, which was first determine by  Simion and Schmidt ~\cite{ss:rp}. They also determined the decomposition of involutions in $\cI_n(213)$ so the generating function $I\cI_n(231)$ has been previously determined by
 Dokos et.\ al.~\cite{DDJSS12}. This section includes this result for completeness. 
 
\begin{prop}[Simion and Schmidt~\cite{ss:rp} Proposition 6]
All involutions $\iota \in \cI_n(231)$  have decomposition $\ii_k[\dd_{j_1},\dd_{j_2},\dots ,\dd_{j_k}]$ for some $k$ with $j_i\geq1$ for all $i\in[k]$.\hqed
\label{prop:SSdecompI(231)}
\end{prop}

Using this we can show any permutation that avoids both $231$ and $312$ is actually an involution in $\cI_n(231)$. 

\begin{prop}[Simion and Schmidt~\cite{ss:rp} Propositions 6 and 8] For $n\geq 1$ we have $\fS_n(231,312)=\cI_n(231)$ and further $\cI_n(231)=\cI_n(312)$.
\label{prop:S(213,312)=I(231)}
\end{prop}
\begin{proof}
Obviously $\cI_n(231)\subseteq\fS_n(231,312)$. Since all $\sigma \in \fS_n(231,312)$ avoid $231$ it is really a matter of showing that $\sigma$ is really an involution.  We will do this by showing $\sigma= \ii_k[\dd_{j_1},\dd_{j_2},\dots ,\dd_{j_k}]$ for some $k$ with $j_i\geq1$ for all $i\in[k]$, which we will do using induction. Since this is an involution we will be done at this point. 

This is easy to see for $n=1$, so we assume $n>1$ and all permutations in $\fS_m(231,312)$ have this form for $m<n$. Since $\sigma$ avoids $231$  we can write $\sigma= 132[\sigma_1,1,\sigma_2]$ as we noted in Proposition~\ref{prop:decomp} for some $\sigma_1\in \fS_{n_1}(213,312)$ and $\sigma_2\in \fS_{n_2}(213,312)$ with $n_1<n$. Since $\sigma_2$ also avoids $312$ we know $\sigma_2$  must have no ascents so is equal to $\dd_{n_2}$. Hence, $\sigma= 12[\sigma_1,\dd_{n_2+1}]$. By induction $\sigma_1$ has the decomposition stated, so we can conclude that $\sigma$  does as well. 

The map $r_1$ and the decomposition in Proposition~\ref{prop:SSdecompI(231)} imply $\cI_n(231)=\cI_n(312)$.
\end{proof}

Using the set equality shown in Proposition~\ref{prop:S(213,312)=I(231)} we have $I\cI_n(231)$, which was originally shown by  Dokos et.\ al. who determined $I_n(132,213)=\ol{I\cI}_n(231)$ using~\cite{DDJSS12} (Lemma 2.1).

\begin{prop} [Dokos et.\ al.~\cite{DDJSS12} Proposition 4.3] With $I\cI_0(231)=1$ we have for $n\geq 1$
that
$$\displaystyle I\cI_n(231)=\sum_{j=1}^n q^{\binom{j}{2}}I\cI_{n-j}(231).$$
\end{prop}

\begin{proof}
We can define  $I\cI_0(231)=1$ so let $n>0$. For $\iota \in \cI_n(213)$ we can use Simion and Schmidt's decomposition in Proposition~\ref{prop:SSdecompI(231)} to write $\iota = 12[\dd_j,\tau]$ for some $j\geq 1$ and $\tau\in \cI_{n-j}(231)$. Since there are $\binom{j}{2}$ inversions in $\dd_j$ and no inversions between $\dd_j$ and $\tau$ we find $\inv(\iota)=\binom{j}{2}+\tau$, which proves the equation.
\end{proof}

\subsection{The patterns $132$ and $213$}
\label{subsec132inv}

Guibert and Mansour in~\cite{GM02} study involutions avoiding 132 and describe  a decomposition and a generation function that counts  the number of occurrences of the patterns $\ii_k$. When $k = 2$ this counts the number of coinversions. Specifically,  their Theorem 4.2 in~\cite{GM02} produces the  generating function for involutions $C_I(x,q)=\sum_{\iota \text{ avoids }132}x^{|\iota|}q^{\coinv{\iota}}$ using the generating function for permutations $C_I(x,q)=\sum_{\sigma \text{ avoids }132}x^{|\sigma|}q^{\coinv{\sigma}}$, which is
$$C_I(x,q)=\frac{1+xC_I(xq,q)}{1-x^2C_S(x^2q^2,q^2)}\text{ where } C_S(x,q)=\frac{1}{1-xC_S(xq,q)}.$$
We begin this section by recounting a decomposition of involutions in $\cI_n(132)$ that can be found in~\cite{GM02} and~\cite{ss:rp} and then give a recursive definition of the generating function $\ol{I\cI_n}(132)$.  We also describe the generating function for fixed-point-free involutions avoiding 132 as this will be very useful in determining $\ol{I\cI}_n(132)$. 

\begin{lemma} [Guibert and Mansour~\cite{GM02} Proposition 3.17] The set ${\mathcal I}_n(132)$ is the disjoint union of 
\begin{enumerate}[(i)]
\item $\{12[\alpha,1]: \alpha \in \cI_{n-1}(132)\}$ and
\item $\{45312[\alpha,1, \beta, r_1(\alpha),1]: \alpha \in \fS_{k-1}(132), \beta \in \cI_{n-2k}(132) \text{ and } 1\leq k\leq \floor{\frac{n}{2}}\}$. 
\item Also, $F\cI_{2m}(132)=\{21[\alpha,r_1(\alpha)]:\alpha \in \fS_{m}(132)\}$.
\end{enumerate}
\label{lemma:132decomp}
\end{lemma}
\begin{proof}
First we will show that all  $\iota \in \cI_n(132)$ have decomposition as in (i) or (ii). It is known that since $\iota \in \fS_n(132)$ that $\iota = 231[\sigma_1,1,\sigma_2]$ for some permutations $\sigma_1$ and $\sigma_2$ that avoid $132$ as on the left in Figure~\ref{fig:132}. If $|\sigma_2|=0$ then $\iota$ is part of the set in (i). 
Otherwise, $|\sigma_2|\neq 0$. First we argue that $|\sigma_1|< |\sigma_2|$. We know that $\iota(|\sigma_1|+1)=n$ and $\sigma_2$ occurs in $\iota$ using the values in $[1,|\sigma_2|]$. Since $\iota$ is an involution $\iota(n)=|\sigma_1|+1\in [1,|\sigma_2|]$, which shows $|\sigma_1|<  |\sigma_2|$.
Since involutions are symmetric about the main diagonal and $|\sigma_1|< |\sigma_2|$  we have $\sigma_2=312[\beta,r_{1}(\sigma_1),1]$ for some $\beta \in \cI_{n-2k}(213)$ if $|\sigma_2|=k-1$. This assures that $r_1(\iota)=\iota$, which proves that $\iota$ is an element of the set in (ii). 
 
Next we will show that given an involution with decomposition as in (i) or (ii) that the involution avoids $132$. Consider we have an involution as stated in (i), $\iota =12[\alpha, 1]$ for $\alpha \in \cI_{n-1}(132)$, and a subsequence $abc$ that is a pattern 132. The subword $abc$ can not be part of $\alpha$ because $\alpha$ avoids $132$. We must have that $n$ is part of the pattern, but $n$ can only play the role of 3 in the pattern, which is not possible because $n$ occurs at the rightmost index. Hence, $\iota$ avoids $132$. Now consider an involution as described in (ii), $\iota =45312[1,\alpha,\beta,1,r_{1}(\alpha)]$ for some $ \alpha\in \fS_{k-1}(213)$ and $ \beta\in \cI_{n-2k}(213)$ with $1\leq k\leq\floor{n/2}$. Let $abc$ be a subsequence of $\iota$. We will show $abc$ is not the pattern 132 by considering how $abc$ occurs in the the five blocks. If all three letters occur in the same block then $abc$ is not the pattern $132$ since every block avoids $132$. If they occur in three different blocks then the pattern is still not $132$ since $45312$ avoids $132$. If $ab$  is in one block  that doesn't contain $c$ then due to the decomposition $c>\max\{a,b\}$ or $c<\min\{a,b\}$, which implies $abc$ is not the pattern 132. Say that $bc$ is in one block that doesn't contain $a$, then due to block sizes $bc$ is in either the second or third block, which implies $a>\max\{b,c\}$ and that $abc$ is not the pattern $132$. Hence all permutations described in (i) and (ii) avoid $132$.

Lastly, we will show (iii) the decomposition for fixed-point-free involutions. If $\iota\in F\cI_0(132)$, then $\iota=21[\epsilon,\epsilon]$ so assume $m>0$. In this case $\iota \in F\cI_{2m}(132)$, which implies that $\iota$ is not part of the set in (i) because $\iota$ is fixed-point-free. Since $\iota$ falls under case (ii) we have that $\iota = 45312[1,\alpha,\beta,1,r_1(\alpha)]$ as stated in this lemma. The involution $\beta$ must also avoid $132$ and be fixed-point-free, so by induction $\beta = 21[\gamma,r_1(\gamma)]$. 
Hence, $\iota=21[\tau,r_1(\tau)]$ for some $\tau=231[\alpha,1,\gamma] \in \fS_m(132)$.
\end{proof}

\begin{figure}
\begin{center}
\begin{tikzpicture} [scale = .4]
\draw (8,4.5) rectangle (3.5,0);
\draw (2.5,7.5) rectangle (0,5);
\filldraw [black] 
(3,8) circle (5pt);
\draw (0,0) --  (8,0) -- (8,8) -- (0,8) -- (0,0);
\draw (5.8,2) node {$\sigma_2$};
\draw (1.3,6.3) node {$\sigma_1$};
 \end{tikzpicture}
\hspace{30mm}
\begin{tikzpicture} [scale = .4]
\draw (0,0) rectangle (7,7);
\filldraw [black] 
(8,8) circle (5pt);
\draw (0,0) --  (8,0) -- (8,8) -- (0,8) -- (0,0);
\draw (3.5,3.5) node {$\alpha$};
 \end{tikzpicture} 
 \hspace{5mm}
\begin{tikzpicture} [scale = .4]
\draw (7.5,2.5) rectangle (5,0);
\draw (2.5,7.5) rectangle (0,5);
\draw (4.5,4.5) rectangle (3.5,3.5);
\filldraw [black] 
(3,8) circle (5pt)
(8,3) circle (5pt);
\draw (0,0) --  (8,0) -- (8,8) -- (0,8) -- (0,0);
\draw (4,3.95) node {$\beta$};
\draw (6.3,1.3) node {$r_1(\alpha)$};
\draw (1.3,6.3) node {$\alpha$};
 \end{tikzpicture}
\end{center}
\caption{On the left a general $\sigma \in \fS_n(132)$. On the right the possible diagrams for $\iota \in \cI_n(132)$.}
\label{fig:132}
\end{figure}
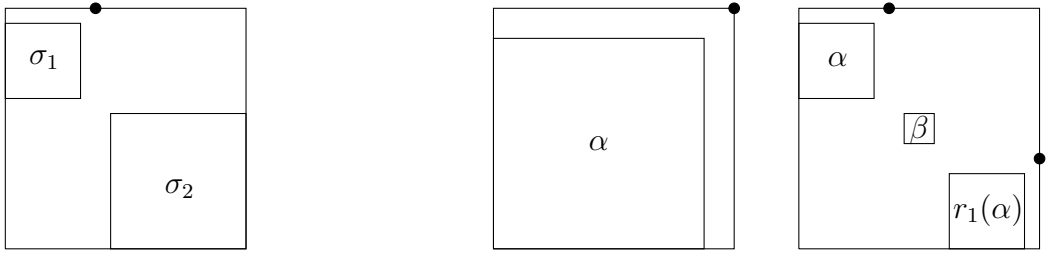

The structure for involutions in $\cI_n(213)$ can be similarly determined.

\begin{lemma} The set ${\mathcal I}_n(213)$ is the disjoint union of 
\begin{enumerate}[(i)]
\item $\{  12[1,\alpha]: \alpha \in \cI_{n-1}(213)\}$ and
\item $\{45312[1,\alpha,\beta,1,r_{1}(\alpha)]: \alpha\in \fS_{k-1}(213), \beta\in \cI_{n-2k}(213),1\leq k\leq\floor{n/2}\}$.
\item Also,  $ F{\mathcal I}_{2m}(213)=\{21[\alpha,r_1(\alpha)]: \alpha \in \fS_{m}(213)\}$.
\end{enumerate}
\label{lemma:213decomposition}
\end{lemma}

\begin{proof}
All these results come from Lemma~\ref{lemma:132decomp} and the map $r_{-1}$ since $r_{-1}(132)=213$. 
\end{proof}

We find that these generating functions are related to the $q$-Catalan numbers, $\tilde{C}_n(q)$, which Carlitz and Riordan~\cite{CR64} defined. We will particularly be seeing $C_n(q)=q^{\binom{n}{2}}\tilde{C}_n(q^{-1})$  in our calculations, which is  recursively defined by
$C_0(q)=1$ and
 $$C_n(q)=\sum_{k=0}^{n-1}q^kC_k(q)C_{n-k-1}(q).$$ 
We use the result by Dokos et.\ al.~\cite{DDJSS12} (Theorem 3.1) that the generating function for $\fS_n(132)$ and coinversions is $C_n(q)$.

\begin{thm} For $m\geq 0$ we have $\ol{IF\cI}_{2m}(132)=\ol{IF\cI}_{2m}(213)$ and
$$\displaystyle \ol{IF\cI}_{2m}(132)=C_m(q^2).$$
\label{thm:barFI_n(132)}
\end{thm}

\begin{proof}
We can write $\iota\in \cI_{2m}(132)$ by Lemma~\ref{lemma:132decomp} as $\iota = 21[\alpha,r_{1}(\alpha)]$ for $\alpha\in \fS_m(132)$. By Proposition~\ref{lemma:inv preserved} $r_{1}$ preserves $\inv$ and so preserves $\coinv$ as well, which tells us $\coinv(\iota)=2\coinv(\alpha)$. 
As result $\ol{IF\cI}_{2m}(132)$ is equivalent to the generating function for $\fS_m(132)$ using $\coinv$ with the substitution of $q^2$ for $q$. Dokos et.\ al.~\cite{DDJSS12} found that this generating function for $\fS(132)$ using $\coinv$ is $C_m(q)$, which proves the result. 

Since the map $r_1$ is a bijection from $\cI_n(132)$ to $\cI_n(213)$ that preserves $\inv$, $\coinv$ and the number of fixed points we must have $\ol{IF\cI}_{2m}(132)=\ol{IF\cI}_{2m}(213)$.
\end{proof}

We now have what we need to describe $I\cI_n(132)$. Recall from Theorem~\ref{thm:SimionSchmidt} that the cardinality of $\cI_n(132)$ is the central binomial coefficient so the generating function for $\ol{I\cI}_n(132)$ will be a $q$-analogue for the central binomial coefficient. This is not the standard $q$-analogue, but one that will parallel the following identity.  A corollary of Gould and Kaucky's work in~\cite{GK66} is
$$a_n=a_{n-1}+\sum_{k=1}^{\floor{n/2}}C_{k-1}a_{n-2k}$$
where $a_n=\binom{n}{\ceil{n/2}}.$ This identity appears in Simion and Schmidt's paper~\cite{ss:rp} (equation 5) with a discussion about integer lattice paths.

\begin{prop} With $\ol{I\cI}_0(132)=1$ we have for $n\geq 1$ that
$$ \ol{I\cI}_n(132)=q^{n-1} \ol{I\cI}_{n-1}(132)+\sum_{k=1}^{\floor{n/2}} q^{2(k-1)}C_{k-1}(q^2)\ol{I\cI}_{n-2k}(132).$$
\label{II132}
\end{prop}

\begin{proof}
 By Lemma~\ref{lemma:132decomp} if $\iota(n)=n$ then $\iota = 12[\alpha,1]$ for some $\alpha\in \cI_{n-1}(132)$, which implies $\coinv(\iota)=\coinv(\alpha)+n-1$. In any other case $\iota(n)=k\neq n$ and 
$\iota = 45312[\alpha,1,\beta,r_1(\alpha),1]$ for some $\alpha\in \fS_{k-1}(132)$ and $\beta \in \cI_{n-2k}(132)$.
 It follows from the decomposition of $\iota$ and from Proposition~\ref{lemma:inv preserved} that $r_1$ preserves $\coinv$ and so $\coinv(\iota)=2\coinv(\alpha)+\coinv(\beta)+2(k-1)$. Putting this all together and using Theorem~\ref{thm:barFI_n(132)} we get the above equality.
\end{proof}

\subsection{The pattern $321$}
\label{sec:inv321}

In this section we describe  involutions  avoiding $321$ particularly focusing the structure of the two-cycles and  the associated standard Young Tableaux, a concept we give a brief introduction to below. When listing the two-cycles of an involution we  use in this section and future sections the convention of listing all two-cycles $(s_1,t_1), (s_2,t_2),..., (s_m,t_m)$ so that each cycle is written with its minimum element on the  left, $s_i<t_i$, and the cycles themselves are ordered so that their minimum elements increase, $s_i<s_{i+1}$.

To introduce standard Young Tableau we first define an {\it integer partition} $\lambda=(\lambda_1, \lambda_2, \dots, \lambda_k)$ of $n$, $\lambda\vdash n$, which is a weakly decreasing sequence of positive integers that sum to $n$. Given an integer partition we can construct its {\it Young diagram} that has $\lambda_i$ boxes in row $i$ left-justified and labeled from top to bottom. We label the columns from left to right and define the {\it size}, $|\lambda |$, of a Young diagram to be the number of boxes. A standard Young Tableau, SYT,  of size $n$ is a Young diagram of size $n$ filled with numbers $1,2,\dots, n$ so that each box contains a unique number, each row is strictly increasing left to right and each column is strictly increasing from top to bottom. We will call the numbers in the boxes {\it fillings} and the underlying integer partitions its {\it shape}. See Figure~\ref{fig:RSK_invo} for an example.  The descent set, $\Des(P)$, of a SYT $P$ is the collection of all fillings $i$ such that $i+1$ appears in a lower row. There is a well-known bijection from  permutations  to pairs of SYT  of the same shape called the Robinson-Schensted-Knuth, RSK, correspondence.  This correspondence has many beautiful properties and we state the ones relevant to this paper in the next proposition. For more information see~\cite{S01} or~\cite{S99}. 
\begin{figure}
\begin{center}
\begin{tikzpicture}[scale = 1]
\begin{scope}[shift={(0,0)}]
\draw (0,0) node { \young({}{}{}{}{}{},{}{}{}{}{}{},{}{}{}{}{},{}{}{}) };
\draw (3.7,0) node { \young(135,24,6) };
\draw (7,0) node { \young(134,256,78) };
 \end{scope}
 \end{tikzpicture}
 \end{center}
\caption{From left to right we have the Young diagram for $\lambda=(5,5,4,2)$, the SYT associated the involution $\hat\iota =216453$ and the SYT associated to $21785634 =\hat\iota + (4,8)$ . }
\label{fig:RSK_invo}
\end{figure}
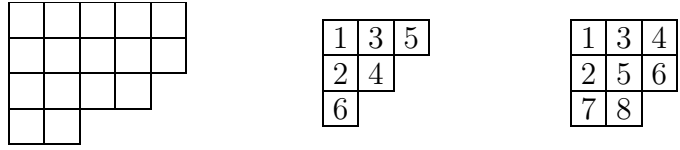

\begin{prop}  Let $\sigma\in \fS_n$ correspond by RSK to the pair of SYT tableau $(P,Q)$ of the same shape and size $n$. 
\begin{enumerate}[(i)]
\item The descent sets $\Des(\sigma)=\Des(Q)$ are equal. 
\item The pair of SYT for the inverse $\sigma^{-1}$ is $(Q,P)$. 
\item If $\sigma$ is an involution then $P=Q$. 
\item The length of the longest increasing sequence in $\sigma$ equals the length of the longest row in $P$.
\item The length of the longest decreasing sequence in $\sigma$ equals the length of the longest column in $P$.
\item The number of fixed points  in $\sigma$ equals the number of columns in $P$ with odd length. \hqed
\end{enumerate}
\label{SYTfacts}
\end{prop}

By (ii) in the above proposition we know an involution $\iota$ corresponds by RSK to $(P,P)$. This means we can associated an involution by RSK to a single SYT $P$. The next lemma is our first  about the structure of  two-cycles in an involution avoiding 321 and fillings of the associated SYT. We note that part (ii) in Lemma~\ref{lemma:321structure} can be seen as a corollary of B\'{o}na and Smith's Proposition 3.1 in~\cite{BS16}. We provide an alternative proof using an algorithm by Beissinger~\cite{B87} that we introduce before the proof. The first part of the next lemma  appears in Manara and Perelli Cippo's paper~\cite{MP11} (Proposition 2.3).

\begin{lemma}
Let $\iota \in \cI_n(321)$  and suppose its two-cycles are
   $(s_1,t_1), (s_2,t_2),..., (s_m,t_m)$.
\begin{enumerate}[(i)]
\item We must have that $t_1 < t_2 < ... < t_m$. 
\item If $\iota$ is fixed-point-free then the associated SYT has two rows of equal length, $m$ columns, and the $i$th column is filled with $s_i$ and $t_i$. 
\end{enumerate}
\label{lemma:321structure}
\end{lemma}

Between the hundreds of sets that are counted by the Catalan numbers, with recurrence $C_0=1$ and 
\begin{equation}C_n=\sum_{k=0}^{n-1}C_kC_{n-k-1},\label{Catalan}\end{equation}
 there are even more established bijections. See Stanley's book~\cite{S99} for more information. To prove the lemma above we show one of these bijections between $F\cI_n(321)$ and SYT with two rows of equal length using the RSK correspondence. To aid us in this proof and some later proofs we will present a short-cut algorithm  by Beissinger~\cite{B87} that  quickly determines SYT of involutions. Say that we have an involution $\hat\iota\in \cI_{n-2}$ and we want to add  another two-cycle $(i,j)$ with $1\leq i<j \leq n$. What we mean is, in one-line notation, increase all numbers $\geq i$ in $\hat\iota$ by one, further increase all numbers $\geq j$ by another one and have this word written in the indices $[n]\setminus \{i,j\}$. We get our new involution by placing $i$ at index $j$ and  $j$  at index $i$.  Beissinger notates this new involution $\iota=\hat\iota+(i,j)$. Her  algorithm describes the SYT of $\iota$ based off of the SYT of $\hat\iota$ where $\iota=\hat\iota+(i,n)$ in the  case of $j=n$. If $\hat{T}$ is the SYT for $\hat\iota$ we get the SYT for $\iota$ following three steps. 
\begin{enumerate}
\item Increase all fillings $\geq i$ in $\hat{T}$ by one. 
\item Insert $i$ as the Robinson-Schensted-Knuth bumping algorithm dictates. Say the bumping algorithm ends on row $r$. 
\item Insert $n$ at the end of row $r+1$. 
\end{enumerate}
With these three steps we arrive at the SYT for $\iota=\hat\iota+(i,n)$. See Figure~\ref{fig:RSK_invo} for an example. 
We now begin the proof for Lemma~\ref{lemma:321structure}.

\begin{proof}[Proof for Lemma~\ref{lemma:321structure}]
To show (i) let $\iota\in \cI_n(321)$ have two-cycles as stated but $t_j>t_{j+1}$ for some $j$. In this case we have the subword $t_jt_{j+1}s_{j+1}$ that is the pattern $321$. 

Now consider $\iota \in F\cI_{2m}(321)$ with two-cycles $(s_1,t_1), (s_2,t_2),..., (s_m,t_m)$. We will show (ii) by inducting on $m$. Part (ii) is true  by vacuum if $m=0$, so we assume that $m>0$ and  part  (ii) is true for all fixed-point-free involutions with length less than $2m$. By part (i) of this lemma the sequence of $t_i$'s increases so the Young diagram with two rows of length $m$ and  $s_i,t_i$ filling the $i$th column in increasing order is indeed a SYT. Call this SYT $T$.  We will show  that $T$ is also the SYT corresponding to $\iota$. Consider the involution $\hat\iota$ with two-cycles $(\hat s_1,\hat t_1), (\hat s_2,\hat t_2),..., (\hat s_{m-1},\hat t_{m-1})$ where $\hat s_i=s_i$ or $\hat s_i=s_i-1$ (similarly $\hat t_i=t_i$ or $\hat t_i=t_i-1$) depending on whether $s_i<s_m$ or $s_i>s_m$ respectively. This assures that $\hat\iota$ is a fixed-point-free involution avoiding $321$ and $\iota = \hat\iota+(s_m,t_m)$. By our inductive assumption the SYT $\hat T$ for $\hat\iota$ has two rows length $m-1$ with $\hat{s}_i$ and $\hat{t}_i$ in the $i$th column. 

Since $\iota= \hat\iota+(s_m,t_m)$ and $t_m=2m$ we can use Beissinger's algorithm. First increase all fillings  $\geq s_m$ in $\hat T$ by one, which means we have $s_i$ and $t_i$ are in the $i$th column for $i<m$. Step two has us insert $s_m$ via the bumping algorithm. Since the maximum of the first row is $s_{m-1}$, which is less than $s_m$, we place $s_m$ at the end of the first row, i.e. $m$th column. Since the algorithm ends on the first row we place $t_m=2m$ at the end of the second row, i.e. the $m$th column. Hence our final tableau is $T$, which proves part (ii). 
\end{proof}

We can actually calculate $\inv$ from $\iota$ easily from the two-cycles of an involution that avoids $321$.

\begin{lemma}
Let $\iota \in \cI_n(321)$  and suppose its two-cycles are
   $(s_1,t_1), (s_2,t_2),..., (s_m,t_m)$.
If  $(a,b)\in \Inv(\iota)$ then $a=s_i$ and $b=t_j$ for some $i$ and $j$ and additionally 
$$\displaystyle \inv(\iota)=\sum_{i=1}^{m}t_i-s_i.$$
\label{lemma:321inv}
\end{lemma}

\begin{proof}
Consider $(a,b)\in \Inv(\iota)$ for $\iota\in \cI_n(321)$ with two-cycles written as in this lemma. There are three possibilities for both $a$ and $b$ as they can be equal to a $s_i$, a $t_i$, or a fixed point. By Lemma \ref{lemma:321structure} the $s_i$'s and $t_i$'s form increasing sequences in $\iota$ so we know $a$ and $b$ cannot  both be  $s_i$'s nor can they be both  $t_i$'s. If $a$ and $b$ are fixed points then $(a,b)$ is not an inversion. If instead $a$ is a fixed point and $b$ is not then 
$ba\iota(b)$  is the pattern $321$. Similarly $b$ can not be a fixed point, so neither $a$ nor $b$ are fixed points. Take the case where $a=t_i$ and $b=s_j$ for some $i$ and $j$, then $\iota(t_i)<t_i<s_j<\iota(s_j)$ and $(a,b)$ is not an inversion.  The only remaining case possible is $a=s_i$ and $b=t_j$ for some $i$ and $j$.

Next we will show for any $i$ that the number of inversions $(s_i,b)\in \Inv(\iota)$ is $t_i-s_i$, and by the first part this is enough to complete the proof. Note that $s_j\in [s_i,t_i-1]$ implies  that index $t_j$ is to the right of index $s_i$ so $(s_i,t_j)$ is an inversion. Similarly, $t_j\in [s_i,t_i-1]$  implies that $(s_i,t_j)$ is an inversion. Because $\iota$ avoids $321$ the interval $[s_i,t_i-1]$ is comprised of only $s_j$'s and $t_j$'s, which implies that the number of inversions $(s_i,b)$ is at least the size of  the interval $|[s_i,t_i-1]|=t_i-s_i$. 
This counts all possible inversions $(s_i,b)$ because of the following. 
From the first part of the proof we know if $(s_i,b)$ is an inversion that $b=t_j$ for some $j$ and since $(s_i,t_j)$ is an inversion we must have $\iota(s_i)=t_i>\iota(t_j)=s_j$. 
There are then two cases $s_j\in [s_i,t_i-1]$ or $s_j<s_i$ and this second case implies $t_j\in [s_i,t_i-1]$.
\end{proof}

Now we have all the tools we need to establish the generating functions for $IF\cI_{2m}(321)$ and $I\cI_n(321)$. The generating function for the fixed-point-free case analogues the Catalan recurrence in equation~\eqref{Catalan}. 

\begin{thm} We have  $IF\cI_{0}(321)=1$ and for $m>0$
$$ IF\cI_{2m}(321)=\sum_{k=1}^{m}q^{2k-1}IF\cI_{2k-2}(321)IF\cI_{2(m-k)}(321).$$
\label{thm:fpfinv321}
\end{thm}

\begin{proof}
By Lemma~\ref{lemma:321structure} any involution $\iota \in F\cI_{2m}(321)$ with two-cycles    $(s_1,t_1), (s_2,t_2),..., (s_m,t_m)$ is associated to an SYT, $P$, with two rows of length $m$ and the $i$th column filled with $s_i$ and $t_i$. Let $k$ number the smallest column such that the collection of fillings in the first $k$ columns is the set $[2k]$. As result the first $k$ columns form a SYT, $P_1$, which is associated an involution $\tau_1$ with two-cycles $(s_1,t_1), (s_2,t_2),..., (s_k,t_k)$. The remaining columns, $(k+1)$st through $m$th, are filled with the numbers $[2k+1,2m]$ and if these fillings are decreased by $2k$ then we have a SYT, $P_2$, associated to an involution $\tau_2$ with two-cycles $(s_{k+1}-2k,t_{k+1}-2k),..., (s_m-2k,t_m-2k)$. Note that by Lemma~\ref{lemma:321inv} we have $\inv(\iota)=\inv(\tau_1)+\inv(\tau_2)$. 

Because rows and columns in SYT increase the maximum filling in any rectangle of squares is located in the lower-right corner and the minimum is in the upper-left corner. 
Recall our choice of $k$. The result is for any $i<k$ the first $i$ columns are not filled with all of $[2i]$, and since columns and rows increase the filling in column $i$ row $2$ must be larger than $2i$. Let $s\in[2i]$ be the smallest filling in columns $i+1$ through $k$. Since rows and columns increase $s$ must be in column $i+1$ and row $1$. As result we can see in $P_1$ that any filling in the second row is larger than the filling of its upper-right neighbor. 
Define $P_1'$  to be  $P_1$ but we remove the upper-left square filled with $1$, remove the lower-right square filed with $2k$, left align the squares and decrease all fillings by one. Because of what we have noted $P_1'$ has two rows length $k-1$ filled with $[2(k-1)]$, is  increasing along rows and columns and so must be a SYT associated to the involution $\tau_1'$ with two-cycles $(s_2-1,t_1-1),(s_3-1,t_2-1),\dots,(s_k-1,t_{k-1}-1)$. Using Lemma~\ref{lemma:321inv} again we have $\inv(\tau_1)=\inv(\tau_1')+2k-1$. Putting it all together $\inv(\iota)=\inv(\tau_1')+\inv(\tau_2)+2k-1$, which implies our recurrence.
\end{proof}

We now use the fixed-point-free case to describe the generating function $I\cI_n(321)$.

\begin{prop} We have $I\cI_1(321)=1$ and for $n>1$ 

$$I\cI_n(321)=IF\cI_{n}(321)+\sum_{k=0}^{\ceil{n/2}-1}  IF\cI_{2k}(321)  I\cI_{n-2k-1}(321).$$

\end{prop}

\begin{proof}
Consider $\iota \in \cI_n(321)$. If  $\iota$ is fixed-point-free then $n$ is even, which gives us the  term $IF\cI_{n}(321)$. Otherwise $\iota$ will a have fixed point. Let $f$ be the smallest index of any fixed point. Note that $f$ is precluded by a fixed-point-free involution of even length, which implies $f=2k+1$ is odd. Because $\iota$ avoids $321$ $\iota$ can be written as the inflation $123[\tau_1,1,\tau_2]$ where $\tau_1\in F\cI_{2k}(321)$ and $\tau_2\in \cI_{n-2k-1}(321)$. Since $\inv(\iota)=\inv(\tau_1)+\inv(\tau_2)$  we have our result. 
\end{proof}
\subsection{The pattern 123}
In this section we describe $\ol{I\cI}_n(123)$. We show  when $n$ is odd that $\ol{I\cI}_n(123)$ only has non-zero terms with even powers of $q$, which is also true for $\ol{F\cI}_{2m}(123)$.
Our approach in this section is to decompose $\iota$ by writing the involution as an addition of several two-cycles using Beissinger's~\cite{B87} notation defined in Section~\ref{sec:inv321}. We first determine which two-cycles we can add to an involution and preserve the avoidance of 123.

\begin{lemma} We have that
\begin{enumerate}[(i)]
\item if $\hat\iota\in\cI_{n-2}(123)$ and $a\leq\min\Asc(\hat\iota)+1$ where $\Asc(\hat\iota)\neq\emptyset$ then $\iota = \hat\iota + (a,n)$ avoids 123, and
\item if $\iota\in\cI_n(123)$ and $\iota = \hat\iota + (a,n)$ then either $\Asc(\hat\iota)=\emptyset$ or $a\leq\min\Asc(\hat\iota)+1$.
\end{enumerate}
\label{lem:123andminasc}
\end{lemma}
\begin{proof}Assume $\hat\iota\in\cI_{n-2}(123)$, $a\leq\min\Asc(\hat\iota)+1$, $\Asc(\hat\iota)\neq\emptyset$ and $\iota = \hat\iota + (a,n)$. We will show $\iota$ avoids 123 by considering instead if $\iota$ contains the pattern 123. Because $\hat\iota$ avoids 123 the pattern of 123 in $\iota$ must involve $n$ or $a$. The pattern can not use $n$ because $a\leq\min\Asc(\hat\iota)+1$ implies $\iota$ decreases before $n$. If the  pattern involves $a$ then $a$ plays the role of 3 and there exists a  coinversion $(i,j)$ with $i,j<a$ so $\iota(i)<\iota(j)$, which contradicts $\iota$ decreasing before $a$.  Hence $\iota$ avoids 123. 

Now assume that $\iota\in\cI_n(123)$ and $\iota = \hat\iota + (a,n)$. If we have $\Asc(\hat\iota)=\emptyset$ then we are done. In any other case $\Asc(\hat\iota)\neq \emptyset$. We know that $\iota$ avoids 123 so $\iota$ is  decreasing before $n$ on the indices in $[a-1]$. This means that $\hat\iota$ must also be decreasing on the indices $[a-1]$, which implies that $a\leq\min\Asc(\hat\iota)+1$.
\end{proof}

Given an involution $\iota$, we say we {\it pull off a two-cycle} when we  write $\iota=\hat\iota +(a,|\iota|)$ for $a<|\iota|$. Consider $\iota$ where we pull off many two-cycles,
$$\iota = \tau + (a_1,m+2) + (a_2,m+4) + \dots + (a_{\ell},m+2\ell).$$
We can continue to do this until $\tau$ has  a fixed point at the end or $\tau$ is the empty permutation.
However, to better describe the sequence $(a_1,a_2,\dots,a_{\ell})$ we will pull off two-cycles until $\tau$ ends in a fixed point or until $\tau$ for the first time has an empty ascent set. 

First, consider the case where $\tau$ ends with a fixed point. Note that since $\tau$ avoids 123 we must have $\tau$ decreasing before this fixed point, which implies that $\tau = 12[\dd_{m-1},1]$ and $\min\Asc(\tau) =m-1$. Further this means by Lemma~\ref{lem:123andminasc}  that $a_1\in [m]$. 
This also tells us the minimum ascent of $\tau + (a_1,m+2)$ is $a_1-1$ if $a_1\neq 1$. By Lemma~\ref{lem:123andminasc}  the values $a_2$ can take are $1\leq a_2\leq a_1$. This seems to imply that the sequence $(a_1,a_2,\dots,a_{j-1})$ is weakly decreasing, which is not fully the case. If we had $a_2 = 1$ then the $\min\Asc( \tau + (a_1,m+2)+(a_2,m+4))$ is $a_1$ so instead of having $a_3\leq a_2$ we actually have $a_3\leq a_1+1$, which brings us to the following definition.
Let $A_{m,\ell}$ define  a set of sequences $(a_1,a_2,\dots,a_{\ell})$ of positive integers such that 
\begin{enumerate}
\item $a_1\leq m$,
\item if $a_1,a_2,\dots, a_{i}$ are all $1$ then $a_{i+1}\leq m+i$ and
\item if $a_{i}\neq 1$ and $a_{i+1},a_{i+2},\dots,a_{i+r}$ all equal to 1 then $a_{i+r+1}\leq a_{i}+r$. 
\end{enumerate}

Next consider the case where $\tau$ has an empty ascent set so $\tau = \dd_{m}$. Recall the earlier assumption that we had stopped pulling off two-cycles because  $\tau + (a_1,m+2)$ had an ascent but $\tau$ did not. This does not change the requirements for the sequence $(a_1,a_2,\dots a_{\ell})$ except now $a_1\neq 1$ and $a_1\leq m+1$. For this   we  define
$$B_{m,\ell}=\{(a_1,a_2,\dots,a_{\ell})\in A_{m,\ell}:a_1\neq 1\}.$$

\begin{lemma} for all $\iota \in \cI_n(123)$ we can write $\iota$ uniquely as $$\iota = \tau + (a_1,m+2) + (a_2,m+4) + \dots + (a_{\ell},m+2\ell)$$ 
\begin{enumerate}[(i)]
\item where $\tau =12[\dd_{m-1},1]$, $m>1$  and $(a_1,a_2,\dots,a_{\ell})\in A_{m,\ell}$  or 
\item  $\tau =\dd_m$  and $(a_1,a_2,\dots,a_{\ell})\in B_{m+1,\ell}$.
\end{enumerate}
Conversely, any $\iota$ from case (i) or (ii) will avoid 123. 
\label{lem:123decomp}
\end{lemma}
\begin{proof}
Certainly if $\iota$ avoids 123 we can write $\iota = \tau + (a_1,m+2) + (a_2,m+4) + \dots + (a_{\ell},m+2\ell)$ where $\tau$ has a fixed point at $m$, $\tau=12[\dd_{m-1},1]$, or $\ell$ was the smallest integer where $\tau$ has an empty ascent set, $\tau=\dd_m$. These two cases intersect when $\tau= 1$, so to make these cases distinct the case when $\tau=1$ will fall under $\tau=\dd_1$ and we will only let $\tau$ fall under the case $\tau=12[\dd_{m-1},1]$ when $m>1$. If we do not have one case or the other then we could pull off another two-cycle from $\tau$. This means to  show the first part of the lemma we only have to show that the sequence $(a_1,a_2,\dots,a_{\ell})$ is in $A_{m,\ell}$ or $B_{m+1,\ell}$ respectively. For ease, define $\tau_i = \tau + (a_1,m+2) + (a_2,m+4) + \dots + (a_{i},m+2i)$ so $\iota=\tau_{\ell}$. If $\tau=12[\dd_{m-1},1]$ then $\min\Asc(\tau)=m-1$ so by Lemma~\ref{lem:123andminasc} $a_1\in [m]$. If instead
$\tau=\dd_m$ where $\tau_1$ has an ascent then $2\leq a_1\leq m+1$. This proves the condition on $a_1$ in both cases. We show the second and third conditions by inducting on $\ell$ and assume that $(a_1,\dots,a_{\ell-1})$ satisfies the second and third conditions. Say that there exists an $a_{i}\neq 1$ but $a_{i+1},a_{i+2},\dots, a_{\ell-1}$ all equal 1.
This means that $\min\Asc(\tau_i)=a_{i}-1$ and $\min\Asc(\tau_{\ell-1})=a_{i}+\ell-i-2$ where $\ell-i-1$ is the number of terms in $[i+1,\ell-1]$ so $a_{\ell}\leq a_{i}+\ell-i-1$. This proves the third condition so all we have left is to consider the case where $a_i= 1$ for all $i<\ell$. In this case $a_1,\dots,a_{\ell-1}$ all are 1, which can only happen in the case where $\tau=12[\dd_{m-1},1]$. So  $\min\Asc(\tau_{\ell-1})=m+\ell-2$, which implies that $a_{\ell}\leq m-1+\ell-1$ and shows that we satisfy the second condition. 

Conversely, assume that $\iota$ is as in case (i) or (ii) from this lemma. We will show that $\iota$ avoids 123 by induction on $\ell$. Certainly by Lemma~\ref{lem:123andminasc} we know $12[\dd_{m-1},1]+(a_1,m+2)$ avoids 123 since $a_1\in [m]$ and $\min\Asc(12[\dd_{m-1},1])=m-1$. Since $\dd_m+(a_1,m+2)$ avoids 123 and gains an ascent  for any choice of $a_1\in[2,m+1]$ any $\iota$ from case (i) or (ii) avoids 123 and has an ascent when $\ell=1$. Assume $\ell>1$ and $\tau_{\ell-1}$ avoids 123 and has an ascent. If $a_{\ell-1}\neq 1$ then $\min\Asc(\tau_{\ell-1})=a_{\ell-1}-1$. Whether $(a_1,a_2,\dots,a_{\ell})$ is in $A_{m,\ell}$ or $B_{m+1,\ell}$ we still have $a_{\ell-1}\geq a_{\ell}$ so by Lemma~\ref{lem:123andminasc} we must have that $\tau_{\ell}=\iota$ avoids 123. Next consider the case when $a_{i}\neq 1$ but $a_{i+1},a_{i+2},\dots, a_{\ell-1}$ all equal 1. Whether $(a_1,a_2,\dots,a_{\ell})$ is in $A_{m,\ell}$ or $B_{m+1,\ell}$ we still have $a_{\ell}\leq a_{i}+\ell-i-1$, $\min\Asc(\tau_{i})=a_i-1$ and $\min\Asc(\tau_{\ell-1})=m+\ell-2$, which by Lemma~\ref{lem:123andminasc} implies $\tau_{\ell}=\iota$ avoids 123. Our last case is when $a_i= 1$ for all $i<\ell$ so $a_{\ell}\leq m+\ell-1$, which only happens in the case when $\tau=12[\dd_{m-1},1]$. Hence, $\min\Asc(\tau_{\ell-1})=m+\ell-2$, which implies by Lemma~\ref{lem:123andminasc} that $\tau_{\ell}=\iota$ avoids 123. 
\end{proof}

This lemma describes how we can decompose $\iota$  avoiding 123 uniquely as an addition of two-cycles. We are particularly interested in this because we can calculate $\inv(\iota)$ from an addition of two-cycles. However, our calculation turns out nicer when considering $\coinv$ instead. 

\begin{lemma}
If $|\tau|=m$ and $\iota = \tau + (a_1,m+2) + (a_2,m+4) + \dots + (a_{\ell},m+2\ell)$
then $$\coinv(\iota)=\coinv(\tau)+2(a_1+a_2+\dots + a_{\ell}-\ell).$$
\label{lem:123inv}
\end{lemma}
\begin{proof}
We will first consider $\iota = \tau + (a,n)$ for $|\iota|=n$. The coinversions of $\iota$ come from the coinversions of $\tau$, the coinversions from index $n$ and the coinversions from index $a$. The number of coinversions from $\tau$ is $\coinv(\tau)$. The number of coinversions from index $a$ is $a-1$ because all indices $i$  to the left of $\iota(a)=n$ in $\iota$ form a coinversion $(i,a)$ because $\iota(i)<n$. The number of coinversions from index $n$  is $a-1$ because $\iota(n)=a$ is at the end of $\iota$ and all numbers smaller than $a$ are to its left. There is  not a coinversion between $n$ and $a$ so $\coinv(\iota)=\coinv(\tau)+2a-2$. 
Applying this to the full sum of two-cycles, $\iota= \tau + (a_1,m+2) + (a_2,m+4) + \dots + (a_{\ell},m+2\ell)$, gives us $\coinv(\iota)=\coinv(\tau) + 2(a_1+a_2+\dots + a_{\ell}-\ell).$
\end{proof}

Using this lemma we can calculate the generating function $\ol{I\cI}_n(123)$ from the set of sequences  $(a_1,a_2,\dots,a_{\ell})$ in $A_{m,\ell}$ or $B_{m,\ell}$ so we define the function
$$A_{m,\ell}(q)=\sum_{(a_1,a_2,\dots,a_{\ell})\in A_{m,\ell}}q^{a_1+a_2+\dots +a_{\ell}-\ell}$$
for the set $A_{m,\ell}$ and  define $B_{m,\ell}(q)$ similarly for the set $B_{m,\ell}$. We give a recurrence on $\ell$ for these functions and then describe $\ol{I\cI}_n(123)$. 

\begin{lemma}Given $A_{m,0}(q)=1$ for all $m\geq 1$ we have for $m,\ell\geq 1$ that
$$A_{m,\ell}(q)=A_{m+1,\ell-1}(q)+\sum_{i = 2}^{m}q^{i-1}A_{i,\ell-1}(q),$$
and with $B_{1,\ell}(q)=0$ for $\ell\geq 0$,  $B_{m,0}(q)=1$ for all $m>1$ we have for $m>1$ and $\ell\geq 1$
$$B_{m,\ell}(q)=\sum_{i = 2}^{m}q^{i-1}A_{i,\ell-1}(q).$$
\end{lemma}
\begin{proof} First, we will prove the recurrence for $A_{m,\ell}(q)$. Consider the sequence $(a_1,a_2,\dots, a_{\ell})\in A_{m,\ell}$. Because the associated term is $q^{a_1+a_2+\dots +a_{\ell}-\ell}$ we can say that each $a_i$ contributes $q^{a_i-1}$ to the product. For any $a_1\in[2,m]$ we know $(a_2,\dots, a_{\ell})\in A_{a_1,\ell-1}$, which gives us the terms in the  summation. If instead $a_1 = 1$ then $a_2$ can be at most $m+1$ so $(a_2,\dots, a_{\ell})\in A_{m+1,\ell-1}$, which gives us the term $A_{m+1,\ell-1}(q)$ and completes the proof for the first recurrence. 

For the second recurrence consider $(a_1,a_2,\dots, a_{\ell})\in B_{m,\ell}$ so we always have $a_1\in[2,m]$. It follows that $(a_2,\dots, a_{\ell})\in A_{a_1,\ell-1}$ since $a_2$ can actually be 1. Since $a_1\neq 1$ this finishes the proof of the second recurrence. 
\end{proof}
We now have everything we need to determine $\ol{I\cI}_n(123)$, but we will do so in two cases. The first will be when $n$ is odd and the second when $n$ is even. This distinction will be important since we consider the number of fixed points in $\iota$, which is tied to the parity of $n$. If an involution has $k$ two-cycles then these two-cycles form a perfect matching and together use $2k$ indices. The remaining $n-2k$ indices are fixed points so the number of fixed points in $\iota$ always shares  parity with $n$. Any involution avoiding 123 will have at most two fixed points else we form the pattern 123. If $n=2k+1$ is odd then certainly there is exactly one fixed point. 

\begin{thm}For $n=2k+1$ and $k\geq 0$  we have
$$\ol{I\cI}_{2k+1}(123)=\sum_{j=1}^kq^{2j}A_{2j+1,k-j}(q^{2})+\sum_{j = 0}^kB_{2j+2,k-j}(q^{2}).$$
\label{thm:123invodd}
\end{thm}
\begin{proof}
If $\iota$ avoids 123 and has  length $n=2k+1$ then $\iota$ must have exactly one fixed point. Looking at Lemma~\ref{lem:123decomp} we have two cases. If $\iota$ falls under case (i) then $\iota = 12[\dd_{2j},1] + (a_1,2j+3) + \dots + (a_{k-j},2k+1)$ for some $1\leq j\leq k$ where $(a_1,\dots, a_{k-j})\in A_{2j+1,k-j}$. Also this tells us by Lemma~\ref{lem:123inv} that $\coinv(\iota)=2j+2(a_1+\dots+a_{k-j}-(k-j))$, which gives us the first summation. 
If  $\iota$ instead falls under (ii) of Lemma~\ref{lem:123decomp} then $\iota = \dd_{2j+1} + (a_1,2j+3) + \dots + (a_{k-j},2k+1)$ for some $0\leq j\leq k$ where $(a_1,\dots, a_{k-j})\in B_{2j+2,k-j}$. By Lemma~\ref{lem:123inv} we have that $\coinv(\iota)=2(a_1+\dots+a_{k-j}-(k-j))$, which gives us the second summation. 
\end{proof}

If $\iota$ avoids 123 and $n=2k$ is even then there can be either no fixed point or two fixed points. 

\begin{thm}For $n=2k$ and $k\geq 1$  we have
$$\ol{I\cI}_{2k}(123)=\sum_{j=1}^kB_{2j+1,k-j}(q^{2})+\sum_{j = 1}^kq^{2j-1}A_{2j,k-j}(q^{2}).$$
\label{thm:123inveven}
\end{thm}
\begin{proof}
If $\iota$ avoids 123 and has length $n=2k$ then $\iota$ must have zero or two fixed points. Considering the case where $\iota$ has zero fixed points $\iota$ must fall under case (ii) in Lemma~\ref{lem:123decomp}  and $\iota = \dd_{2j}+ (a_1,2j+2) + \dots + (a_{k-j},2k)$ for some $1\leq j\leq k$ where $(a_1,\dots, a_{k-j})\in B_{2j+1,k-j}$. By Lemma ~\ref{lem:123inv} we have $\coinv(\iota)=2(a_1+\dots+a_{k-j}-(k-j))$, which gives us the first summation. 
If $\iota$ instead has two fixed points then $\iota$ fall under case (i) in Lemma~\ref{lem:123decomp} meaning $\iota = 12[\dd_{2j-1},1]+ (a_1,2j+2) + \dots + (a_{k-j},2k)$ for some $1\leq j\leq k$ where $(a_1,\dots, a_{k-j})\in A_{2j,k-j}$. By Lemma~\ref{lem:123inv} we have  $\coinv(\iota)=2j-1+2(a_1+\dots+a_{k-j}-(k-j))$, which gives us the second summation. 
\end{proof}

We can use the proof of Theorem~\ref{thm:123inveven} to determine the generating function for $\coinv$ and involutions avoiding 123 in the fixed-point-free case. 

\begin{cor}For $n=2k$ and $k\geq 1$  we have
$$\ol{IF\cI}_{2k}(123)=\sum_{j=1}^k B_{2j+1,k-j}(q^{2}).$$ 
\vspace{-1cm}

\hqed
\vspace{.2cm}
\end{cor}

One interesting observation about $\ol{I\cI}_n(123)$ happens for odd $n=2k+1$. From the formula in Theorem~\ref{thm:123invodd} we can see that $\ol{I\cI}_{2k+1}(123)$ can only have  even powers of $q$. As result $\coinv(\iota)$ is even for all $\iota\in \cI_{2k+1}(123)$. This is similarly true for $\iota\in F\cI_{2k}(123)$. 
\begin{cor}
For $\iota$ avoiding $123$,  $\coinv(\iota)$ is odd if and only if $|\iota|$ is even and $\iota$ has a fixed point. \hqed
\label{cor:123inv}
\end{cor}



\section{Length three patterns and maj}
\label{maj}

Just like for inversions we find that the  $M\cI$-Wilf equivalence classes for length three patterns are determined trivially, so this section's focus will be on describing the generating functions. Some of these functions have been studied by others like  Dokos et.\ al.~~\cite{DDJSS12} who determined $M\cI_n(231)$ and Barnabei et.\ al.~\cite{BBES14} who independently  found that $M\cI_n(321)$ is the standard central $q$-binomial coefficient whose proof gives a connection to hook decompositions.  The bijection we present later in Section~\ref{subsec321} is shorter and gives a connection to core, an unrelated concept  that is used for proving symmetric chain decompositions in poset theory. 

In order to be complete we  present a description for every generating function for all length three patterns. 
The functions $M\cI_n(123)$ and $M\cI_n(213)$ will be described using $M\cI_n(321)$ and $M\cI_n(132)$ respectively because we will additionally be proving the symmetry $M\cI_n(\pi_1;q)=q^{\binom{n}{2}}M\cI_n(\pi_2;q^{-1})$ between the pairs of respective patterns in Theorem~\ref{thm:132symm213} and Proposition~\ref{thm:123&321 symmetry}. The symmetry between  the patterns 123 and 321 can be proven using  the Robinson-Schensted-Knuth correspondence and transposing tableaux, a map that has been studied and used in many papers including Simion and Schmidt~\cite{ss:rp}, Barnabei et.\ al.~\cite{BBES14} and Deutsch et.\ al.~\cite{DRS07}. This map has also been described in more explicit detail by B\'{o}na and Smith in~\cite{BS16} (Section 3) whose description subverts the RSK algorithm and transposition. Despite it being a similar symmetry, proving this symmetry between the patterns 132 and 213 will require a different map.

Though mostly for the pattern 132, it will be easier for us to describe the generating function in some cases using the different but related statistics $\comaj$ and $\asc$, defined early in Section~\ref{inv},  and  the generating function in the fixed-point-free case.
 The associated generating functions will be notated
$$\ol{  M\cI}_n(\pi;q,t)=\sum_{\iota \in \cI_n(\pi)} q^{\comaj(\iota)}t^{\asc(\iota)}$$
and 
$$\ol{  MF\cI}_n(\pi;q,t)=\sum_{\iota \in F\cI_n(\pi)} q^{\comaj(\iota)}t^{\asc(\iota)}.$$
Determining these functions is equivalent to determining the ones for the major index due to the following identities. 

\begin{lemma}
We have the following equalities involving the statistics $\maj$,   $\des$,  $\comaj$ and $\asc$. 
\begin{enumerate}
\item For $\sigma \in \fS_n$ we have $\maj(\sigma)=\binom{n}{2}-\comaj(\sigma)$ and $\des(\sigma)=n-1-\asc(\sigma)$.
\item $\displaystyle q^{\binom{n}{2}}t^{n-1}\ol{M\cI}_n(\pi;q^{-1},t^{-1})=\sum_{\iota \in \cI_n(\pi)} q^{\maj(\iota)}t^{\des(\iota)}.$
\item $\displaystyle q^{\binom{n}{2}}t^{n-1}\ol{MF\cI}_n(\pi;q^{-1},t^{-1})=\sum_{\iota \in F\cI_n(\pi)} q^{\maj(\iota)}t^{\des(\iota)}.$
\end{enumerate}
\label{lemma:majproperties}
\end{lemma}
\begin{proof}
All these equalities are true  by the fact that $i\in[n-1]$ is an ascent or a descent for any permutation $\sigma$ of length $n$. So the disjoint union $\Des(\sigma)\cup\Asc(\sigma)$ is $[n-1]$. 
\label{lemma:ascdessymm}
\end{proof}

These bivariate functions determine the generating function for $\maj$ since $q^{\binom{n}{2}}\ol{M\cI}_n(\pi;q^{-1},1)=M\cI_n(\pi)$.

\subsection{$M\cI$-Wilf equivalence classes for length three patterns}

The equivalence classes for the major index are trivially determined.

\begin{prop} The only non-singleton $M{\mathcal I}$-Wilf equivalence class for length three patterns is $[231]_{M\cI}=\{231,312\}$. 
 \label{prop:MIWilfequiv}
\end{prop}
\begin{proof} From Proposition~\ref{prop:S(213,312)=I(231)} we know that $\cI_n(231)=\cI_n(312)$ so these patterns are in the same $M{\mathcal I}$-Wilf equivalence class. From the case of $n=3$ we can see that this is the only non-singleton class for length three patters. 
\end{proof}

Using the map $r_1$ we can easily show that the patterns $\pi$ and $r_1(\pi)=\pi^{-1}$ are always in the same $M\cI$-Wilf equivalence class since $r_1$ is the identity map on involutions. From computational data it does seem that the $M\cI$-Wilf equivalence classes are precisely these formed by a pattern $\pi$ and its inverse.

\begin{conj}
The only non-singleton $M\cI$-Wilf Equivalence classes are $[\pi]_{M\cI}=\{\pi, r_1(\pi)\}$ when $\pi$ is not an involution.
\end{conj}

The above conjecture is checked to be true by computer for patterns up to length 5. For permutations the $M$-Wilf equivalence classes can be much larger for example Dokos el. al.~\cite{DDJSS12} conjectured $[1423]_M=\{1423,2314,2413\}$ and $[3142]_M =\{3142,3241,4132\}$, which was proven by Bloom~\cite{B14} (Theorem 2.1 and Corollary 2.1). Dokos et.\ al.\ also conjectured $132[\ii_m,1,\dd_k]$ and $231[\ii_m,1,\dd_k]$ are $M$-Wilf equivalent and  $213[\dd_m,1,\ii_k]$ and $312[\dd_m,1,\ii_k]$ are as well. Yan, Ge and Zhang~\cite{YGZ15} (Theorem 1.3) proved this conjecture to be true in the case of $k = 1$. 

\subsection{The patterns 231 and 312}

By Proposition~\ref{prop:S(213,312)=I(231)}  we know $\fS_n(312,231)=\cI_n(231)$, so the generating function has already been determined by Dokos et.\ al.~\cite{DDJSS12} to be the following.

\begin{prop}[Dokos et.\ al.\ Proposition 5.2] We have for $n\geq 1$
 $$M\cI_n(231)=\prod_{k=0}^{n-1}(1+q^k).$$ 
 \end{prop}
 \begin{proof}The decomposition of an involution avoiding 231 by Proposition~\ref{prop:SSdecompI(231)} is $\ii_k[\dd_{j_1},\dd_{j_2},\dots ,\dd_{j_k}]$  where $j_i\geq1$ for all $i$. This determines the unique descent set $\Des(\iota)=[n-1]\setminus\{j_1,j_1+j_2,\dots, j_1+\dots+ j_{k-1}\}$. Conversely, given a set $D\subseteq [n-1]$ we can construct $\iota$ with $\Des(\iota)=D$. This tells us $$M\cI_n(231)=\sum_{D\subseteq [n-1]} q^{\sum_{i\in D}i},$$ which is known to be $\prod_{k=0}^{n-1}(1+q^k).$
\end{proof}
 The argument presented above is an extension of the argument used by Simion and Schmidt~\cite{ss:rp} (Proposition 6) to count $\cI_n(231)$.

\subsection{The pattern 132}

The generating function  $M\cI_n(132)$ is different from most in that it has  internal zeros. A polynomial has an {\it internal zero} if there is a term $q^k$ with a zero coefficient but there exists two other terms $q^i$ and $q^j$ with $i<k<j$ that have non-zero coefficients. The internal zeros of $M\cI_n(132)$ occur on a single interval of powers just before the  term $q^{\binom{n}{2}}$, which can even be seen when $n=3$. After proving this fact we describe $M\cI_n(132)$ recursively in several steps using nothing more than the decomposition of involutions avoiding 132.

\begin{prop}If $\iota \in {\cI}_n(132)$ then
\begin{enumerate}[(i)]
\item $\maj(\iota)=\binom{n}{2}$ or $\maj(\iota) \leq \binom{n}{2}-{\lceil n/2 \rceil}$,
\item this bound is sharp and 
\item for every non-negative $k\leq \binom{n}{2}-{\lceil n/2 \rceil}$ there exists some $\iota \in {\cI}_n(132)$ with $\maj(\iota)=k$. 
\end{enumerate}
\end{prop}
\label{thm:132internalzeros}
\begin{proof}
To show (i) we will show that any $\iota\in \cI_n(132)$ has either $\Asc(\iota)=\emptyset$ or an $i \geq {\lceil n/2 \rceil}$ with  $i\in \Asc(\iota)$ by induction on $n$. This is easy to see for $n =1,2$.

Let $n>2$. According to Lemma~\ref{lemma:132decomp} we have two cases for $\iota\in \cI_n(132)$. The first is if $\iota(n)=n$ then $\iota$ has an ascent at $n-1$ that is at least ${\lceil n/2 \rceil}$ so we are done. The second case is that $\iota(n)=k\neq n$ then $\iota$ has the form $\iota =45312[\alpha,1, \beta, r_1(\alpha),1]$ where $\alpha \in \fS_{k-1}(132)$,  $\beta \in \cI_{n-2k}(132)$ and  $1\leq k\leq \floor{\frac{n}{2}}$, which can been seen in Figure~\ref{fig:132}. If $\alpha$ is not the empty permutation then we again have an ascent at $n-1$ and we are done just like in the first case. Consider the case where $\alpha$ is empty so $\iota =321[1, \beta, 1]$ with $\beta\in \cI_{n-2}(132)$. By our inductive assumption we could have $\Asc(\beta)=\emptyset$, which implies $\Asc(\iota)=\emptyset$ so we are done. Otherwise by induction there is some $i \geq {\lceil (n-2)/2 \rceil}$ that is in $\Asc(\beta)$. This implies that $i+1\geq {\lceil n/2 \rceil}$ is in $\Asc(\iota)$ and we have finished the proof of (i). 

Note that  (iii) implies (ii). Say $k \in [0,\binom{n}{2}-{\lceil n/2 \rceil}]$ then there exists a choice of $a$ and $b$ with $b\leq a$ where $k$ is rewritten as  $k = \binom{a+1}{2}-b$. If $b\leq {\lfloor n/2 \rfloor}$  then consider 
$$\iota =453126[\dd_b,1,\dd_{a-2b},\dd_b,1,\ii_{n-a-2}]$$ 
that has $\Des(\iota)=[a]-\{b\}$ and $\maj(\iota)=\binom{a+1}{2}-b$. If $b> {\lfloor n/2 \rfloor}$  then consider 
$$\iota =42315[\dd_{a-b},\dd_{2b-a},1,\dd_{a-b},\ii_{n-a-1}]$$ 
that has $\Des(\iota)=[a]-\{b\}$ and $\maj(\iota)=\binom{a+1}{2}-b$. See Figure~\ref{fig:132allmaj} for an example. This proves (ii) and (iii).
\end{proof}

\begin{figure}
\begin{center}
\begin{tikzpicture} [scale = .4]
\draw (0,5) rectangle (1,6);
\draw (5,0) rectangle (6,1);
\draw (1.5,1.5) rectangle (4.5,4.5);
\draw (7,7) rectangle (8.5,8.5);
\filldraw [black] 
(1.5,6.5) circle (5pt)
(6.5,1.5) circle (5pt)
(5.5,.5) circle (5pt)
(.5,5.5) circle (5pt)
(2,4) circle (5pt)
(4,2) circle (5pt)
(2.6,3.3) circle (5pt)
(3.3,2.6) circle (5pt)
(7.5,7.5) circle (5pt)
(8,8) circle (5pt);
\draw (0,0) --  (8.5,0) -- (8.5,8.5) -- (0,8.5) -- (0,0);
 \end{tikzpicture}
 \caption{The  diagram for $\iota =453126[\dd_1,1,\dd_4,\dd_1,1,\ii_2]$ with $\maj(\iota)=20=\binom{7}{2}-1$.}
 \end{center}
 \label{fig:132allmaj}
 \end{figure}

We describe the generating function for the pattern 132 in several steps just as we did for $\inv$ in Section~\ref{subsec132inv} by first determining the generated function in the fixed-point-free case. First, we present a useful lemma that describes a fact about $\maj$ and $\des$ for involutions avoiding 132. 

\begin{lemma} For $\iota \in F\cI_{2m}(132)$ with $\iota=21[\alpha,r_{1}(\alpha)]$ and $\alpha \in \fS_m(132)$ we have 
 \begin{enumerate}[(i)]
 \item $\asc(\iota)=2\asc(\alpha)$ and 
 \item $\comaj(\iota)=\comaj(\alpha)+\comaj(r_{1}(\alpha))+m\asc(\alpha)$. 
 \end{enumerate}
\label{lemma:132majasc}
\end{lemma}

\begin{proof}Let $\iota \in F\cI_{2m}(132)$ with $\iota=21[\alpha,r_{1}(\alpha)]$ as stated. Part (ii) quickly follows from the fact that $\Asc(\iota)=\Asc(\alpha)\cup (\Asc(r_{1}(\alpha))+m)$ and part (i).

For  (i) it will be sufficient to show that $\asc(\alpha)=\asc(r_{1}(\alpha))$ for any $\alpha \in \fS_m(132)$ using induction on $m$. It is easy to see this is true for $m=1$ so let $m>1$. We can decompose  $\alpha = 231[\alpha_1,1,\alpha_2]$ for some
$\alpha_1$ and $\alpha_2$ that avoid 132. 
By induction we know $\asc(\alpha_1)=\asc(r_{1}(\alpha_1))$ and $\asc(\alpha_2)=\asc(r_{1}(\alpha_2))$. First consider the case where $|\alpha_1|\neq 0$. Since  in this case $\asc(\alpha)=\asc(\alpha_1)+\asc(\alpha_2)+1$ and $\asc(r_{1}(\alpha))=\asc(r_{1}(\alpha_2))+\asc(r_{1}(\alpha_1))+1$ we quickly can conclude that $\asc(\alpha)=\asc(r_{1}(\alpha))$. If instead $|\alpha_1|=0$ then $\asc(\alpha)=\asc(\alpha_2)$ and $\asc(r_1(\alpha))=\asc(r_1(\alpha_2))$ so  $\asc(\alpha)=\asc(r_1(\alpha))$.
\end{proof}

Using the lemma above we describe the fixed-point-free generating function for $\maj$ and involutions avoiding 132.

\begin{prop}  Define $F_{2m}(q,t)= \ol{MF\cI}_{2m}(132;q,t)$. We have $F_0(q,t)=1$ and for $m\geq 1$
$$F_{2m}(q,t)=F_{2(m-1)}(q,qt)+\sum_{k=1}^{m-1}q^{2m+k-1}t^2F_{2k}(q,q^{\frac{2m-2k-1}{2}}t)F_{2(m-k-1)}(q,q^{k+1}t).$$
\end{prop}

\begin{proof}
Let $\iota \in F\cI_{2m}(132)$. 
By Lemma~\ref{lemma:132decomp} we know that $\iota = 21[\alpha,r_{1}(\alpha)]$ for some $\alpha \in \fS_m(132)$. Since $\alpha$ avoids 132 we can write  $\alpha = 231[\alpha_1,1,\alpha_2]$ for some $\alpha_1 \in \fS_k(132)$, $\alpha_2 \in \fS_{m-k-1}(132)$ and $0\leq k\leq m-1$. Also, we will define  $x = 21[\alpha_1,r_{1}(\alpha_1)]$ and  $y= 21[\alpha_2,r_{1}(\alpha_2)]$. 

First consider the case where $k=0$ we then have that $\comaj(\iota)=\comaj(y)+\asc(y)$ and $\asc(\iota)=\asc(y)$. Summing over all possible $y\in F\cI_{2(m-1)}(132)$ will give us the term $F_{2(m-1)}(q,qt)$ in our sum. 

Next consider the case where $1\leq k\leq m-1$. We have $\comaj(\alpha)=\comaj(\alpha_1)+k+\comaj(\alpha_2)+(k+1)\asc(\alpha_2)$. Similarly, $\comaj(r_{1}(\alpha_1))=\comaj(r_{1}(\alpha_2))+\comaj(r_{1}(\alpha))+(m-k-1)\asc(\alpha_1)+m-1$. Also, $\asc(\alpha)=\asc(\alpha_1)+\asc(\alpha_2)+1$. Using the result of Lemma~\ref{lemma:132majasc} we get 
$$\comaj(\iota)=\comaj(x)+\comaj(y)+\frac{2m-2k-1}{2}\asc(x)+(k+1)\asc(y)+2m+k-1$$
and 
$$\asc(\iota)
=\asc(x)+\asc(y)+2.$$
Summing over all possible $x\in F\cI_{2k}(132)$ and $y\in F\cI_{2(m-k-1)}(132)$ gives us the term 
in the summation and we are done. \end{proof}

Now that we have described the fixed-point-free generating function for the pattern 132 we can describe the  generating function $M\cI_n(132;q,t)$ recursively. 

\begin{prop} Defining $M_n(q,t)=\ol{M\cI}_n(132;q,t)$ and $F_{n}(q,t)=\ol{MF\cI}_n(132;q,t)$ we have $M_0(q,t)=M_1(q,t)=1$ and for  $n\geq 2$,
$$ M_n(q,t)=q^{n-1}tM_{n-1}(q,t)+M_{n-2}(q,qt)+\sum_{k=2}^{\floor{n/2}}q^{n+k-2}t^2F_{2(k-1)}(q,q^{\frac{n-2k+1}{2}}t){M}_{n-2k}(q,q^kt).$$
\end{prop}

\begin{proof}Given an involution $\iota$ that avoids 132  by Lemma~\ref{lemma:132decomp} we can write $\iota = 12[\tau,1]$ for some $\tau\in \cI_{n-1}(132)$. In this case $\comaj(\iota)=\comaj(\tau)+n-1$ and $\asc(\tau)=\asc(\iota)+1$, which gives us the term $q^{n-1}tM\cI_{n-1}(132;q,t)$. 

In any other case we can write $\iota = 45312[\alpha, 1, \beta, r_1(\alpha), 1]$ for some $\alpha\in \fS_{k-1}(132)$,  $\beta \in \cI_{n-2k}(132)$ and $1\leq k\leq \floor{\frac{n}{2}}$, which can be seen in Figure~\ref{fig:132}. Let $21[\alpha,r_1(\alpha)]=x$. If $k=1$ then the only ascents in $\iota$ come from $\alpha$, $r_1(\alpha)$ and $\beta$. Particularly, $\comaj(\iota)=\comaj(\beta)+\asc(\beta)$ and $\asc(\iota)=\asc(\beta)$, which gives us the term $M\cI_{n-2}(132;q,qt)$. 
If instead $k>1$ then we have additional ascents at $k-1$ and $n-1$. Using Lemma~\ref{lemma:132majasc} in this case we have 
$$\comaj(\iota)=\comaj(x)+\frac{n-2k+1}{2}\asc(x)+\comaj(\beta)+k\asc(\beta)+n+k-2$$
and
$$\asc(\iota)=\asc(x)+\asc(\beta)+2.$$
This gives us the term in the summation and we are done. 
\end{proof}

\subsection{The pattern 213}
\label{maj213}

We find that $M\cI_n(132)$ and $M\cI_n(213)$ display the  symmetry \\$M\cI_n(132)=q^{\binom{n}{2}}M\cI_n(213;q^{-1})$. However, the reason for this does not come from the map $r_{-1}$  and is also different from the reason for the similar symmetry between the patterns $123$ and $321$, which we discuss in Section~\ref{maj123}. 
The map used to prove this  will be defined in stages, but  as an overview will take $\iota\in \cI_n(213)$ and will map it to $\Des(\iota)$, which is an unique element in 
$$G_n=\{ \{a_1,a_2,\dots,a_{\ell}\}\subseteq[n-1]: a_i<a_{i+1} \text{ and }a_i+a_{\ell -i+1}\geq n\}.$$
This descent set then is mapped to its complement $[n-1]\setminus \Des(\iota)$ that lies in 
$$L_n=\{ \{a_1,a_2,\dots,a_{\ell}\}\subseteq[n-1]: a_i<a_{i+1} \text{ and } a_i+a_{\ell -i+1}\leq n\},$$ 
which is associated to an unique involution in $\cI_n(132)$ with this as its descent set. See Figure~\ref{fig:213<->132} for an example. Though we do prove the bijection from the involutions to their descent sets using induction the full map can not be defined directly using this induction. 

\begin{figure}
$$798456132\in \cI_9(213)\overset{\varphi}{\longleftrightarrow} \{2,3,6,8\}\in G_9 \overset{f}{\longleftrightarrow} \{1,4,5,7\}\in L_9\overset{\psi}{\longleftrightarrow} 867952314\in \cI_9(132)$$
\caption{An example of an involution $\iota\in \cI_n(213)$ that is mapped to $\tau\in \cI_n(132)$  with $\Des(\iota)=[n-1]\setminus \Des(\tau)$. }
\label{fig:213<->132}
\end{figure}

In 
Section~\ref{sec:inv321} 
we introduced the notation Beissinger defined in her paper~\cite{B87} to insert another two-cycle into an involution. We will use this notation to deconstruct an involution $\iota$ as $\iota = \hat{\iota} +(i,n)$ for $i< n$. This will be one key in describing our inductive map with the next lemma giving us the conditions for preserving the avoidance of the pattern 213 as well as describing the resulting descent set.

\begin{lemma} Assume that $\iota = \hat{\iota} +(i,n)$, $i< n$, $\hat{\iota}\in \cI_{n-2}$ and $\iota\in \cI_n$. Let $d=n$ if $ \Des(\hat{\iota})=\emptyset$ and otherwise $d=\min\Des(\hat{\iota})$.  Then $\iota$ avoids $213$ if and only if $i\leq d+1$ and $\hat{\iota}$ avoids 213. Also, 
$$
\Des(\iota)=   \left\{
    \begin{array}{ll}
      \Des(\hat{\iota})+1 & i=d+1  \\
      (\Des(\hat{\iota})+1)\cup\{i,n-1\}& i<d+1.
    \end{array}\right.
$$
\label{lemma:213Des}
\end{lemma}

\begin{proof} 
Let $\iota = \hat{\iota} +(i,n)$, $i< n$, $\hat{\iota} \in \cI_{n-2}$, and $\iota\in \cI_n$. First consider the case where $\Des(\hat{\iota})=\emptyset$ or equivalently $\hat{\iota} = 12\dots (n-2)$. In this case we let $d=n$ and for all $i<n$ it is easy to see that $\hat{\iota}+(i,n)$ avoids $213$ and has descent set $\{i,n-1\}$. The second part of the lemma holds as well in this case since $i<d+1$. 

In any other case $\Des(\hat{\iota})\neq \emptyset$ and we let $d=\min \Des(\hat{\iota})$. First we will assume that $\iota$ avoids $213$. It follows that  $\hat{\iota}$, a subword of $\iota$, must also avoid $213$. Since $\iota$ avoids $213$ we know  $\iota$  increases before $n$ at index $i$, which implies $i\leq d+1$. Instead assume that $\hat{\iota}$ avoids $213$ and $i\leq d+1$. If $\iota$ were to have the pattern $213$ then either $n$ or $i$ must be part of the pattern since $\hat{\iota}$ avoids $213$. The only possibility is that $n$ plays the roll of $3$ and $\iota$ has a descent before index $i$, which is impossible because this descent would come from a descent in $\hat{\iota}$ and we assumed that $i\leq d+1$. 

Finally, we will finish by showing the second part of the lemma in the case where $\Des(\hat{\iota})\neq \emptyset$   by determining the descent set of $\iota=\hat{\iota}+(i,n)$ from the descent set of $\hat{\iota}$. The general  descent set of $\iota$ is the union of $\Des(\hat{\iota})\cap [1,i-2]$, $\{i\}$ and $(\Des(\hat{\iota})\cap [i,n-3])+1$ with additionally $\{n-1\}$ if $\hat{\iota}(n-2)\geq i$. Because $i\leq d+1$ we must have $\Des(\hat{\iota})\cap [1,i-2]=\emptyset$. Consider the case where $i=d+1$. Because $\hat{\iota}$ avoids $213$ we must have no descent before the occurrence of $n-2$ in $\hat{\iota}$ so $\hat{\iota}(n-2)=d$. Since $i=d+1$ we have $\hat{\iota}(n-2)<i$ so $n-1\notin\Des(\iota)$ and $\Des(\iota)=\{i\}\cup((\Des(\hat{\iota})\cap [i,n-3])+1)=\Des(\hat{\iota})+1$. 
If instead $i<d+1$ we still have $\hat\iota(n-2)=d$ but now $\hat{\iota}(n-2)\geq i$ so $\Des(\iota)=\{i,n-1\}\cup ((\Des(\hat{\iota})\cap [i,n-3])+1)=(\Des(\hat{\iota})+1)\cup \{i,n-1\}$. With this we are done. 
\end{proof}

We next proceed through some technical lemmas that will step by step prove our bijection $\cI_n(213)\rightarrow\cI_n(132)$.

\begin{lemma} We have a bijection $\varphi:\cI_n(213)\rightarrow G_n$ such that $\varphi(\iota)=\Des(\iota)$. 
\label{lemma:varphi}
\end{lemma}

\begin{proof} Let $\iota \in \cI_n(213)$. We first show that $\varphi$ is well defined, that is $\Des(\iota)\in G_n$, which we will show using induction. It is not hard to see this in the case of $n=1$ or $2$. We now assume that $n>2$ and all $\hat{\iota}\in \cI_k(213)$ have $\Des(\hat{\iota})\in G_k$ for any $k<n$.

 If $n$ is a fixed point then we must have that $\iota = 12\dots n$ because $\iota$ avoids $213$ so then $\Des(\iota)=\emptyset \in G_n$. 
Otherwise we have  by Lemma~\ref{lemma:213Des} that $\iota = \hat{\iota}+(i,n)$ for some $i\leq d+1$ where $d=\min \Des(\hat{\iota})$ if $\Des(\hat{\iota})\neq \emptyset$ and otherwise $d = n$. By our inductive assumption $\Des(\hat{\iota})=\{a_1,\dots, a_k\}\in G_{n-2}$. By Lemma~\ref{lemma:213Des} if $i=d+1$ then $\Des(\iota)=\Des(\hat{\iota})+1$ and it follows $\Des(\iota)\in G_n$. If instead $i<d+1$ by Lemma~\ref{lemma:213Des} we have $\Des(\iota)=\{i,a_1+1,\dots a_k+1,n-1\}$ that again implies $\Des(\iota)\in G_n$. 

Next we define the inverse map $\varphi^{-1}:G_n\rightarrow \cI_n(213)$  inductively. Let $A\in G_n$. We define $\varphi^{-1}(\emptyset)=12\dots n$ and otherwise for $A\neq \emptyset$
\begin{equation}
   \hat{A} = \left\{
     \begin{array}{ll}
       (A\setminus \{\min A,n-1\})-1 & n-1\in A,\\
       A-1 & n-1\notin A.
     \end{array}
   \right.
   \label{eq:setB}
\end{equation} 
so 
$$\varphi^{-1}(A)=\varphi^{-1}(\hat{A})+(\min A,n).$$
This is well defined because $\hat{A}\in G_{n-2}$ and    $\varphi^{-1}(A)$ avoids $213$  since $\min A\leq \min \hat{A}+1$ by Lemma~\ref{lemma:213Des}. 

We lastly need to show that these two maps are indeed inverses. The cases of $n=1$ or $2$ are easy, so we can assume that $n>2$ and that $\varphi$ is a bijection $\cI_k(213)\rightarrow G_k$ for $k<n$. Let $A\in G_n$. If $A=\emptyset$ then $\varphi \circ \varphi^{-1}(A)=\varphi(12\dots n)=A$. In any other case we define $\hat{A}$ as in equation~\eqref{eq:setB} and we have $\varphi \circ \varphi^{-1}(A)=\Des(\varphi^{-1}(\hat{A})+(\min A,n))$. By induction $\Des(\varphi^{-1}(\hat{A}))=\hat{A}$. Consider the case where $n-1\notin A$ then we defined $\hat{A}=A-1$ so $\min A=\min \hat{A}+1$ and by Lemma~\ref{lemma:213Des} this implies $\Des(\hat{A}+(\min A,n))=\hat{A}+1=A$ and we are done. Otherwise $n-1\in A$ and $\hat{A}=(A\setminus\{\min A,n-1\})-1$, which implies $\min A <\min \hat{A}+1$ so by Lemma~\ref{lemma:213Des} we have $\Des(\varphi^{-1}(\hat{A})+(\min A,n))=(\hat{A}+1)\cup \{i,n-1\}=A$ and we are done. 

For the other direction we need to show  for $\iota \in \cI_n(213)$ that $\varphi^{-1}\circ\varphi(\iota)=\iota$. If $n$ is a fixed point of $\iota$ then $\iota = 12\dots n$ because $\iota$ avoids $213$. Then $\varphi^{-1}(\varphi(12\dots n))=\varphi^{-1}(\emptyset)=12\dots n$. We will now assume that $n$ is not a fixed  point and $\iota = \hat{\iota}+ (i,n)$ for $i<n$ and $\hat{\iota}\in \cI_{n-2}(213)$. We next consider the set $\varphi(\iota)=\Des(\iota)=A$ and its associated $\hat{A}$ set determined by equation~\eqref{eq:setB}. Note that $i=\min A$ because $\iota$ avoids $213$ and if we have $\hat{A}=\Des(\hat{\iota})$ then we  have $\varphi^{-1}(\hat{A})=\hat{\iota}$ by induction so $\varphi^{-1}(A)=\hat{\iota}+(\min A,n)=\iota$. So all we have to show is that $\Des(\hat{\iota})=\hat{A}$. Consider the case where  $i= \min \Des(\hat{\iota})+1$ so then 
by Lemma~\ref{lemma:213Des} we know $\Des(\iota)=\Des(\hat{\iota})+1$, which implies that $n-1\notin \Des(\iota)$ so by equation~\eqref{eq:setB} $\hat{A}=\Des(\hat{\iota})$. In the other case  $i<\min \Des(\hat{\iota})+1$ so we have $\Des(\iota)=(\Des(\hat{\iota})+1)\cup \{i,n-1\}$ and $n-1\in \Des(\iota)$. Also in this case  $\hat{A}=\Des(\hat{\iota})$, hence, $\varphi$ is a bijection. 
\end{proof}

We have a similar lemma describing the conditions for when an involution $\hat{\iota} +(1,i)$ avoids 132 and its resulting descent set.

\begin{lemma} Assume that $\iota = \hat{\iota} +(1,i)$, $i>1$, $\hat{\iota} \in \cI_{n-2}$ and $\iota\in \cI_n$. Let $d=-1$ if $ \Des(\hat{\iota})=\emptyset$ and otherwise $d=\max\Des(\hat{\iota})$.  Then $\iota$ avoids $132$ if and only if $i\geq d+2$ and $\hat{\iota}$ avoids 132. Also, 
$$
\Des(\iota)=   \left\{
    \begin{array}{ll}
      \Des(\hat{\iota})+1 & i=d+2  \\
      (\Des(\hat{\iota})+1)\cup\{1,i-1\}& i>d+2.\hfill
    \end{array}\right.
$$
\vspace{-1cm}

\hqed
\end{lemma}

We exclude the proof because it is similar to the proof of Lemma~\ref{lemma:213Des}. There is also a  map $\cI_n(132)\rightarrow L_n$ similar to $\varphi:\cI_n(213)\rightarrow G_n$. 

\begin{lemma}
There is a bijection $\psi:\cI_n(132)\rightarrow L_n$ where $\iota$ is sent to $\Des(\iota)$. \hqed
\label{lemma:psi}
\end{lemma}

The proof is similar to the proof of Lemma~\ref{lemma:varphi}, so it is left out. The last piece of the bijection we need is the one between the sets $G_n$ and $L_n$. 
 In this next lemma for sets $A=\{a_1,a_2,\dots ,a_k\}$ and $B=\{b_1,b_2,\dots ,b_k\}$, where we write the elements in increasing order,  define  $A\leq B$ if $a_i\leq b_i$ for all $i$. This relation is only for sets with equal cardinality.  

\begin{lemma}The map $f:G_n\rightarrow L_n$ defined by $f(A)=[n-1]\setminus A$ is a bijection. 
\label{lemma:f}
\end{lemma}

\begin{proof}
This map has an inverse, which is itself, so our work in this proof will be to show that this map is well defined. For this proof we will let $A=\{a_1,a_2,\dots, a_k\}\subseteq [n-1]$ where the elements are written in increasing order and $S=\{s_1,s_2,\dots, s_k\}\subseteq[n-1]$ where $s_i=n-a_{k-i+1}$ so that the elements of $S$ are also in increasing order. Similarly, we will let $B=[n-1]\setminus A=\{b_1,b_2,\dots, b_{n-k-1}\}$ with $b_i<b_{i+1}$ and $T=\{t_1,t_2,\dots, t_{n-k-1}\}\subseteq[n-1]$ where $t_i=n-b_{n-k-i}$ so that the elements of $T$ increase. Our goal is to show $A\in G_n$ if and only if $B\in L_n$. 

First we will note that $S=[n-1]\setminus T$. Secondly we will note that $a_{k-i+1}+s_i=n$ so $a_i+a_{k-i+1}\geq n$ is equivalent to $a_i\geq s_i$. Thus, $A\in G_n$ if and only if $A\geq S$. Similarly $B\in L_n$ if and only if $B\leq T$. 

We will next argue that if $S\leq A$ then $[n-1]\setminus S\geq [n-1]\setminus A$. We will argue this  by inducting on the cardinality $|S|=|A|=k$. The base case is $k=0$ where $S=A=\emptyset$ where by vacuum we have $S\leq A$ and certainly $[n-1]\setminus \emptyset \geq [n-1]\setminus \emptyset$. Otherwise $S$ and $A$ have minimum elements $s_1\leq a_1$ respectfully. It is not hard to see that $S\setminus \{s_1\}\leq A\setminus \{a_1\}$ so by induction $[n-1]\setminus(S\setminus \{s_1\})\geq [n-1]\setminus(A\setminus \{a_1\})$ or equivalently $([n-1]\setminus S)\cup \{s_1\}\geq ([n-1]\setminus A)\cup \{a_1\}$. Generally it is not hard to see for sets $|U|=|V|$ that if  $U\leq V$ with $u\in U$, $v\in V$ and $v\leq u$ that $U\setminus\{u\}\leq V\setminus \{v\}$. From this since $s_1\leq a_1$ we have $[n-1]\setminus S \geq [n-1]\setminus A$. 

Putting everything together $A\in G_n$ if and only if $S\leq A$ if and only if $[n-1]\setminus S \geq [n-1]\setminus A$. This is equivalently $T\geq B$, which is true if and only if $B\in L_n$. 
\end{proof}

Now putting all our maps  together we get a bijection from $\cI_n(213)$ to $\cI_n(132)$ such that if $\iota$ is mapped to $\tau$ we have that $\Des(\iota)=[n-1]\setminus \Des(\tau)=\Asc(\tau)$. 

\begin{thm} For $n\geq 0$, 
$$M{\mathcal I}_n(213) = q^{\binom{n}{2}}M{\mathcal I}_n(132;q^{-1}).$$
\label{thm:132symm213}
\end{thm}

\begin{proof}
This equality is true because there is a bijection $\psi^{-1}\circ f\circ\varphi:\cI_n(213)\rightarrow \cI_n(132)$ using Lemmas~\ref{lemma:varphi},~\ref{lemma:psi} and~\ref{lemma:f}. Further, if $\iota \in \cI_n(213)$ then $\varphi(\iota)=\Des(\iota)$, $f(\Des(\iota))=[n-1]\setminus\Des(\iota)$ and $\psi^{-1}([n-1]\setminus \Des(\iota))=\tau$ so $\Des(\tau)=[n-1]\setminus \Des(\iota)$. 
\end{proof}

As a corollary we can prove that $M\cI_n(213)$ has similar but symmetric internal zeros as were found in $M\cI_n(132)$, which was proven in Proposition~\ref{thm:132internalzeros}. 

\begin{cor} If $\iota \in {\cI}_n(213)$ then 
\begin{enumerate}[(i)]
\item $\maj(\iota)=0$ or $\maj(\iota) \geq {\lceil n/2 \rceil}$ 
\item this bound is sharp and
\item for every $k\geq {\lceil n/2 \rceil}$, $k\leq \binom{n}{2}$ there exists some $\iota \in {\cI}_n(213)$ with $\maj(\iota)=k$. 
\end{enumerate}
\end{cor}
\begin{proof}Using Proposition~\ref{thm:132internalzeros} and  the map in Theorem~\ref{thm:132symm213}   we quickly get this result. 
\end{proof}

\subsection{The pattern 321}
\label{subsec321}

 In this section we will show another interpretation for the standard $q$-analogue of the binomial coefficient defined in equation~\eqref{eq:qbinom}.
It turns out that $M\cI_n(321)$ is equal
to the standard $q$-analogue for the central binomial coefficient. This result was proven independently by Barnabei et.\ al.~\cite{BBES14} (Theorem 3.3) whose proof gives a connection to hook decompositions. Our proof has the advantage of being shorter and gives a connection to the concept of a core. The  core is a concept due to Greene and Kleitman~\cite{GK76} (page 82) that originated in the study of posets. It has  traditionally been used to  prove that a poset has a symmetric-chain decomposition, but our use of it is new and quite different from the original. Also, our proof can be easily generalized to give another interpretation for the general $q$-analogue for the binomial coefficient, not just the central one, which we present at the end of this section in Corollary~\ref{cor:321t<=k}. This result  also appears in~\cite{BBES16} (Corollary 14) by  Barnabei et.\ al. Parts of the bijection we present  can be seen in~\cite{BBES16} and~\cite{EFPT15} in their  association between involutions avoiding 321 and Dyke paths.

Given a length $n$ word composed of left parentheses and right parentheses the core is a subsequence of the word and is defined inductively. To find the core we begin by matching a left parenthesis with right parenthesis if they are adjacent and the left parenthesis is on the left. Excluding all previously matched parentheses we continue to match more in a similar matter until there are no more possible matchings. The subsequence that contains all of the matched parentheses is called the {\it core}. We say that a specific parenthesis is in the core if that parenthesis is part of a matching. Similarly, we will say an index $i$ is part of the core if the parenthesis at index $i$ is part of the core. For example the word $(()()))((()($ has the core $(()())()$ and the indices $\{1, 2, 3, 4, 5, 6, 10, 11\}$ are in the core.

Given a binary word of $0$s and $1$s we can similarly define its core. Consider all $1$s to be left parentheses and all $0$s to be right parentheses. With this we equate the word $(()()))((()($  to $11010011101$ and its core is $11010010$, which still occurs on the indices $\{1,2,3,4,5,6,10,11\}$. Note that the core itself is a perfect matching whose index set inside the word can be broken down uniquely into disjoint intervals with the following properties. The first property is that the subsequence of the core associated to any one of the intervals is itself a perfect matching. The second is that no interval can be broken into two intervals that both satisfy the first  property. We will call each interval, or the subsequence of the core associated to that interval, a {\it block}. In our example we have two blocks that are $110100$ and $10$.

The next two lemmas will establish some basic facts the core and about descents in involutions avoiding 321. 

\begin{lemma}
Let $\iota \in \cI_n(321)$.
 If $d\in \Des(\iota)$ then $\iota$ has the two-cycle $(d,d+1)$ or two distinct two-cycles $(d,t)$ and $(s,d+1)$ such that $d<t$ and $s<d+1$.

 \label{lemma:321Des}
\end{lemma}

\begin{proof}
Let $d\in \Des(\iota)$, which implies that $(d,d+1)$ is an inversion. By Lemma~\ref{lemma:321structure} we then have two-cycles $(d,t)$ with $d<t$ and $(s,d+1)$ with $s<d+1$ that may not be distinct. If the two-cycles are distinct we are done, and if they are not then we have the two-cycle $(d,d+1)$.
\end{proof}

\begin{lemma}
Let $w$ be a binary word. 
\begin{enumerate}[(i)]
\item The subword of $w$ composed of all elements not in the core is weakly increasing. 
\item If the $i$th $1$ inside the core occurs at index $s$ in $w$ and the $i$th $0$ occurs at $t$ in $w$ then $s<t$ and all indices in the interval $[s,t]$ are inside the core. 
\end{enumerate}
\label{lemma:binarystructure}
\end{lemma}

\begin{proof}
Say that the subword of $w$ composed of all elements outside the core is not weakly increasing, which means there is some strict decrease. For a binary word to have a strict decrease it would need a $1$ directly followed by a $0$. By the inductive construction of the core, these two letters would be matched and be inside the core, which is a contradiction.

Each block of the core has an equal number of $1$s and $0$s since it is a perfect matching. As result, the first block will have the $1$st through $m_1$th $1$ and $0$, The second block will have the $(m_1+1)$st through $m_2$th $1$ and $0$, and so generally the $i$th block will have the $(m_{i-1}+1)$st through $m_i$th $1$ and $0$. This means that the $j$th $1$ and the $j$th $0$ will always be in the same block. Since blocks occur on consecutive indices all letters between the $j$th $1$ and the $j$th $0$ are in the core. \end{proof}

Now we are ready to delve into the main topic of this section, that $M{\mathcal I}_n(321)$ is a standard $q$-analogue for the central binomial coefficient. Our method of proof is to show that there is a bijection from $M{\mathcal I}_n(321)$ to another combinatorial object that has been well established to have a generating function equal to standard $q$-analogues for binomial coefficient, see~\cite{S97} (exercise 1.56).

\begin{prop}If $W_{n, k}$ is the set of binary words of length $n$ with 
$n-k$ zeros and $k$ ones then 
$${n \brack k}_q = \sum_{w\in W_{n,k}} q^{\maj(w)}.$$
\label{prop:qbinom}
\vspace{-1cm}

\hqed
\end{prop}

In proving the next theorem we establish a bijection from involutions avoiding 321 to binary words that preserves the descent set using the concept of core. 

\begin{thm} 
For $n\geq 0$ we have the following equality of $q$-analogues,
$$M{\mathcal I}_n(321)={n \brack {\ceil{n/2}}}_q.$$
 \label{theorem:321qanalougeequiv}
\end{thm}

\begin{proof}
To prove the equality we will use a well-known interpretation of the $q$-analogue for binomial coefficients  stated in Proposition~\ref{prop:qbinom}. We will construct a bijection $\phi:\cI_n(321)\rightarrow W_{n,\ceil{n/2}}$ that preserves the decent set. Preserving the decent set will preserve the major index, which will give us the equalities 
 
 $$M{\mathcal I}_n(321)=  \sum_{\iota \in {\mathcal I}_n(321)} q^{\maj (\iota)}= \sum_{w\in W_{n,\ceil{ n/2}}} q^{\maj(w)}={n \brack {\ceil{ n/2}}}_q.$$

Let $\iota \in \cI_n(321)$ have two-cycles  $(s_1,t_1), (s_2,t_2),..., (s_m,t_m)$ such that $s_i<t_i$ and $s_i<s_{i+1}$. Also, let the fixed points be $f_1,f_2,\dots f_{n-2m}$ such that $f_i<f_{i+1}$ for all $i$. We want to define a binary word $\phi(\iota)=w=w_1\dots w_n$  that has ${\lceil n/2 \rceil}$ ones and ${\lfloor n/2 \rfloor}$ zeros. Note that $2m\leq n$ so $m\leq \floor{n/2}$, which means $a_1=\ceil{n/2}-m \geq 0$  and $a_0=\floor{n/2}-m\geq 0$. We define $w$ to be the binary word with $w_{s_i}=1$, $w_{t_i}=0$ and we replace the remaining letters that form the subword $f_1f_2\dots f_{n-2m}$ with $0^{a_0}1^{a_1}$ where $i^j$ is the word of $j$ consecutive $i$'s. We can easily see that $w$ has ${\ceil{ n/2 }}$ ones and ${\floor{n/2}}$ zeros, so $\phi$ is well defined. For example $\phi(132458967)=010111100$.

Say that we have a binary word $w\in W_{n,\ceil{n/2}}$. 
Let $s_i$ be the index in $w$ at which the $i$th $1$ in the core appears and $t_i$ be the index in $w$ at which  the $i$th $0$ in the core appears. Note that $s_i<s_{i+1}$ since the $i$th 1 is before the $(i+1)$st $1$. Similarly $t_i<t_{i+1}$. We define the involution $\phi^{-1}(w)=\iota$ to be 
  $$ \iota(j) = \left\{
     \begin{array}{lr}
       t_i & j=s_i,\\
       s_i & j=t_i,\\
     j&\text{else}.
     \end{array}
   \right.$$
We can easily see since all $s_i$'s and $t_i$'s are distinct so $\iota$ has the two-cycles $(s_i,t_i)$ and everything else is a fixed point. This means that $\iota$ is an involution. Say $f_1,f_2,\dots, f_{n-2m}$ are the fixed points listed so that $f_i<f_{i+1}$. 

We will now show that this involution avoids $321$. 
Consider the subword $x$ of $\iota$ that occurs at the indices $\{s_1,\dots, s_m,f_1,\dots f_{n-2m}\}$. We will show that this subword is increasing by showing that it doesn't have any inversions. Since $\iota(s_i)<\iota(s_{i+1})$ and $f_i<f_{i+1}$ any inversion will have to occur between a pair of indices $s_i$ and $f_j$. By Lemma~\ref{lemma:binarystructure} since $s_i<t_i$ we know that all indices in the interval $[s_i,t_i]$ are in the core. Say $s_i<f_j$. Since index $f_j$ is not in the core and all indices in $[s_i,t_i]$ are in the core we must have that $\iota(s_i)=t_i<f_j$. Similarly if $f_j<s_i$ then $f_j<t_i$. This shows that $x$ doesn't have any inversions and is increasing. 
The subword $\iota(t_1)\iota(t_2)\dots \iota(t_m)$ is also  increasing. This means that $\iota$ is composed of two disjoint increasing subsequences, so the longest decreasing subsequence has length at most two. From this we can conclude that $\iota$ avoids $321$ and $\phi^{-1}$ is well defined.

Next, we will show that the two maps are inverses. It suffices to show that $\phi(\phi^{-1}(w))=w$ since $|\cI_n(321)|=|W_{n,\ceil{n/2}}|$. Let $w\in W_{n,\ceil{n/2}}$, $\phi^{-1}(w)=\iota$ and $\phi(\iota)=v$. If $w_i$ is the $r$th $1$ in the core of $w$ then $\iota$ has a two-cycle $(i,j)$ where $w_j$ is the $r$th $0$ in the core of $w$.
 Since $w_{j}$ is the index of the $r$th $0$ in the core and according to Lemma~\ref{lemma:binarystructure} the $r$th $1$ occurs before the $r$th $0$, we have that $i<j$. This means that $v_i=1$ and $v_j=0$, so for all indices $i$ inside the core of $w$ we have that $w_i=v_i$. 
Say the core of $w$ has $2m$ elements then $\iota$ has $m$ two-cycles. Let $a_1=\ceil{n/2}-m$ and $a_0=\floor{n/2}-m$. By definition of $\phi$, the  subword of $v$ corresponding to fixed points of $\iota$, or equivalently the indices outside the core of $w$, is $0^{a_0}1^{a_1}$. 
Note that the subword of $w$ composed of letters outside the core is made of $a_1$ ones and $a_0$ zeros. By Lemma~\ref{lemma:binarystructure} this subword is weakly increasing and so must equal $0^{a_0}1^{a_1}$. Hence $v_i=w_i$ for all indices $i$ outside the core of $w$ so $w=v$.  

Lastly, we will show that if $\phi(\iota)=w$ then $\Des(\iota)=\Des(w)$. First we will make a quick note about $\phi^{-1}$. If $w_i=1$ is the $r$th one in the core of $w$, then the $r$th zero occurs at $w_j$ for some $i<j$. This means that the corresponding involution $\iota$ has the two-cycle $(i,j)$ with $i<j=\iota(i)$. Similarly, if $w_j$ is the $r$th $0$ in the core then $\iota$ has the two-cycle $(i,j)$  with $i =\iota(j)<j$. Say $d\in \Des(w)$ then $w_d=1$ and $w_{d+1}=0$, which implies that both these indices are in the core. 
From the map $\phi^{-1}$ since $w_d=1$ is in the core $\iota$ must  have a two-cycle $(d,\iota(d))$ with $d<\iota(d)$. 
Similarly, $\iota$ must have a two-cycle $(\iota(d+1),d+1)$ with $\iota(d+1)<d+1$. It is possible for these two-cycles to be the same, but in either case this implies that $\iota(d)>\iota(d+1)$ and $d\in \Des(\iota)$. 
Conversely consider $d\in \Des(\iota)$. According to Lemma~\ref{lemma:321structure} we must have the two-cycle $(d,d+1)$ or a pair of  two-cycles $(d,\iota(d))$ and $(\iota(d+1),d+1)$ with $\iota(d+1)<d+1$ and $d<\iota(d)$. In either case this implies that $w_d=1$, $w_{d+1}=0$ and $d\in \Des(w)$. Hence $\Des(w)=\Des(\iota)$. 
\end{proof}

By slightly modifying the proof from Theorem~\ref{theorem:321qanalougeequiv} we derive another interpretation for the standard $q$-analogue for the binomial coefficients. 
This result  also appears in~\cite{BBES16}  by  Barnabei et.\ al.

\begin{cor}[Barnabei et.\ al.~\cite{BBES16} {Corollary 14}]
Let $t(\iota)$ be the number of two-cycles in $\iota$ and $k\leq n/2$. Then we have the following equality of $q$-analogues,
$$ \underset{t(\iota)\leq k}{\sum_{\iota \in \cI_n(321)}} q^{\maj(\iota)}= {n\brack k}_q.$$
\label{cor:321t<=k}
\end{cor}

\begin{proof} This proof will be similar to the proof of Theorem~\ref{theorem:321qanalougeequiv}. The bijection will instead be defined from length $n$ binary words with  $k\leq n/2$ ones and $n-k$ zeros to involutions in $\cI_n(321)$ that have at most $k$ two-cycles. The map and its inverse will be defined exactly the same except for a small modification in $\phi$ where we alter the number of ones and zeros we want in our binary word. The bijection will be well defined since the changed number of ones and zeros will bound the maximum number of possible two-cycles. 
\end{proof}

Barnabei et.\ al.~\cite{BBES16} used Corollary~\ref{cor:321t<=k} to describe the generating function for $\maj$ and involutions avoiding 321 where the number of two-cycles is fixed. 

\begin{cor}[Barnabei et.\ al.~\cite{BBES16} {Corollary 14}] For $n\geq 1$ and $k\leq n/2$ we have
$$\underset{t(\iota)= k}{\sum_{\iota\in\cI_{n}(321)}}q^{\maj(\iota)}={n\brack k}_q - {n\brack k-1}_q.$$
\label{cor:321t=k}
\end{cor}
\begin{proof}Using Corollary~\ref{cor:321t<=k} we have the series of equalities
$$\underset{t(\iota)= k}{\sum_{\iota\in\cI_{n}(321)}}q^{\maj(\iota)}=\underset{t(\iota)\leq k}{\sum_{\iota\in\cI_{n}(321)}}q^{\maj(\iota)}-\underset{t(\iota)\leq k-1}{\sum_{\iota\in\cI_{n}(321)}}q^{\maj(\iota)}={n\brack k}_q - {n\brack k-1}_q,$$
which finishes the proof.
\end{proof}

Though fixed-point-free involutions were used to determine $I\cI_n(321)$ in Section~\ref{sec:inv321} we see in this section that they are not required to determine $M\cI_n(321)$. However, using the previous corollary we can ascertain the generating function in the fixed-point-free case. 

\begin{cor} For $n=2m\geq 2$  we have
$$MF\cI_{2m}(321)={2m\brack m}_q - {2m\brack m-1}_q.$$
\vspace{-1cm}

\hqed
\vspace{.5cm}
\label{cor:321fpf}
\end{cor}

\subsection{The pattern 123}
\label{maj123}

There is  a similar symmetry regarding $M\cI_n(123)$ and $M\cI_n(321)$ as we found for the patterns 132 and 213. This symmetry is not present when restricting to fixed-point-free involutions because  the enumerations for $F\cI_{2m}(321)$ and $F\cI_{2m}(123)$ have been shown to be different by Deutsch, Robertson, and Saracino in~\cite{DRS07} who enumerated the avoidance classes by number of fixed points. They found $|F\cI_{2m}(123)|=\binom{2m-1}{m}$ but $|F\cI_{2m}(321)|=C_m$ in~\cite{DRS07} (Theorem 2.1).

It has been shown by many including Simion and Schmidt~\cite{ss:rp}, Barnabei et.\ al.~\cite{BBES14} and Deutsch et.\ al.~\cite{DRS07} that there is a symmetry between $M\cI_n(123)$ and $M\cI_n(321)$, specifically $M\cI_n(123)=q^{\binom{n}{2}}M\cI_n(321;q^{-1})$, which is shown  again here using the RSK correspondence and tableau transposition. Another detailing of this map by B\'{o}na and Smith can be found in~\cite{BS16} (Section 3) whose description subverts the RSK algorithm and transposition. This symmetry is essential in determining the form of $M\cI_n(123)$ since we have established that $M\cI_n(321)$ is the standard $q$-analogue for the central binomial coefficient in Theorem~\ref{theorem:321qanalougeequiv}.

\begin{prop}
For $n\geq 0$,  
$$M{\mathcal I}_n(123) =q^{\binom{n}{2}} M{\mathcal I}_n(321;q^{-1}).$$ 
\label{thm:123&321 symmetry}
\end{prop}
The well-used proof is a bijection between involutions using SYT, which were defined in Section~\ref{sec:inv321}. There are two facts that we will need to recall from Proposition~\ref{SYTfacts}. The first is that the length of the longest decreasing sequence in a permutation is equal to the length of the first column in its SYT, and the length of the longest increasing sequence in a permutation is equal to the length of the first row in its SYT. The second is that $d$ is a descent of an involution $\iota$ if and only if $d+1$ appears in a lower row than $d$ in the associated SYT.

\begin{proof}[Proof of Theorem~\ref{thm:123&321 symmetry}]
It suffices to define a map from  $\cI_n(321)$ to $\cI_n(123)$ such that $\iota$ is mapped to an involution with descent set $[n-1]\setminus \Des(\iota)$. 
This is sufficient because then $\iota$ will be mapped to an involution with $\maj$ equal to $\binom{n}{2}-\maj(\iota)$.

The set $\cI_n(321)$ contains involutions with longest decreasing sequences of length one or two. Similarly, the set $\cI_n(123)$ contains involutions with longest increasing sequences of length one or two. 
So the collection of SYT associated to  $\cI_n(321)$ is all SYT of size $n$ with at most two rows, and the collection of SYT associated to  $\cI_n(123)$ is all SYT of size $n$ with at most two columns. Note that the transpose of a SYT with at most two columns is a SYT with at most two rows. So if we use RSK correspondence on $\iota \in \cI_n(321)$ to get a SYT $P$,  transpose the SYT to get $P^T$ and then using the  RSK correspondence  on $P^T$ to get another involution $\tau\in \cI_n(123)$ we have defined a well-defined bijection  from $\cI_n(321)$ to $\cI_n(123)$. This map  is illustrated in Figure~\ref{fig:321to123}.

Let $\iota \in \cI_n(321)$, $P$ be its SYT and $\iota\mapsto\tau$ by the map described in the previous paragraph. We will show that $i\in \Des(\iota)$ if and only if $i\notin \Des(\tau)$, which will imply that $\Des(\tau)=[n-1]\setminus \Des(\iota)$.

 It is known that if $i\in \Des(\iota)$ then $i+1$ is in a row below $i$ in $P$. Because rows and columns strictly increase this further implies that  $i+1$ is in is the same column as $i$ or to the right of $i$ in $P$. Hence in $P^T$ we have $i+1$ in the same row as $i$ or in a row above $i$ and thus $i\notin\Des(\tau)$. For a very similar reason if $i\notin \Des(\iota)$ then $i\in \Des(\tau)$, which completes the proof. 
\end{proof}

\begin{figure}
\begin{align*}
\begin{tikzpicture} 
\draw (0,0) node {$\cI_n(321)$};
\draw (0,1) node {3516247};
\draw (2,.5) node {$\longrightarrow$};
 \end{tikzpicture}
 &&
 \begin{tikzpicture} 
 \draw (0,2) --  (2,2);
 \draw (0,1.5) --  (2,1.5);
  \draw (0,1) --  (1.5,1);
  \draw (0,1) --  (0,2);
\draw (.5,1) --  (.5,2);
\draw (1,1) --  (1,2);
\draw (1.5,1) --  (1.5,2);
\draw (2,1.5) --  (2,2);
  \draw (1,0) node {$P$};
\draw (.25,1.75) node {$1$};
\draw (.75,1.75) node {$2$};
\draw (1.25,1.75) node {$4$};
\draw (1.75,1.75) node {$7$};
\draw (.25,1.25) node {$3$};
\draw (.75,1.25) node {$5$};
\draw (1.25,1.25) node {$6$};
\draw (3,.5) node {$\longrightarrow$};
 \end{tikzpicture} 
  &&
 \begin{tikzpicture} 
 \draw (0,3) --  (1,3);
  \draw (0,2.5) --  (1,2.5);
 \draw (0,2) --  (1,2);
 \draw (0,1.5) --  (1,1.5);
  \draw (0,1) --  (.5,1);
  \draw (0,1) --  (0,3);
\draw (.5,1) --  (.5,3);
\draw (1,1.5) --  (1,3);
  \draw (1,.5) node {$P^{T}$};
\draw (.25,2.75) node {$1$};
\draw (.25,2.25) node {$2$};
\draw (.25,1.75) node {$4$};
\draw (.25,1.25) node {$7$};
\draw (.75,2.75) node {$3$};
\draw (.75,2.25) node {$5$};
\draw (.75,1.75) node {$6$};
 \end{tikzpicture} 
 &&\begin{tikzpicture} 
\draw (0,0) node {$\cI_n(123)$};
\draw (0,1) node {4271653};
\draw (-2,.5) node {$\longrightarrow$};
 \end{tikzpicture}
\end{align*}
  \caption{Illustration of $\phi:\cI_n(321)\rightarrow \cI_n(123)$ with $\Des(\phi(\iota))=[n-1]\setminus \Des(\iota)$.}
 \label{fig:321to123}
\end{figure}
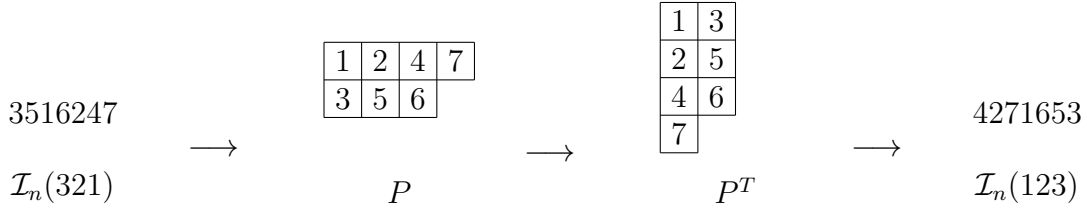

Automatically by Theorem~\ref{theorem:321qanalougeequiv} and Theorem~\ref{thm:123&321 symmetry} we know $M\cI_n(123)$. 

\begin{cor}[Barnabei et.\ al.~\cite{BBES14} Corollary 4.3]
We have for $n\geq 0$ 
$$M\cI_n(123)=q^{\binom{n}{2}} {{n}\brack{\ceil{n/2}}}_{q^{-1}}.$$
\vspace{-1cm}

\hqed
\vspace{.2cm}
\end{cor}

For every pattern excluding 231 we have considered the generating function in the fixed-point-free case. We do so here for the pattern 123. We rely on the result in Corollary~\ref{cor:321t=k}. 
\begin{cor}We have for $m\geq 1$, 
$$\ol{MF\cI}_{2m}(123)=\sum_{k=0}^{2\floor{m/2}}(-1)^{k}{2m\brack k}_{q}.$$
\end{cor}
\begin{proof}
 Proposition~\ref{SYTfacts} says if $\iota$ avoids 123  then the SYT $P$ of $\iota$ has one or two columns. Also, if $\iota$ is fixed-point-free then both of these columns have even length and $|\iota|=2m$ for some $m$. The $\tau\in\cI_{2m}(321)$ associated to $P^T$ must then have at most two rows of even length. This is equivalent to $\tau$ avoiding 321 and having an even number of two-cycles. Since $\comaj(\iota)=\maj(\tau)$ we have the equality 
$$\sum_{\iota\in F\cI_{2m}(123)}q^{\comaj(\iota)}=\underset{t(\iota) \text{ is even}}{\sum_{\iota\in\cI_{2m}(321)}}q^{\maj(\iota)}.$$
If we have an even $2j$ two-cycles with $2\leq 2j\leq m$ by Corollary~\ref{cor:321t=k} we have
$$\underset{t(\iota)=2j}{\sum_{\iota\in\cI_{2m}(321)}}q^{\maj(\iota)}={2m\brack 2j}_q - {2m\brack 2j-1}_q.$$
summing over all  $1\leq j\leq \floor{m/2}$ and including the identity at $j=0$ gives us the result. 
\end{proof}


\section{Multiple patterns}
\label{multi}
In this section we consider   $\cI_n(\pi_1,\pi_2,\dots, \pi_j)=\cI_n(S)$ the set of all involutions $\iota \in \cI_n$ that avoid all the patterns in  $S=\{\pi_1,\pi_2,\dots, \pi_j\}\subseteq \fS_3$ where $S$ contains more than one pattern. Two sets $S$ and  $T$ of patterns are $\cI$-Wilf equivalent if $|\cI_n(S)|=|\cI_n(T)|$ and we write $[S]_{\cI}=[\pi_1,\pi_2,\dots ,\pi_j]_{\cI}$ for the collection of sets of patterns. The cardinalities for multiple pattern avoidance in involutions has been classified and enumerated by Guibert and  Mansour in~\cite{GM02a} (Examples 2.6, 2.8, 2.12, 2.18, 2.20)
 who classify all pattern sets containing 132, Egge and Mansour~\cite{EM04} who enumerate the sets containing the pattern 231 and Wulcan who enumerates all pairs of length three patterns in~\cite{W02}. 
 Also in this section we further describe the generating functions for $\inv$ and $\maj$ for multiple patterns. Since the avoidance classes and associated generating functions for multiple patterns are fairly simple we instead consider the statistics $\inv$, $\maj$ and $\des$ altogether as a single generating function. 
For a set of patterns $S$ define
$$F_n(S;p,q,t)=F_n(S)=\sum_{\iota\in\cI_n(S)}p^{\inv(\iota)}q^{\maj(\iota)}t^{\des(\iota)}$$
and similarly define
$$\bar{F}_n(S;p,q,t)=\bar{F}_n(S)=\sum_{\iota\in\cI_n(S)}p^{\coinv(\iota)}q^{\comaj(\iota)}t^{\asc(\iota)}.$$
The functions $F_n(S)$ and $\bar{F}_n(S)$ determine each other since
$$F_n(S)=(pq)^{\binom{n}{2}}t^{n-1}\bar{F}_n(S;p^{-1},q^{-1},t^{-1}).$$
Two sets $S$ and $T$ of patterns are $I\cI$-Wilf equivalent if $F_n(S;q,1,1)=F_n(T;q,1,1) $ and $M\cI$-Wilf equivalent if $F_n(S;1,q,1)=F_n(T;1,q,1) $. Let $[S]_{I\cI}=[\pi_1,\dots ,\pi_j]_{I\cI}$ be the set of sets of patterns that are $I\cI$-Wilf equivalent to $S$ and we similarly define 
$[S]_{M\cI}$ for the major index. We include the description for all equivalence classes and generating functions for all sets of multiple patterns for completeness of this study, but do not include the proofs.  
We find, as before, any avoidance class that contains the pattern $231$ can be described as the avoidance class of permutations, which were studied by  Dokos et.\ al.~\cite{DDJSS12}.  In~\cite{BBES16} (Section 4.3) Barnabei et.\ al.\ find $M\cI_n(213,321)$. Since the avoidance class for any set that contains 123 and 321 becomes empty for $n>5$ we exclude all sets with these two patterns. 
 
 \begin{prop} The decompositions for involutions that avoid two patterns are as follows. 
\begin{enumerate}[(i)]
\item If $\iota \in \cI_n(123, 132)$ then either $\iota = 12[ \dd_{n-1}, 1]$ or $\iota = 45312[ \dd_k,1,  \tau, \dd_k, 1]$ where $\tau \in \cI_{n-2k-2}(123, 132)$ and $0\leq k<\floor{n/2}$.
\item If $\iota \in \cI_n(123, 213)$ then either $\iota = 12[1, \dd_{n-1}]$ or $\iota = 45312[1, \dd_k, \tau, 1, \dd_k]$ where $\tau \in \cI_{n-2k-2}(123, 213)$ and $0\leq k<\floor{n/2}$.
\item If $\iota \in \cI_n(123, 231)=\cI_n(123, 312)$ then $\iota = 12[\dd_{k},\dd_{n-k}]$ for $k \in [n]$ . 
\item If $\iota \in \cI_n(132, 231)=\cI_n(132, 312)$ then $\iota = 12[\dd_k, \ii_{n-k}]$ for $1\leq k \leq n$.
\item If $\iota \in \cI_n(132,321)$ then $\iota = 213[\ii_k, \ii_k, \ii_{n-2k}]$ for some $0\leq k\leq \floor{n/2}$.
\item If $\iota \in \cI_n(132, 213)$ then $\iota=\ii_n$ or $\iota = 321[\ii_k,\tau,  \ii_k]$ with $\tau \in  \cI_{n-2k}(213, 132)$ and $1\leq k\leq \floor{n/2}$.
\item If $\iota \in \cI_n(213, 231)=\cI_n(213, 312)$ then $\iota = 12[\ii_k, \dd_{n-k}]$ for $0\leq k \leq n-1$.
\item If $\iota \in \cI_n(213,321)$ then $\iota = 132[\ii_{n-2k}, \ii_k, \ii_k]$ for some $0\leq k\leq \floor{n/2}$.
\item If $\iota \in \cI_n(312,321)$ then $\iota = \ii_k[\tau_1,\ldots , \tau_l]$ where $\tau_j =1$ or $\tau_j=21$ for all $j$. \hqed
\end{enumerate}
\label{prop:twoavoid}
\end{prop}  

From these decomposition we can quickly determine the cardinalities. 

\begin{prop} The $\cI$-Wilf equivalence classes for pairs of permutations in $\fS_3$. 
\begin{enumerate}[(i)]
\item $[123,132]_{\cI}=\{\{123, 132\},\{123, 213\},  \{132,213\}\}$ with  $|\cI_n(123,132)|=2^{\lfloor n/2 \rfloor}$. 
\item $[123, 231]_{\cI}=\{\{123, 231\},  \{213, 231\}, \{132,231\}\}$ with $|\cI_n(123,312)|=n$
\item $[132,321]_{\cI}=\{\{132,321\}, \{213,321\}\}$ with $|\cI_n(321,132)|=\lfloor n/2 \rfloor +1$.
\item $[213,321]_{\cI}=\{ \{231,321\} \}$ with $|\cI_n(231,321)|=F_n$ the Fibonacci numbers. \hqed
\end{enumerate}
\end{prop}

We can also quickly determine the generating function $F_n(S)$ for any pair of patterns, which we present in Table~\ref{double}. 

\begin{table}
\begin{center}
\begin{tabular}{|c|c|}
\hline
$S=\{\pi_1,\pi_2\}$ & $F_n(S)$ or $\bar{F}_n(S)$\\
\hline
$\{123, 132\}$ & $\displaystyle \bar{F}_n(S;p,q,t)=(pq)^{n-1}t+\bar{F}_{n-2}(S;p,q,qt)+\sum_{k=1}^{\floor{n/2}-1} p^{2k}q^{n+k-1}t^2\bar{F}_{n-2k-2}(S;p,q,q^{k+1}t)$\\
\hline
$\{123, 213\}$ & $\displaystyle\bar{F}_n(S;p,q,t)=p^{n-1}qt+\bar{F}_{n-2}(S;p,q,qt)+\sum_{k=1}^{\floor{n/2}-1} p^{2k}q^{n-k+1}t^2\bar{F}_{n-2k-2}(S;p,q,q^{k+1}t)$\\
\hline
$\{123, 213\}$ &  $\displaystyle \bar{F}(S;p,q,t)=1+\sum_{k = 1}^{n-1}p^{k(n-k)}q^{k}t$\\
\hline
$\{132, 231\}$ & $\displaystyle F(S;p,q,t)=\sum_{k = 1}^n(pq)^{\binom{n}{2}}t^{k-1}$\\
\hline
$\{132,321\}$ & $\displaystyle F(S;p,q,t)=1+\sum_{k = 1}^{\floor{n/2}}p^{k^2}q^kt$\\
\hline
$\{132, 213\}$ & $\displaystyle \bar{F}_n(S;p,q,t)=(pq)^{\binom{n}{2}}t^{n-1}+\sum_{k = 1}^{ \floor{n/2}}p^{k(k-1)}q^{n(k-1)}t^{2(k-1)}\bar{F}_{n-2k}(S;p,q,q^{k}t)$\\
\hline
$\{213, 312\}$ & $\displaystyle F_n(S;p,q,t)=\sum_{k = 0}^{n-1}p^{\binom{n-k}{2}}q^{\binom{n-k}{2}+k(k-1)}t^{n-k-1}$\\
\hline
$ \{213,321\}$ & $ \displaystyle F_n(S;p,q,t)=1+\sum_{k = 1}^{\floor{n/2}}p^{k^2}q^{n-k}t$\\
\hline
$ \{312,321\}$ & $F_n(S)=F_{n-1}(S)+pq^{n-1}tF_{n-2}(S)$\\
\hline
\end{tabular}
\end{center}
\caption{The generating functions for doubletons. }
\label{double}
\end{table}

\begin{table}
\begin{center}
\begin{tabular}{|c|c|}
\hline
$S$ & $F_n(S)$ or $\bar{F}_n(S)$\\
\hline
$\{123,132,213\}$ & $\displaystyle \bar{F}_n(S;p,q,t)=\bar{F}_{n-2}(S;p,q,qt)+p^{2}q^{n}t^2\bar{F}_{n-4}(S;p,q,q^2t)$\\
\hline
$\{123,132,231\}$ & $\displaystyle \bar{F}_n(S)=1+(pq)^{n-1}t$\\
\hline
$\{123,213,231\}$ & $\displaystyle \bar{F}_n(S)=1+p^{n-1}qt$\\
\hline
$\{132,213,321\}$&$\displaystyle  F_{2k+1}(S)=1$ or $\displaystyle F_{2k}(S)=1+p^{k^2}q^kt$\\
\hline
$\{132,231,321\}$&$\displaystyle F = 1+pqt$\\
\hline
$\{132,213,231\}$&$\displaystyle F = 1+(pq)^{\binom{n}{2}}t^{n-1}$\\
\hline
$\{213,231,321\}$&$\displaystyle 1+pq^{n-1}t$\\
\hline
\end{tabular}
\end{center}
\caption{The generating functions for tripletons. }
\label{triple}
\end{table}

\begin{table}
\begin{center}
\begin{tabular}{|c|c|}
\hline
$S$ & $F_n(S)$ or $\bar{F}_n(S)$\\
\hline
$\{123,213,132,312\}$ & $F_n(S)=(pq)^{\binom{n}{2}}$\\
\hline 
$\{321,213,132,312\}$ & $F_n(S)=1$ \\
\hline
\end{tabular}
\end{center}
\caption{The generating functions for four patterns. }
\label{four}
\end{table}

We next describe the same for triples of patterns. Again when describing the sets and functions we exclude triples that both contain 231 and 312 or 123 and 321. 


\begin{prop} The decompositions for involutions that avoid three patterns. 
\begin{enumerate}[(i)]
\item If $S = \{123,132,213\}$ then for $\iota \in \cI_n(S)$ we have $\iota = 321[1, \tau, 1]$ for $\tau \in \cI_{n-2}(S)$ or $\iota = 321[12, \tau, 12]$ for $\tau \in \cI_{n-4}(S)$.
\item If $\iota \in \cI_n(123,132,231)$ then $\iota$ is $12[\dd_{n-1},1]$ or $\dd_n$. 
\item If $\iota \in \cI_n(123,213,231)$ then $\iota$ is $12[1,\dd_{n-1}]$ or $\dd_n$. 
\item If $\iota \in \cI_n(321,132,213)$ then $\iota$ is $\ii_n$ and if $n=2k$ then $\iota$ could be $12[\ii_k,\ii_k]$. 
\item If $\iota \in \cI_n(321,132,231)$ then $\iota$ is $\ii_n$ or $12[21,\ii_{n-2}]$. 
\item If $\iota \in \cI_n(321,213,231)$ then $\iota$ is $\ii_n$ or $12[\ii_{n-2},21]$. 
\item If $\iota \in \cI_n(213,132,231)$ then $\iota$ is $\ii_n$ or $\dd_n$. \hqed
\end{enumerate}
\end{prop}

From the decomposition we can quickly determine the cardinalities. 

\begin{prop} The $\cI$-Wilf equivalence classes for triples of permutations in $\fS_3$. 
\begin{enumerate}[(i)]
\item If $S = \{123,132,213\}$ then $|\cI_n(S)|=F_{\floor{n/2}}$.
\item If $S\in \{\{123,132,213\},\{123,213,231\},\{321,132,231\},\{321,213,231\},\{213,132,312\}\}$ then $|\cI_n(S)|=2$.
\item If $S=\{321,132,213\}$ then $|\cI_n(S)|=2$ when $n$ is even and $|\cI_n(S)|=1$ when $n$ is odd. \hqed
\end{enumerate}
\end{prop}

The generating functions for three patterns are in Table~\ref{triple}.

\begin{prop} The decompositions for involutions that avoid four patterns. 
\begin{enumerate}[(i)]
\item If $\iota\in\cI_n(123,213,132,312)$ then $\iota = \dd_n$. 
\item If $\iota\in\cI_n(321,213,132,312)$ then $\iota = \ii_n$. \hqed
\end{enumerate}
\end{prop}
The generating functions for four patterns are in Table~\ref{four}.
Also,  the $I\cI$-Wilf and $M\cI$-Wilf  equivalence classes for multiple patterns mirror those in the singleton case. Given a set $S$ of patterns define $r_m(S)=\{r_m(\pi):\pi\in S\}$ and similarly define $R_{\theta}(S)$. 
\begin{thm}
The $I\cI$-Wilf  and $M\cI$-Wilf  equivalence classes for multiple patterns in $\fS_3$ are described as follows. 
\begin{enumerate}[(i)]
\item The only equalities between $I\cI$-Wilf equivalence classes for  sets $S\subseteq \fS_3$ of the same size are between $S$, $r_1(S)$, $r_{-1}(S)$ and $R_{180}(S)$. 
\item The only equalities between $M\cI$-Wilf equivalence classes for  sets $S\subseteq \fS_3$ of the same size are between $S$ and $r_1(S)$. \hqed
\end{enumerate}
\end{thm}


\section{Symmetries for permutations}
\label{permsymm}

In Sections~\ref{maj213} and~\ref{maj123}  we demonstrated that the pairs of patterns 123 and 321 as well as 132 and 213 exhibit the symmetry $M\cI_n(\pi_1)=q^{\binom{n}{2}}M\cI_n(\pi_2;q^{-1})$. In both cases this symmetry holds for the larger class of permutations in that $M_n(\pi_1)=q^{\binom{n}{2}}M_n(\pi_2;q^{-1})$. We thank Vasu Tewari for asking about this generalization. We prove this by describing maps  $\fS_n(\pi_1)\rightarrow\fS_n(\pi_2)$ that commute with $r_1$ (i.e. taking inverses) as well as if $\sigma_1\mapsto\sigma_2$ then $\Asc(\sigma_1)=\Des(\sigma_2)$.

The map $\fS_n(123)\rightarrow \fS_n(321)$ with the stated properties is classical and an elegant generalization of the map for involutions in Proposition~\ref{thm:123&321 symmetry}.

\begin{prop}There exists a map $\fS_n(123)\rightarrow\fS_n(123)$ that
\begin{enumerate}
\item commutes with $r_1$ and 
\item if $\sigma_1\mapsto\sigma_2$ then $\Asc(\sigma_1)=\Des(\sigma_2)$.
\end{enumerate}
\end{prop}
\begin{proof} In this proof we use facts stated in Proposition~\ref{SYTfacts}. We will define a bijective map $\fS_n(123)\rightarrow\fS_n(123)$ that has the two stated properties. To define the map we first take  a permutation $\sigma\in \fS_n(123)$, which by  RSK corresponds to a pair of SYT $(P,Q)$ of the same shape. This shape has at most two-columns because the longest increasing sequence of $\sigma$ has length at most two. The transposed pair $(P^T,Q^T)$ of SYT of the same shape have at most two rows which will correspond to another permutation that has the longest decreasing sequence  length at most two, so is in $\fS_n(321)$. This certainly defines a bijection $\fS_n(123)\rightarrow\fS_n(321)$. 

We have proven before in Proposition~\ref{thm:123&321 symmetry} that $\Des(Q)=\Asc(Q^T)$.  If $\sigma_1\mapsto\sigma_2$ and $(P,Q)$ and $(P^T,Q^T)$ correspond to $\sigma_1$ and $\sigma_2$ by RSK respectively we know that $\Des(\sigma_1)=\Des(Q)$ and $\Des(\sigma_2)=\Des(Q^T)=Asc(\sigma_1)$, which proves property (ii). 

To show that this map commutes with $r_1$ we only need to show that if $\sigma_1\mapsto\sigma_2$ then $r_1(\sigma_1)\mapsto r_1(\sigma_2)$. Because $r_1(\sigma)$ is the inverse of $\sigma$ we must have that  $r_1(\sigma_1)$ corresponds to $(Q,P)$. Then $r_1(\sigma_1)$ will map to the permutation associated to $(Q^T,P^T)$ that is the inverse of $\sigma_2$ or equivalently $r_1(\sigma_2)$, which proves property (i). 
\end{proof}

\begin{cor}
For $n\geq 0$ we have the symmetry
$$M_n(123)=q^{\binom{n}{2}}M_n(321;q^{-1}).$$
\vspace{-1.1cm}

\hqed
\end{cor}

Though the proof for Proposition~\ref{thm:123&321 symmetry} generalizes quickly to permutations the proof provided in Theorem~\ref{thm:132symm213} for the pair of patterns 132 and 213 does not leave room for an obvious generalization. However, just defining a map to prove the symmetry between these two patterns for permutations is not too difficult to  either. One can define a map $\tilde\theta:\fS_n(132)\rightarrow\fS_n(213)$  inductively by mapping $\sigma=231[\alpha,1,\beta]$ that avoids 132 to $\sigma=312[\tilde\theta(\alpha),1,\tilde\theta(\beta)]$ that avoids 213. This map certainly changes ascents to descents so will map a permutation with an ascent set of $A$ to a permutation with a descent set of $A$. On the other hand, this map also certainly does not commute with $r_1$, a property we are interested in.

The rest of this section is dedicated to defining a map  $\theta:\fS_n(132)\rightarrow\fS_n(213)$ that has the additionally property of commuting with $r_1$, $\theta \circ r_1=r_1\circ\theta$. Before we define the map we need some definitions. We say that $\sigma(i)$ is a  {\it left-to-right maximum} of  $\sigma$ if $\sigma(i)$ is larger than everything to its left, $\sigma(i)=\max\{\sigma(1)\dots \sigma(i)\}$. The {\it LR maximums} of $\sigma$ will refer to the subsequence of all {\it left-to-right maximums} of $\sigma$. Similarly, $\sigma(i)$ is a  {\it right-to-left minimum} of  $\sigma\in \fS_n$ if $\sigma(i)$ is  smaller than everything to its right,  $\sigma(i)=\min\{\sigma(i)\dots \sigma(n)\}$. The {\it RL minimums} of $\sigma$ will refer to the subsequence of all {\it right-to-left minimums} of $\sigma$. We can similarly define RL maximums. For example the LR maximums of $371958264$ are $3,7,9$, the RL minimums are $1,2,4$ and the RL maximums  are $9,8,6,4$.  See Figure~\ref{fig:LRmaxExample} for an illustration. Note that the RL minimums and the LR maximums are increasing sequences but the RL maximums form a decreasing sequence. This is a fact we use often throughout the rest of the section. The next lemma details a few more specific properties of these subsequences that will be needed in proving properties of  our map $\theta$.

\begin{figure}
\begin{center}
\begin{tikzpicture} [scale = .4]
\begin{scope}[shift={(0,0)}]
\draw[gray,fill] (3,1) rectangle (1,9);
\draw[gray,fill] (7,2) rectangle (1,9);
\draw[gray,fill] (9,4) rectangle (1,9);
\draw[step=1cm,black,dashed] (1,1) grid (9,9);
\filldraw [black] 
(1,3) circle (5pt)
(2,7) circle (5pt)
(3,1) circle (5pt)
(4,9) circle (5pt)
(5,5) circle (5pt)
(6,8) circle (5pt)
(7,2) circle (5pt)
(8,6) circle (5pt)
(9,4) circle (5pt);
\draw (1,1) rectangle (9,9);
\end{scope}
\begin{scope}[shift={(11,0)}]
\draw[gray,fill] (1,3) rectangle (9,1);
\draw[gray,fill] (2,7) rectangle (9,1);
\draw[gray,fill] (4,9) rectangle (9,1);
\draw[step=1cm,black,dashed] (1,1) grid (9,9);
\filldraw [black] 
(1,3) circle (5pt)
(2,7) circle (5pt)
(3,1) circle (5pt)
(4,9) circle (5pt)
(5,5) circle (5pt)
(6,8) circle (5pt)
(7,2) circle (5pt)
(8,6) circle (5pt)
(9,4) circle (5pt);
\draw (1,1) rectangle (9,9);
\end{scope}
\end{tikzpicture}
\caption{The LR maximums of   $371958264$  are $3,7,9$, the RL minimums are $1,2,4$ and the RL maximums  are $9,8,6,4$.}
\label{fig:LRmaxExample}
\end{center}
\end{figure}
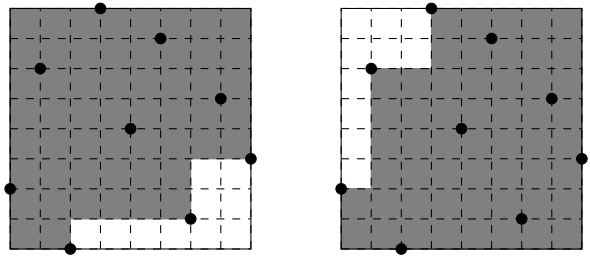

\begin{lemma} We have the following facts about RL minimums and LR maximums. 
\begin{enumerate}[(i)]
\item Given $x\in \fS_n$ with RL minimums $x(i_1),x(i_2),\dots ,x(i_a)$ we must have the LR maximums in $r_1(x)$ be $i_1,i_2,\dots, i_a$. 
\item  Given any $x,y$ that avoid 213 the LR maximums occur on consecutive indices and the RL minimums occur on consecutive values. This means the LR maximums of $y$ are  $y(1),y(2),\dots, y(b)$ for some $b$ with $y(b)=n$ and the RL minimums of $x$ are  $x(i_1),x(i_2),\dots, x(i_a)=1,2,\dots, a$  for some $a$ where $i_a=n$.
\item Given $y\in \fS_n(213)$ the union of the LR maximums and RL minimums  form the pattern $132[\ii_r,\ii_s,\ii_t]$ in $y$. 
\end{enumerate}
\label{lem:LRmaxFacts}
\end{lemma}

\begin{proof} To prove (i) we will consider the diagram for a permutation $x$. 
Given the diagram shade  everything above, to the left and between for every point. The right-to-left minimums will be those dots on the edge of the shading. We illustrate this shading in Figure~\ref{fig:LRmaxExample}. Similarly given the diagram of $x$  shade for every point everything below, to the right and between. The left-to-right maximums will be those dots on the edge of the shading. Using these facts it is easy to see that if the  RL minimums of $x$ are $x(i_1),x(i_2),\dots ,x(i_a)$ then the LR maximums in $r_1(x)$ are $i_1,i_2,\dots, i_a$.

Next, we prove (ii). Let $y$ avoid 213 and $y(j_1),y(j_2),\dots, y(j_b)$ be its LR maximums. Certainly $j_1=1$. We will also assume that there exists some index $k<j_a$ where $y(k)$ is not a left-to-right maximum. Because $j_1=1$ there exists a left-to-right maximum to the left of $y(k)$ because $y(k)$ is not a left-to-right maximum so there is some $y(j_p)>y(k)$ to the left of $y(k)$. The subsequence $y(j_p)y(k)y(i_a)$ forms the pattern 213 and we have a contradiction. Hence  the LR maximums of $y$ are  $y(1),y(2),\dots, y(b)$ for some $b$ with $y(b)=n$. Recall that $r_1(x)$  avoids 213 if $x$ avoids 213. We get the second part of (ii) using part (i) so the RL minimums of $x$ are  $x(i_1),x(i_2),\dots, x(i_a)=1,2,\dots, a$  for some $a$ where $i_a=n$.

Lastly we prove (iii). Using part (ii) we know that  the LR maximums of $y$ are $y(1),y(2),\dots, y(b)$ for some $b$ with $y(b)=n$ and the RL minimums are  $y(i_1),y(i_2),\dots ,y(i_c)=1,2,\dots, c$  for some $c$ where $i_c=n$. 
Consider the case where $y(1)\neq y(i_1)=1$. This implies that $c<y(1)$ since otherwise $y(1)1c$ forms the pattern 213 in $y$. Because $c<y(1)$  the two sequences form the pattern $21[\ii_b,\ii_c]$. Now consider if $y(1)= y(i_1)=1$ and further $y(i)=i$ for all $i\leq j$ for some $j$ with $y(j+1)\neq j+1$. This implies that our two sequences intersect for the first $j$ terms and then form the pattern $21[\ii_{b-j},\ii_{c-j}]$. All together we have that the LR maximums and the RL minimums form the pattern $132[\ii_j,\ii_{b-j},\ii_{c-j}]$, which finishes  part (iii). 
\end{proof}

In order to more easily define the map $\theta:\fS_n(132)\rightarrow\fS_n(213)$ we will define two operations $*$ and $\star$ on permutations and prove some properties about these operations. After we finish these lemmas defining our map $\theta$  and proving all the properties  for $\theta$  will be effortless. 

We define an operation  $x*y$ on permutations $x$ and $y$, which will be a key feature in our map. 
\begin{enumerate}
\item Let $x(i_1),x(i_2),\dots, x(i_a)$ be the sequence of RL minimums of $x$ and $y(j_1),y(j_2),\dots, y(j_b)$ be the sequence of LR maximums of $y$.
\item Note that $x(i_1),\dots, x(i_a),y(j_1),\dots ,y(j_b)$ forms the pattern $21[\ii_a,\ii_b]$ in $21[x,y]$.  Replace this pattern of $21[\ii_a,\ii_b]$ with $\ii_{a+b}$ in $21[x,y]$ to get $x*y$. 
\end{enumerate}
 See Figure~\ref{fig:*andstarExample} for an example. This operation $*$ has some nice properties all of which we prove in the next lemma.

\begin{lemma} Let $x\in \fS_k$ and $y\in \fS_{\ell}$. The operation $*$ is 
\begin{enumerate}[(i)]
\item  associative on permutations that avoid 213,  
\item takes  $x,y$ that avoid 213 to $x*y$ that avoids 213, 
\item has $r_1(x*y)=r_1(y)*r_1(x)$ and 
\item has $\Des(x*y)=\Des(x)\cup(\Des(y)+k)$.
\item If $x$ and $y$ avoid 213 then the left $k$ points of $x*y$ form the pattern $x$ and the bottom $\ell$ points of $x*y$ form the pattern $y$. 
\end{enumerate}
\label{lem:*facts}
\end{lemma}
\begin{proof} First we will show (i) that $*$ is associative on permutations that avoid 213. Let $x$, $y$ and $z$ be permutations avoiding 213. We will show that $(x*y)*z=x*(y*z)$. 
By Lemma~\ref{lem:LRmaxFacts} if $y$ avoids 213 the union of the LR maximums and RL minimums  form the pattern $132[\ii_r,\ii_s,\ii_t]$. The RL minimums of $x$ will form the pattern $\ii_a$   and the LR maximums of $z$ will form the pattern $\ii_c$. All together these LR maximums and RL minimums form the pattern $52431[\ii_a,\ii_r,\ii_s,\ii_t,\ii_c]$ in $321[x,y,z]$. We will show that we replace this pattern with $132[\ii_r,\ii_{a+s},\ii_{t+c}]$ in either $(x*y)*z$ or $x*(y*z)$, which will prove $(x*y)*z=x*(y*z)$. 

When determining $x*y$ we find the RL minimums of $x$ and the LR maximums of $y$ and then replace the pattern $21[\ii_a,\ii_{r+s}]$ in $21[x,y]$ with $\ii_{a+r+s}$ and get $x*y=\tau$. 
In a larger view the union of the LR maximums of  $y$ and the RL minimums of $x$ and $y$ form the pattern  $4132[\ii_a,\ii_r,\ii_s,\ii_t]$ in $21[x,y]$ that we replace with $132[\ii_r,\ii_{a+s},\ii_t]$ to get $\tau = x*y$ with the $\ii_r$ and $\ii_t$ portion forming the RL minimums of $\tau$.
So the RL minimums of $\tau$ are from the pattern $\ii_{r+t}$. To  find $\tau*z$ we need the  LR maximums of $z$, which form the pattern $\ii_c$. We replace the pattern $21[\ii_{r+t},\ii_{c}]$ in $21[\tau,z]$ with $\ii_{r+t+c}$. In conclusion we have replaced the pattern $52431[\ii_a,\ii_r,\ii_s,\ii_t,\ii_c]$ in $321[x,y,z]$ with $132[\ii_r,\ii_{a+s},\ii_{t+c}]$. By a very similar argument when determining $x*(y*z)$ we replace the pattern $52431[\ii_a,\ii_r,\ii_s,\ii_t,\ii_c]$ in $321[x,y,z]$ with $132[\ii_r,\ii_{a+s},\ii_{t+c}]$, which proves $(x*y)*z=x*(y*z)$. 

Secondly, we will show (v).  Note that the RL minimums of $x$ in $21[x,y]$  decrease in value in forming $x*y$ but remain an increasing subsequence. The values in the $x$ part of $21[x,y]$ not part of the RL minimums of $x$ remain unchanged in $x*y$. Since $x$ avoids 213 we know from part (ii) of Lemma~\ref{lem:LRmaxFacts} the RL minimums of $x$ are  $x(i_1),x(i_2),\dots, x(i_a)=1,2,\dots, a$  for some $a$ where $i_a=n$. This means that the left $|x|$ points of $x*y$ are order isomorphic to $x$. Using part (i) of Lemma~\ref{lem:LRmaxFacts} we can conclude that the bottom $|y|$ points of $x*y$ are order isomorphic to $y$.

Next we show (ii) by showing that $x*y$ avoids 213 if both $x$ and $y$ avoid 213. We will do so by induction on the length of $x$. Let $x\in \fS_k(213)$ and $y\in \fS_{n-k}(213)$. The base case is when $k=0$ and $\epsilon *y=y$ avoids 213 by assumption. We now assume that $k>0$. We must have $n$ occurring at some index $i$ in the $x$ part of $21[x,y]$. 
Let $\bar{x}$ be $x$ with $x(i)$ removed. The first case is if $x(i)$ is part of the RL minimums of $x$. This would means since $x$ avoids 213 that $i=k$, $x = \ii_k$ and $\bar{x} = \ii_{k-1}$. The LR maximums of $y$ by Lemma~\ref{lem:LRmaxFacts} must be $y(1),y(2),\dots ,y(b)$ for some $b$. We then have that $\bar{x}*y$ is $y(1)\dots y(b)(n-k+1)(n-k+2)\dots (n-1)y(b+1)\dots y(n-k)$, which avoids 213 by induction. Further we know that ${x}*y$ is $y(1)\dots y(b)(n-k+1)(n-k+2)\dots (n) y(b+1)\dots y(n-k)$. If ${x}*y$ contained a 213 then  $n$ must play the role of 3, which is impossible because ${x}*y$ strictly increases before $n$. 

The next case is when $x(i)$ is not part of the RL minimums of $x$. Then $x*y$ is $\bar{x}*y$ but we insert $n$ at index $i$ in $\bar{x}*y$. By induction  $\bar{x}*y$ avoids 213 so if $x*y$ contains a 213 then $n$ plays the role of 3 and the pattern is in the left $k$ indices of $x*y$. By part (v) the left $k$ points of $x*y$ are order isomorphic to $x$ so if $x*y$ contains a 213 in the left $k$ points then $x$ contains the pattern 213, which is a contraction. 

Next we will show (iii), that $r_1(x*y)=r_1(y)*r_1(x)$. Note that in forming $x*y$ we needed to find $x(i_1),x(i_2),\dots ,x(i_a)$ the RL minimums of $x$ and $y(j_1),y(j_2),\dots, y(j_b)$ the LR maximums of $y$. These points form the pattern $21[\ii_a,\ii_b]$ in $21[x,y]$ and we replace this pattern with $\ii_{a+b}$. By Lemma~\ref{lem:LRmaxFacts} the RL minimums of $r_1(y)$ are $j_1,j_2,\dots, j_b$ and the LR maximums of $r_1(x)$ are $i_1,i_2,\dots, i_a$. So in forming $r_1(y)*r_1(x)$ we replace the pattern  $21[\ii_b,\ii_a]$ in $21[r_1(y),r_1(x)]$  with $\ii_{a+b}$, which is  the same thing as $r_1(x*y)$. 

Finally we prove (iv) that if $x\in \fS_k(213)$ and $y\in \fS_{\ell}(213)$ then $\Des(x*y)=\Des(x)\cup(\Des(y)+k)$. Consider $\tau=x*y$. By part (v) the left $k$ points of $x*y$ are order isomorphic to $x$. Thus, the first $k$ indices of $\tau$ have the same descents as $x$. By Lemma~\ref{lem:LRmaxFacts}  the RL maximums of $y$ are $y(1),y(2),\dots,x(b)$. In forming $\tau$ these points increase in value. The subsequence $\tau(k+1)\tau(k+2)\dots\tau(n)$ while isn't order isomorphic to $y$, it does have the same ascents and descents. Because $i_a=k$, the first point in $y$ is a LR maximum and  the last point in $x$ is a RL minimum we must have an increase at $k$ in $\tau$. All together this implies that $\Des(x*y)=\Des(x)\cup(\Des(y)+k)$.  
\end{proof}

\begin{figure}
\begin{center}
\begin{tikzpicture} [scale = .5]
\begin{scope}[shift={(0,1)}]
\filldraw [black] 
(0,3) circle (5pt)
(1,1) circle (5pt)
(2,2) circle (5pt)
(3,0) circle (5pt);
\draw (0,0) rectangle (3,3);
\draw (4,1.5) node {$*$};
\end{scope}
\begin{scope}[shift={(5,1.5)}]
\filldraw [black] 
(0,0) circle (5pt)
(1,2) circle (5pt)
(2,1) circle (5pt);
\draw (0,0) rectangle (2,2);
\draw (3,1) node {$=$};
\end{scope}
\begin{scope}[shift={(9,0)}]
\filldraw [black] 
(0,6) circle (5pt)
(1,4) circle (5pt)
(2,5) circle (5pt)
(3,0) circle (5pt)
(4,2) circle (5pt)
(5,3) circle (5pt)
(6,1) circle (5pt);
\draw (0,0) rectangle (6,6);
\end{scope}
\begin{scope}[shift={(18,1)}]
\filldraw [black] 
(0,2) circle (5pt)
(1,3) circle (5pt)
(2,1) circle (5pt)
(3,0) circle (5pt);
\draw (0,0) rectangle (3,3);
\draw (5,1.5) node {$\star$\hspace{.05cm} $1=$};
\end{scope}
\begin{scope}[shift={(25,.5)}]
\filldraw [black] 
(0,2) circle (5pt)
(1,4) circle (5pt)
(2,3) circle (5pt)
(3,1) circle (5pt)
(4,0) circle (5pt);
\draw (0,0) rectangle (4,4);
\end{scope}
 \end{tikzpicture}
 \caption{On the left $4231*132=7561342$ and on the right $3421\star 1=35421$.}
 \label{fig:*andstarExample}
  \end{center}
 \end{figure}
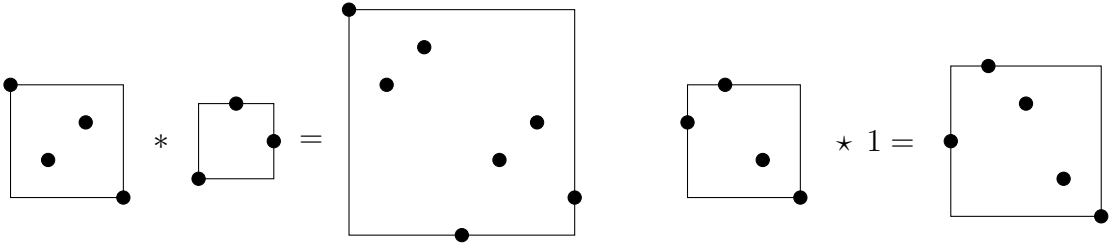
 
 We next  define our second operation  $\star$  for $\sigma\star 1$ when $\sigma$  avoids 213. Because $\sigma$ avoids $213$ we can write $\sigma=21[x,y]$ with $x,y\neq \epsilon$ or $\sigma=12[1,z]$. We now define $\sigma\star 1$. 
\begin{enumerate}
\item We define $1\star 1=21$. 
\item If $\sigma=12[1,z]$ then $\sigma\star 1=12[1,z\star 1]$. 
\item If $\sigma=21[x,y]$ with $x,y\neq \epsilon$ let $\bar{x}=x\star 1$ and $\bar{y}=y\star 1$. Say $\bar{x}$ has RL minimums $\bar{x}(i_1),\bar{x}(i_2),\dots, \bar{x}(i_a)$ and $\bar{y}$ has LR maximums $\bar{y}(j_1),\bar{x}(j_2)\dots ,\bar{y}(j_b)$. These together  form the pattern $21[\ii_a,\ii_{b}]$ in $21[\bar{x},\bar{y}]$. We replace this  pattern with $\ii_{a+b}$ but remove the point at $(i_a,y(j_b))$. We define $\sigma\star 1$ to be this new permutation.  
\end{enumerate}
Note that step (3) is like defining $21[x,y]\star 1 = (x\star 1) * (y\star 1)$ but we remove the point $(i_a,y(j_b))=(|x|+1,|y|+1)$  where $|x|$ gives the length of a permutation. In Figure~\ref{fig:*andstarExample} we illustrate $3421\star 1=35421$.  The operation $\star$ has some nice properties that we prove in the next lemma. 

\begin{lemma}Let $\sigma\in \fS_n(213)$. The operation $\star$ is
\begin{enumerate}[(i)]
\item  well defined in  that the decomposition choice of $\sigma=21[x,y]$ plays no role in the output,
\item takes a $\sigma$ avoiding 213 and outputs $\sigma\star 1$ which avoids 213, 
\item has $r_1(\sigma\star 1)=r_1(\sigma)\star 1$ and 
\item $\Des(\sigma\star 1)=\Des(\sigma)\cup\{n\}$.
\end{enumerate}
\label{lem:starfacts}
\end{lemma}
\begin{proof}
First we will prove that the decomposition choice of $\sigma=21[x,y]$ plays no role in the output $\sigma\star 1$. Consider $\sigma=321[x,y,z]$ that avoids 213. Let $\alpha = 21[x,y]$. By definition $\alpha \star 1$ is $(x\star 1)*(y\star 1)$ with $(|x|+1,|y|+1)$ removed and $21[\alpha,z]\star 1$ is $(\alpha\star 1)*(z\star 1)$ with $(|\alpha|+1,|z|+1)$ removed. Putting this together we have $321[x,y,z]\star 1$ equal to $((x\star 1)*(y\star 1))*(z\star 1)=(x\star 1)*(y\star 1)*(z\star 1)$ by Lemma~\ref{lem:*facts} with points $(|x|+1,|z|+|y|+2)$ and $(|x|+|y|+2,|z|+1)$ removed. This is the same output we would get if we had instead combined $y$ and $z$ first.

%

Next we will show (ii) that if  $\sigma$ avoid 213 then $\sigma\star 1$ avoids 213 by inducting on the length $|\sigma|=n$. Certainly if $n=1$ then $\sigma$ avoids 213 and $\sigma\star 1=21$ avoids 213. Now assume $n>1$.  Because $\sigma$ avoids 213 we have either $\sigma=21[x,y]$ with $x,y\neq\epsilon$ or $\sigma=12[1,z]$. If $\sigma=21[x,y]$ then because $x$ and $y$ have length smaller than $n$ we can use our inductive assumption and can assume that  $\bar{x}=x\star 1$ and $\bar{y}=y\star 1$ avoid 213. By Lemma~\ref{lem:*facts} we know that $\bar{x}*\bar{y}$ avoids 213 and since $\sigma\star 1$ is a pattern of $\bar{x}*\bar{y}$ we can conclude $\sigma\star 1$ avoids 213 in this case. Our other case is when $\sigma=12[1,z]$ and we defined $\sigma \star 1=12[1,z\star 1]$. Because $z$ has length smaller than $n$ we can assume that $z\star 1$ avoids 213, which implies that $\sigma \star 1=12[1,z\star 1]$ avoids 213. 

We will also show (iii), $r_1(\sigma\star 1)=r_1(\sigma)\star 1$,  by induction. Certainly $r_1(\sigma\star 1)=r_1(\sigma)\star 1$ if $\sigma = 1$. Assume the length of $\sigma$ is greater than 1. Again because  $\sigma$ avoids 213 we have either $\sigma=21[x,y]$ with $x,y\neq\epsilon$ or $\sigma=12[1,z]$. If $\sigma=21[x,y]$ then $\sigma\star 1=(x\star 1)*(y\star 1)$ with the point $(|x|+1,|y|+1)$ removed. Consider $r_1(\sigma)=21[r_1(y),r_1(x)]$ then $r_1(\sigma)\star 1=(r_1(y)\star 1)*(r_1(x)\star 1)$ with $(|y|+1,|x|+1)$ removed. By our inductive assumption $r_1(x\star 1)=r_1(x)\star 1$ and $r_1(y\star 1)=r_1(y)\star 1$ so using Lemma~\ref{lem:*facts} we have that $r_1(\sigma)\star 1=r_1(y\star 1)*r_1(x\star 1)=r_1((x\star 1)*(y\star 1))$ with $(|y|+1,|x|+1)$ removed because $(|x|+1,|y|+1)$ was removed from $(x\star 1)*(y\star 1)$. This proves for this case $r_1(\sigma\star 1)=r_1(\sigma)\star 1$ as we wanted. The other case is when $\sigma=12[1,z]$. Then $r_1(\sigma\star 1)=r_1(12[1,z\star 1])=12[1,r_1(z)\star 1]$ because $r_1(z\star 1)=r_1(z)\star 1$ by our inductive assumption. This proves $r_1(\sigma\star 1)=12[1,r_1(z)]\star 1=r_1(\sigma)\star 1$.

Finally, we show (iv) that $\Des(\sigma\star 1)=\Des(\sigma)\cup\{n\}$, which we also prove by inducting on $|\sigma|=n$. Certainly if $n=1$ then $\Des(\sigma\star1)=\{1\}$ so we can assume that $n>1$. Again, since $\sigma$ avoids 213 we have either $\sigma=21[x,y]$ with $x,y\neq\epsilon$ or $\sigma=12[1,z]$. In the first case Let $|x|=k$ and $|y|=\ell$ so if $\sigma=21[x,y]$ then $\Des(\sigma)=\Des(x)\cup \{k\}\cup(\Des(y)+k)$. By induction $\Des(x\star1)=\Des(x)\cup\{k\}$ and $\Des(y\star 1)=\Des(y)\cup\{\ell\}$. By Lemma~\ref{lem:*facts} we have that $\Des((x\star1)*(y\star1))=\Des(x)\cup\{k\}\cup (\Des(y)+k+1)\cup\{\ell+k+1\}$. Now we only have to consider the removal of the point $(k+1,\ell+1)$. We have that $k\in \Des(x\star 1)$ so the point at index $k$ is higher than the one at index $k+1$ implying that the value at index $k$ is not a RL minimum of $x\star 1$. Further the value at $k+1$ in $21[x\star 1,y\star 1]$ is the maximum of the union of RL minimums of $x\star 1$ and LR maximums of $y\star 1$ in $21[x\star 1,y\star 1]$. That means when we do the pattern replacement in forming $(x\star1)*(y\star1)$ we still have a descent at index $k$ even after removing the point $(k+1,\ell+1)$. As result $\Des(\sigma\star 1)=\Des(x)\cup\{k\}\cup (\Des(y)+k)\cup\{\ell+k\}=\Des(\sigma)\cup\{n\}$. Next we consider the case where $\sigma=12[1,z]$ so $\Des(\sigma)=\Des(z)+1$. By induction $\Des(z\star 1)=\Des(z)\cup\{n-1\}$ so since $\sigma\star 1 = 12[1,z\star 1]$  has descent set $\Des(z\star 1)+1$ we can conclude that $\Des(\sigma\star 1)=\Des(\sigma)\cup\{n\}$. 
\end{proof}

At this point we have defined all the operations and facts we need to define $\theta:\fS_n(132)\rightarrow\fS_n(213)$ and swiftly prove that $\theta$ commutes with $r_1$ and $\Asc(\sigma)=\Des(\theta(\sigma))$. We will define $\theta$ inductively. Consider a permutation $\sigma$ that avoids 132. Let $\theta(1)=1$ and assume now that $\sigma$ has length at least two. Either $\sigma$ can be decomposed  as $\sigma=21[\alpha,\beta]$ for $\alpha,\beta\neq \epsilon$ or $\sigma=12[\gamma, 1]$. In the first case we define
 \begin{equation}
 \theta(21[\alpha,\beta])=\theta(\alpha)*\theta(\beta)
 \label{eq:theta1}
 \end{equation}
 and in the second case
 \begin{equation}
 \theta(12[\gamma, 1])=\theta(\gamma)\star 1.
  \label{eq:theta2}
  \end{equation}
 See Figure~\ref{fig:132to213permex} for an example. We  will now prove that $\theta$ is well-defined, commutes with $r_1$ and has $\Asc(\sigma)=\Des(\theta(\sigma))$.

\begin{lemma} The map $\theta:\fS_n(132)\rightarrow\fS_n(213)$ defined above is well defined,
\begin{enumerate}[(i)]
\item commutes with $r_1$ and
\item has $\Asc(\sigma)=\Des(\theta(\sigma))$.
\end{enumerate}
\label{lem:almostTheta}
\end{lemma}
\begin{proof}
Let $\sigma\in \fS_n(132)$. We have two cases to consider in proving all the properties, which are either $\sigma=21[\alpha,\beta]$ with $\alpha,\beta\neq\epsilon$ or $\sigma=12[\gamma,1]$. First we will mention why this map is well defined and show that $\theta(\sigma)$ avoids 213 by inducting on $n$. The $n=1$ case is straight forward so assume $n>1$. In the first case if $\sigma=21[\alpha,\beta]$ then $\theta(\sigma)=\theta(\alpha)*\theta(\beta)$. By our inductive assumption $\theta(\alpha)$ and $\theta(\beta)$ avoid 213. Using  Lemma~\ref{lem:*facts} we can conclude that $\theta(\alpha)*\theta(\beta)=\theta(\sigma)$ also avoids 213. In the second case where  $\sigma=12[\gamma,1]$ we have that $\theta(\sigma)=\theta(\gamma)\star 1$. By our inductive assumption $\theta(\gamma)$ avoids 213 so by  Lemma~\ref{lem:starfacts}  we can conclude that $\theta(\gamma)\star 1$ also avoids 213. 

Next, we show that $\theta$ commutes with $r_1$. Consider  $\theta\circ r_1(\sigma)$. If we have the first case that  $\sigma=21[\alpha,\beta]$ then $\theta\circ r_1(\sigma)=\theta(21[r_1(\beta),r_1(\alpha)])=r_1(\beta)*r_1(\alpha)=r_1(\alpha*\beta)$ by Lemma~\ref{lem:*facts}. Since $r_1(\alpha*\beta)=r_1\circ\theta(\sigma)$ we can conclude that $\theta$ and $r_1$ commute in this case. Next consider when $\sigma=12[\gamma,1]$ then $(\theta\circ r_1)(\sigma)=\theta(12[r_1(\gamma),1])=r_1(\gamma)\star 1=r_1(\gamma\star 1)$ by Lemma~\ref{lem:starfacts}. Since $r_1(\gamma\star 1)=r_1\circ \theta(\sigma)$ we can conclude that $\theta$ and $r_1$ commute in all cases. 

Finally, we will prove that $\Asc(\sigma)=\Des(\theta(\sigma))$. We will prove this by inducting on $n$. The case of $n=1$ is again straight forward so we will assume $n>1$. 
First consider if  $\sigma=21[\alpha,\beta]$ where $\alpha\in \fS_k(132)$, $\beta\in \fS_{\ell}(132)$ and $k,\ell \neq 0$. By induction we know $\Asc(\alpha)=\Des(\theta(\alpha))$ and $\Asc(\beta)=\Des(\theta(\beta))$. 
Also we know $\Des(\sigma)=\Des(\alpha)\cup\{k\}\cup(\Des(\beta)+k)$, so all we need to show that is $\Asc(\sigma)=\Asc(\alpha)\cup(\Asc(\beta)+k)$. We have $\theta(\sigma)=\theta(\alpha)*\theta(\beta)$ so $\Des(\theta(\sigma))=\Des(\theta(\alpha))\cup(\Des(\theta(\beta))+k)$ by Lemma~\ref{lem:*facts}. This equals $\Asc(\alpha)\cup(\Asc(\beta)+k)$ by our inductive assumptions so we are done in this case. Now consider if $\sigma=12[\gamma,1]$ so $\Des(\sigma)=\Des(\gamma)$. We want to show that $\Des(\theta(\sigma))=\Asc(\gamma)\cup\{n-1\}$. Using Lemma~\ref{lem:starfacts} we get $\Des(\theta(\gamma))=\Des(\theta(\gamma)\star 1)=\Des(\theta(\gamma))\cup\{n-1\}$. By our inductive assumption $\Asc(\gamma)=\Des(\theta(\gamma))$ so we further have $\Des(\theta(\sigma))=\Asc(\gamma)\cup\{n-1\}$, which completes the proof. 
 \end{proof}

\begin{figure}
\begin{center}
\begin{tikzpicture} [scale = .5]
\begin{scope}[shift={(0,0)}]
\draw (0,0) rectangle (3,3);
\filldraw [black] 
(4,4) circle (5pt)
(0,3) circle (5pt)
(1,0) circle (5pt)
(2,1) circle (5pt)
(3,2) circle (5pt);
\draw (0,0) rectangle (4,4);
\draw (5,2) node {$\overset{\theta}{\rightarrow}$};
\end{scope}
\begin{scope}[shift={(6,0)}]
\filldraw [black] 
(0,2) circle (5pt)
(1,4) circle (5pt)
(2,3) circle (5pt)
(3,1) circle (5pt)
(4,0) circle (5pt);
\draw (0,0) rectangle (4,4);
\end{scope}
\begin{scope}[shift={(15,0)}]
\draw (0,2.5) rectangle (3.5,6);
\draw (3.5,0) rectangle (6,2.5);
\filldraw [black] 
(0,5) circle (5pt)
(1,6) circle (5pt)
(2,3) circle (5pt)
(3,4) circle (5pt)
(4,1) circle (5pt)
(5,0) circle (5pt)
(6,2) circle (5pt);
\draw (0,0) rectangle (6,6);
\draw (7,3) node {$\overset{\theta}{\rightarrow}$};
\end{scope}
\begin{scope}[shift={(23,0)}]
\filldraw [black] 
(0,6) circle (5pt)
(1,4) circle (5pt)
(2,5) circle (5pt)
(3,0) circle (5pt)
(4,2) circle (5pt)
(5,3) circle (5pt)
(6,1) circle (5pt);
\draw (0,0) rectangle (6,6);
\end{scope}
 \end{tikzpicture}
 \caption{On the left $\theta(41235)=35421$ and on the right $\theta(6745231)=7561342$.}
 \label{fig:132to213permex}
  \end{center}
 \end{figure}
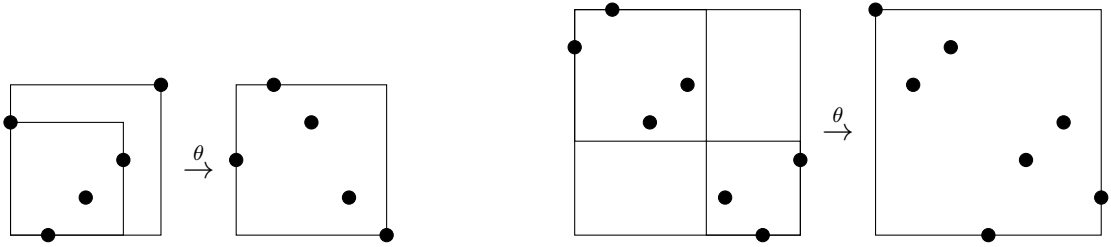

So far we have shown that $\theta$ has all the properties we want save for showing that this map is indeed a bijection. We first define some helpful terminology. Let $\sigma\in \fS_n(213)$. We will say that $\sigma$ is {\it $*$-splittable} if $\sigma=x*y$ for some permutations $x,y\neq\epsilon$ that avoid 213. Consider the sequence $\sigma(m_1),\sigma(m_2),\dots, \sigma(m_l)$ of the RL maximums of $\sigma$. Also let $u_i=m_{i}-m_{i-1}$ and  $m_0=0$. This keeps track of the number of indices to the left of the $i$th RL maximum but to the right of the $(i-1)$st RL maximum including the $i$th RL maximum. Let $v_i=\sigma(m_{i})-\sigma(m_{i+1})$ and $\sigma(m_{l+1})=0$. This keeps track of the number of indices below the $i$th RL maximum but above the $(i+1)$st RL maximum including the $i$th RL maximum. Because $\sigma$ avoids 213 we know that $\sigma$ will be strictly increasing for all points that are to the left and below of  $(m_i,\sigma(m_i))$. Let $p_i$ equal the number of points that are to the left and below  of $(m_i,\sigma(m_i))$  including $(m_i,\sigma (m_i))$. We illustrate this geometrically in Figure~\ref{fig:uivipiDiagram}. We will use this notation of $u_i$, $v_i$, $p_i$ and $\sigma(m_i)$ for the next couple lemmas. Recall when constructing $x*y$ we find $x(i_1)\dots x(i_a)$ the RL minimums of $x$ and $y(j_1)\dots y(j_b)$ the LR maximums for $y$. These values form the pattern $21[\ii_a,\ii_b]$ in $21[x,y]$, which we replace with $\ii_{a+b}$. The result of this is that the point $(y(j_b),x(i_a))$ becomes a RL maximum in $\sigma=x*y$. Say this RL maximum is $\sigma(m_j)$. We  further  have that $u_j\geq b+1$, $v_j\geq a+1$ and $p_j=a+b$. From this we can conclude that $u_j+v_j\geq p_j+2$. This is actually a sufficient condition for $*$-splittable. 

\begin{figure}
\begin{center}
\begin{tikzpicture} [scale = .5]
\draw[fill,gray] (0,0) rectangle (3.6,3.6);
\draw[dashed] (1,5)--(0,5);
\draw[dashed] (1,5)--(1,0);
\draw[dashed] (3.5,3.5)--(0,3.5);
\draw[dashed] (3.5,3.5)--(3.5,0);
\draw[dashed] (5,1)--(5,0);
\draw[dashed] (5,1)--(0,1);
\filldraw [black] 
(1,5) circle (3pt)
(3.5,3.5) circle (3pt)
(5,1) circle (3pt);
\draw (0,0) rectangle (6,6);
\draw (-.3,5.5) node {\tiny $(m_{i-1},\sigma(m_{i-1}))$};
\draw (5,4) node {\tiny$(m_{i},\sigma(m_{i}))$};
\draw (6.6,1.5) node {\tiny$(m_{i+1},\sigma(m_{i+1}))$};
\draw (2.3,4.3) node {\tiny$u_i$};
\draw [decorate,decoration={brace,amplitude=6pt},rotate=0] (1.1,3.7) -- (3.5,3.7);
\draw [decorate,decoration={brace,amplitude=6pt},rotate=270] (-1.1,0) -- (-3.5,0);
\draw (-.7,2.2) node {\tiny$v_i$};
\draw (2,2) node {\tiny$p_i$};
 \end{tikzpicture}
 \caption{Given the RL maximums $(m_i,\sigma(m_i))$ of $\sigma$ we illustrate  the values $u_i=m_{i}-m_{i-1}$, $v_i=\sigma(m_{i})-\sigma(m_{i+1})$ and $p_i$, which is the number of points to the left and below the $i$th RL maximum.}
 \label{fig:uivipiDiagram}
  \end{center}
 \end{figure}
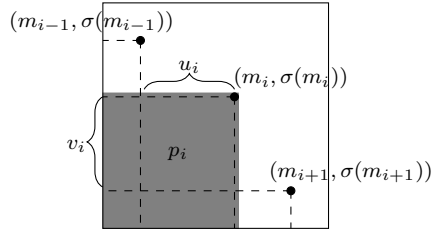

\begin{lemma}A permutation $\sigma$ is $*$-splittable if there exists a RL maximum $\sigma(m_j)$ such that $u_j+v_j\geq p_j+2$. Particularly if $\sigma$ avoids 213 and $\sigma = x*y$ then both $x$ and $y$ avoid 213. 
\label{lem:*split}
\end{lemma}
\begin{proof}
Consider $\sigma$ with all the mentioned conditions. The $p_j$ points to the left and below of $(m_j,\sigma(m_j))$ and including $(m_j,\sigma(m_j))$ form the pattern $\ii_{p_j}$. We will replace this pattern with $21[\ii_k,\ii_{p_j-k}]$ for some $k\in [p_j-1]$. In all cases our new permutation has the decomposition $21[x,y]$. Note that if $p_j-k\leq u_j-1$ that the right-most part of $\ii_k$ is a RL minimum of $x$. Also if $k\leq v_j-1$ then the right-most part of $\ii_{p_j-k}$ becomes a LR maximum of $y$. If we choose $k=p_j-u_j+1$ and $p_j-k=u_j-1$ we get both conditions $p_j-k\leq u_j-1$ and $k\leq v_j-1$. It follows that  the points from $\ii_k$ in $x$ become the RL minimums of $x$ and the points from $\ii_{p_j-k}$ in $y$ become the LR maximums of $y$. This implies $\sigma=x*y$ so $\sigma$ is $*$-splittable. 

Consider the case where $\sigma$ is $*$-splittable as in the previous paragraph. We want to show that $x$ and $y$ avoid 213. By part (v) of Lemma~\ref{lem:*facts}  the left $|x|$ points of $\sigma=x*y$  are order isomorphic to $x$. Because $\sigma$ avoided 213 we must have that $x$ does as well. By part (v) of Lemma~\ref{lem:*facts} the bottom $|y|$ points of $\sigma=x*y$ are isomorphic to $y$. Because $\sigma$ avoided 213 we must have that $y$ does as well. 
\end{proof}

Now consider the case where $\sigma\in \fS_n(213)$ is not $*$-splittable, so $u_i+v_i\leq p_i+1$ for all $i$. We want to show in this case that  $\sigma$ is {\it $\star$-splittable} where we can write $\sigma=z\star 1$ for some $z$. Recall when constructing $z\star 1$ where $z=21[x,y]$ with $x\in \fS_s(213)$ and $y\in \fS_t(213)$ that we had $z\star 1=(x\star 1)*(y\star 1)$ 
with the point $(s+1,t+1)$ removed. Certainly $(x\star 1)*(y\star 1)$ is $*$-splittable so by Lemma~\ref{lem:*split} there is some index $j$ with $u_j+v_j\geq p_j+2$. Note that the point $(s+1,t+1)$ is above the $(j+1)$st RL maximum and to the right of the $(j-1)$st RL maximum, so when removing  
the point $(s+1,t+1)$ to get $z\star1$  we decrease $u_j$, $v_j$ and $p_j$ by  one. This implies our inequality $u_j+v_j\geq p_j+2$ becomes $u_j+v_j\geq p_j+1$. This together with the assumption that $\sigma$ is not $*$-splittable  implies that $u_j+v_j= p_j+1$. We will show that $u_i+v_i\leq p_i+1$ for all $i$ is a sufficient condition for a permutation avoiding 213 to be $\star$-splittable. 

\begin{lemma}A permutation $\sigma$ avoiding 213 is $\star$-splittable if $u_i+v_i\leq p_i+1$ for all $i$. Particularly if $\sigma$ avoids $213$ and $\sigma = w\star 1$ then $w$ avoids $213$.
\label{lem:starsplit}
\end{lemma}
\begin{proof}
Consider $\sigma$ with all the mentioned conditions. 
Note that we are only concerned with the cases where $n=|\sigma|>1$. 
Also assume that there exists a $j$ such that $u_j+v_j=p_j+1$ for a RL maximum that is not the first or the last. We will first show that $\sigma$ is $\star$-splittable with this condition by induction on $|\sigma|$. If $|\sigma|=2$ then only $\sigma=21=1\star 1$ satisfies the conditions but then $\sigma$ is $\star$-splittable. Let $|\sigma|>2$. The $p_j$ points to the left and below of $(m_j,\sigma(m_j))$ and including $(m_j,\sigma(m_j))$ form the pattern $\ii_{p_j}$. Consider the subcollection of points counted by $p_j$ that are to the right of $(m_{j-1},\sigma(m_{j-1}))$ and above $(m_{j+1},\sigma(m_{j+1}))$. This subcollection isn't empty since it contain $(m_j,\sigma(m_j))$ so there exists a left-most point in the subcollection. We will insert a new point $(s,t)$ just below and to the left of this point. Note that with this we increase $u_j$, $v_j$ and $p_j$ by one, which we will notate $\check{u}_j$, $\check{v}_j$ and $\check{p}_j$  so we have $\check{u}_j+\check{v}_j=\check{p}_j+2$. If we take the pattern $\ii_{\check{p}_j}$ created by these $\check{p}_j$ points and replace it with $21[\ii_{\check{p}_j-\check{u}_j+1},\ii_{\check{u}_j-1}]$ we create a permutation  with decomposition $21[\bar{x},\bar{y}]$ just as we had in Lemma~\ref{lem:*split} with $\sigma=\bar{x}*\bar{y}$. 
Note that this means that $s=|\bar{x}|$ and $t=|\bar{y}|$ where $(s,t)$ was the point we added earlier. 
The associated $u^{\bar{x}}_i$, $v^{\bar{x}}_i$ and $p^{\bar{x}}_i$ values for $\bar{x}$ are the old values $u^{\bar{x}}_i=u_i$, $v^{\bar{x}}_i=v_i$ and $p^{\bar{x}}_i=p_i$ for $1\leq i\leq j-1$ and $u^{\bar{x}}_j=1$, $v^{\bar{x}}_j=\check{p}_j-\check{u}_j+1$ and $p^{\bar{x}}_j=\check{p}_j-\check{u}_j+2$. So $u^{\bar{x}}_i+v^{\bar{x}}_i\leq p^{\bar{x}}_i+1$ for all $i$.  By induction $\bar{x}$ is $\star$-splittable and $\bar{x}=x\star 1$ for some $x$ avoiding 213. 
The associated $u^{\bar{y}}_i$, $v^{\bar{y}}_i$ and $p^{\bar{y}}_i$ values for $\bar{y}$ are the old values $u^{\bar{y}}_i=u_i$, $v^{\bar{y}}_i=v_i$ and $p^{\bar{y}}_i=p_i$ for $j+1\leq i\leq l$ and $u^{\bar{y}}_j=\check{u}_j+1$, $v^{\bar{y}}_j=1$ and $p^{\bar{y}}_j=\check{u}_j+2$. So $u^{\bar{y}}_i+v^{\bar{y}}_i\leq p^{\bar{x}}_i+1$ for all $i$.
By induction $\bar{y}$ is $\star$-splittable and $\bar{y}=y\star 1$ for some $y$ avoiding 213. 
We now have $\sigma=(x\star 1)*(y\star 1)$ with the point $(|x|+1,|y|+1)=(s,t)$ removed. Hence $\sigma=21[x,y]\star 1$  is $\star$-splittable. 

We now will show in this case that if $\sigma$ avoids 213 and $\sigma = w\star 1$ then $w$ avoids 213. We will do so by induction on $|\sigma|$. The base case is straightforward so assume $|\sigma |>2$. We will first argue  that when we introduced the point $(s,t)$ into $\sigma$ that we did not create the pattern 213. By the  way we included this point we know that there is another point $(s+1,t+1)$ since we inserted $(s,t)$ just to the left and below of another point. If we did create the pattern 213 by inserting $(s,t)$ then this pattern can not involve $(s+1,t+1)$. However, we can replace the point $(s,t)$ in the  213 pattern with the point $(s+1,t+1)$ and maintain the  213 pattern, which is a contradiction  since $\sigma$ avoids 213. Since $\sigma$ with this point included is $\bar{x}*\bar{y}$  we can conclude that $\bar{x}$ and $\bar{y}$ avoid 213 by Lemma~\ref{lem:*split}. By induction this means both $x$ and $y$ avoid 213 so $21[x,y]$ avoids 213, which completes our argument in this case because $\sigma=21[x,y]\star 1$.

Next we consider the case where  $u_i+v_i\neq p_i+1$ for all $i\neq 1,l$ so $u_i+v_i\leq p_i$ for all $i\neq 1,l$ and $u_i+v_i\leq p_i+1$ for  $i= 1,l$. If $l$, the number of RL maximums, is $1$ then $\sigma=\ii_n$ and $u_1=v_1=p_1$, which contradicts $u_1+v_1\leq p_1+1$ unless $n=1$. If $l=2$ then we have $u_1=p_1$ and $v_2=p_2$, which implies $v_1=u_2=1$ and $\sigma=12[\ii_{n-2},\dd_2]$. By definition the definition of $\star$ we have $\sigma = \ii_{n-1}\star 1$ so $\sigma$ is $\star$-splittable. Now assume $l>2$. 
We will consider several cases depending on $\sigma(1)$ and will end up showing $\sigma=12[1,\tau]$. The first case is when $\sigma(1)$ is a RL maximum, which forces $\sigma(1)=n$ where $|\sigma|=n$. It follows that the second RL maximum is $(m_2,n-1)$ and all the points counted by $p_2$  are also counted by $u_2$. It follows that $u_2+v_2\geq p_2+1$ and because $l>2$ we have a contradiction. Our second case is when $\sigma(1)$ is below the first RL maximum but above the second. Because $\sigma$ avoids 213  all the points to the left of the first RL maximum must be weakly above $(1,\sigma(1))$. Because the second RL maximum is below $(1,\sigma(1))$ we have $u_1=v_1=p_1$. For the inequality $u_1+v_1\leq p_1+1$ to hold we need $p_1=1$, which contradicts $(1,\sigma(1))$ not being a RL maximum. 
Our third case is when $\sigma(1)$ is below the second RL maximum but is above the last. Say that $\sigma(1)$ is above $(m_j,\sigma(m_j))$ but is below the $(j-1)$st RL maximum. Consider the $p_{j-1}$ points associated to the $(j-1)$st RL maximum. Because $\sigma$ avoids 213 none of these points can be below $(1,\sigma(1))$ and to the left of $(m_{j-1},\sigma(m_{j-1}))$, so $v_{j-1}$ counts all the $p_{j-1}$ points. We then have $u_{j-1}+v_{j-1}\geq p_{j-1}+1$, which is a contradiction. Our last case is when $\sigma(1)$ is below the last RL maximum. Because $\sigma$ avoids 213 we must then have that $\sigma(1)=1$ so $\sigma=12[1,\tau]$. Note that in $\tau$ the $u_i$'s and $v_i$'s are the same as they were in $\sigma$ except $u_1$ decreases by one and $v_l$ decreases by one where $l$ indicates the number of RL maximums. We also have that $p_i$ decreases by 1 for all $i$. Because we assumed that $u_i+v_i\leq p_i$ for all $i\neq 1,l$ and $u_i+v_i\leq p_i+1$ for $i= 1,l$ in $\sigma$ we now have that $u^{\tau}_i+v^{\tau}_i\leq p_i^{\tau}+1$ for all $i$ for the associated values in $\tau$. This means $\tau$ is $\star$-splittable so $\tau=z\star 1$ for some $z$. It follows that $\sigma=12[1,z\star 1]=12[1,z]\star 1$ and $\sigma$ is $\star$-splittable. 

Lastly, we will show in this second case that if $\sigma$ avoids 213 and $\sigma = w\star 1$ then $w$ avoids 213 by induction on $|\sigma|$. Again the base case is straightforward so assume $|\sigma|>2$. Because $\sigma$ avoids 213 and $\sigma=12[1,\tau]$ we must have that $\tau$ avoids 213. By induction if $\tau$ avoids 213 and $\tau=z\star 1$ then $z$ avoids 213, which further implies that $12[1,z]$ avoids 213. This completes our argument in this case because $\sigma=12[1,z]\star 1$.
\end{proof}

Now that we proven the conditions for $*$-splittable and $\star$-splittable we can prove that $\theta$ is a bijection. 

\begin{thm}
The map $\theta:\fS_n(132)\rightarrow\fS_n(213)$ is a well-defined bijection that  commutes with the map $r_1$ and has $\Asc(\sigma)=\Des(\theta(\sigma))$. 
\label{thm:theta}
\end{thm}

\begin{proof}
From Lemma~\ref{lem:almostTheta} we know that $\theta$ is well-defined, commutes with $r_1$ and has $\Asc(\sigma)=\Des(\theta(\sigma))$. The last thing we need to show is that $\theta$ is indeed a bijection. Because $|\fS_n(132)|=|\fS_n(213)|=C_n$ for all $n$ it suffices to show that $\theta$ is surjective. We will prove this by induction on $n$. The base case of $n=1$ is true so assume that $n>1$ and $\theta:\fS_k(132)\rightarrow\fS_k(213)$ is bijective for all $k<n$.

Let $\sigma\in \fS_n(213)$ and $u_i$, $v_i$ and $p_i$ be defined as in Lemmas~\ref{lem:*split} and \ref{lem:starsplit}. If there exists a $j$ such that $u_j+v_j\geq p_j+2$ then by Lemma~\ref{lem:*split} we know that $\sigma$ is $*$-splittable and $\sigma=x*y$ where $x$ and $y$ avoid 213. By induction there exists $\alpha$ and $\beta$ avoiding 132 such that $\theta(\alpha)=x$ and $\theta(\beta)=y$, which implies $\theta(21[\alpha,\beta])=\theta(\alpha)*\theta(\beta)=\sigma$. 

The alternative is that $u_i+v_i\leq  p_i+1$ for all $i$. By Lemma~\ref{lem:starsplit} we know that $\sigma$ is $\star$-splittable so $\sigma=z\star 1$ for some $z$ that avoids 213. By induction there exists a $\gamma$ avoiding 132 such that $\theta(\gamma)=z$. Further $\theta(12[\gamma,1])=\theta(\gamma)\star 1 = \sigma$. Hence $\theta$ is surjective for all $n$ and thus a bijection. 
\end{proof}

The map in Theorem~\ref{thm:theta} is sufficient to prove the symmetry $M_n(132)=q^{\binom{n}{2}}M_n(213;q^{-1}).$

\begin{cor}
For $n\geq 0$ we have the symmetry
$$M_n(132)=q^{\binom{n}{2}}M_n(213;q^{-1}).$$
\vspace{-1.2cm}

\hqed
\end{cor}

Because $\theta$ commutes with $r_1$ we can conclude that $\theta$ restricts to involutions implying that $\theta$ is also a bijection $\cI_n(132)\rightarrow\cI_n(213)$. This reproves Theorem~\ref{thm:132symm213}, a result we had shown in  Section~\ref{maj213}. 

\begin{cor}
For $n\geq 0$ we have the symmetry
$$M\cI_n(132)=q^{\binom{n}{2}}M\cI_n(213;q^{-1}).$$
\vspace{-1.2cm}

\hqed
\end{cor}

This symmetry does not seem to be limited to the pairs of patterns 123 and 321 as well as 132 and 213. It appears we have the symmetries $M_n(\pi_1)=q^{\binom{n}{2}}M_n(\pi_2;q^{-1})$ and $M\cI_n(\pi_1)=q^{\binom{n}{2}}M\cI_n(\pi_2;q^{-1})$ for any pair of patterns of the form $\pi_1=12[\ii_k,\dd_{m-k}]$ and $\pi_2=12[\dd_{k+1},\ii_{m-k-1}]$ for any $n$, $m$ and $k\in\{0,\dots m\}$. 
\begin{conj}
For the pair of patterns $\pi_1=12[\ii_k,\dd_{m-k}]$ and $\pi_2=12[\dd_{k+1},\ii_{m-k-1}]$ we have  for $n\geq 0$ the symmetry
$$M_n(\pi_1)=q^{\binom{n}{2}}M_n(\pi_2;q^{-1}).$$
\label{conj}
\end{conj}

\begin{conj}
For the pair of patterns $\pi_1=12[\ii_k,\dd_{m-k}]$ and $\pi_2=12[\dd_{k+1},\ii_{m-k-1}]$ we have  for $n\geq 0$ the symmetry
$$M\cI_n(\pi_1)=q^{\binom{n}{2}}M\cI_n(\pi_2;q^{-1}).$$
\label{conj:invo}
\end{conj}

This has been confirmed for all $m,n\leq 9$ and $k\in\{0,\dots m\}$ and this symmetry for involutions does not seem to happen for any other pairs of patterns. However, this symmetry for permutations does appear to happen for many more pairs. These additional pairs seem to arise from the pairs $\pi$ and $r_0(\pi)$ and equalities coming from the $M$-Wilf equivalence classes. This is because $r_0:\fS_n(\pi)\rightarrow \fS_n(r_0(\pi))$ is a bijection with $\Asc(\sigma)=\Des(r_0(\sigma))$. However, the only pairs whose symmetry  further restricts to involutions does seem to just be the mentioned pairs.
This suggests there does exist maps $\fS_n(\pi_1)\rightarrow\fS_n(\pi_2)$ that commute with $r_1$ and has $\Asc(\sigma)=\Des(\theta(\sigma))$.

One implication of Conjecture~\ref{conj} is another conjecture by Dokos et.\ al. Yan, Ge and Zhang~\cite{YGZ15} (Theorem 1.3) proved  Conjecture~\ref{conj:dokos} in the case of $k = 1$. 

\begin{conj}[Dokos et.\ al.~\cite{DDJSS12} Conjecture 2.7] The following pairs are $M$-Wilf equivalent. 
\begin{enumerate}[(i)]
\item $132[\ii_m,1,\dd_k]$ and $231[\ii_m,1,\dd_k]$.
\item $213[\dd_m,1,\ii_k]$ and $312[\dd_m,1,\ii_k]$.
\end{enumerate}
\label{conj:dokos}
\end{conj}
Conjecture~\ref{conj} implies Conjecture~\ref{conj:dokos} because of the following. Certainly the pair $\pi_1=132[\ii_m,1,\dd_k]$ and $\pi_2=312[\dd_m,1,\ii_k]$ has the symmetry $M_n(\pi_1)=q^{\binom{n}{2}}M_n(\pi_2;q^{-1})$ because $r_0(132[\ii_m,1,\dd_k])=312[\dd_m,1,\ii_k]$. For the same reason the pair $\pi_3=231[\ii_m,1,\dd_k]$ and $\pi_4=213[\dd_m,1,\ii_k]$ displays the same symmetry. If Conjecture~\ref{conj} was true we then would certainly have the equalities  of $M_n(\pi_1)=M_n(\pi_3)$ and $M_n(\pi_2)=M_n(\pi_4)$ that Conjecture~\ref{conj:dokos} implies.

\section*{Acknowledgements}\label{sec:acknow} The author would like to thank Bruce Sagan, Stephanie van Willigenburg and Vasu Tewari for the mentorship and conversations that motivated this research.


\bibliographystyle{plain}

\bibliography{dahlref}
\end{document}